\newif\ifRSPA
\RSPAfalse

\newif\ifLTEX
\LTEXtrue

\ifRSPA
\documentclass[openacc]{rsproca_new}
\else
\documentclass[a4paper]{amsart}
\fi

\usepackage{mathtools,amssymb}
\usepackage{xparse}

\ifLTEX
\usepackage{graphicx}
\else
\usepackage[dvipdfmx]{graphicx}
\fi
 
\usepackage{hyperref}
\usepackage{subcaption}
\usepackage{enumitem}
\usepackage{amsthm}
\usepackage{xspace}

\ifRSPA
\else
\usepackage{newtxtext,newtxmath}
\fi

\newtheoremstyle{nclaim}
  {6pt}   
  {\topsep}   
  {\itshape}  
  {0pt}       
  {\bfseries} 
  {.}         
  {5pt plus 1pt minus 1pt} 
  {\thmname{#1}~\thmnumber{#2}\thmnote{\hspace{2pt}(#3)}}

\newtheorem{thm}{Theorem}[section]
\newtheorem{lem}[thm]{Lemma}
\newtheorem{lemma}[thm]{Lemma}
\newtheorem{cor}[thm]{Corollary}
\newtheorem{prop}[thm]{Proposition}

\newtheorem{claim}{Claim}

\theoremstyle{definition}
\newtheorem{dfn}[thm]{Definition}
\newtheorem{rem}[thm]{Remark}

\newtheorem{obs}[thm]{Observation}
\newtheorem{ex}[thm]{Example}

\newcommand{\R}{\mathbb{R}}
\newcommand{\inte}{\mathrm{int}}
\newcommand{\uprime}{^\prime}
\newcommand{\upr}{\uprime}
\newcommand{\dupr}{^{\prime\prime}}

\newcommand{\csum}{\mathbin{\#}}
\newcommand{\wt}[1]{\widetilde{#1}}
\newcommand{\wtP}{\widetilde{P}}

\newcommand{\fund}{\pi_1}
\newcommand{\Torus}{T^3}
\newcommand{\Rth}{\R^3}

\newcommand{\patt}{net-like pattern\xspace}%
\newcommand{\patts}{net-like patterns\xspace}%
\newcommand{\ssn}[1]{^{(#1)}}

\newcommand{\tzmove}{$\text{$2$-$0$}$ move\xspace}
\newcommand{\tTmove}{$\text{$2$-$3$}$ move\xspace}

\newcommand{\auAd}[1]{$^{#1}$}

\begin{document}

\ifRSPA
\title{Handlebody decompositions of 3-manifolds and polycontinuous patterns}
\author{N.~Sakata\auAd{1}, R.~Mishina\auAd{1}, M.~Ogawa\auAd{1}, K.~Ishihara\auAd{2}, Y.~Koda\auAd{3}, M.~Ozawa\auAd{4}, and K.~Shimokawa\auAd{1}}

\address{%
  \auAd{1}Department of Mathematics, Saitama University, Saitama, 338-8570, Japan\\
  \auAd{2}Department of Education, Yamaguchi University, Yamaguchi, 753-8511, Japan\\
  \auAd{3}Department of Mathematics, Hiroshima University, Hiroshima, 739-8511, Japan\\
  \auAd{4}Department of Natural Sciences, Faculty of Arts and Sciences, Komazawa University, Tokyo, 154-8525, Japan
}
\subject{topology, materials science, applied mathematics}
\else

\title{Handlebody Decompositions of $3$-manifolds and Polycontinuous Patterns}
\author[N. Sakata]{Naoki Sakata\auAd{1}}
\author[R. Mishina]{Ryosuke Mishina\auAd{1}}
\author[M. Ogawa]{Masaki Ogawa\auAd{1}}
\author[K. Ishihara]{Kai Ishihara\auAd{2}}
\author[Y. Koda]{Yuya Koda\auAd{3}}
\author[M. Ozawa]{Makoto Ozawa\auAd{4}}
\author[K. Shimokawa]{Koya Shimokawa\auAd{1}}
\address{\auAd{1}Department of Mathematics, Saitama University, Saitama, 338-8570, Japan}
\email{sakata@casis.sakura.ne.jp; m.ogawa.691@ms.saitama-u.ac.jp; kshimoka@rimath.saitama-u.ac.jp}
\address{\auAd{2}Department of Education, Yamaguchi University, Yamaguchi, 753-8511, Japan}
\email{kisihara@yamaguchi-u.ac.jp}
\address{\auAd{3}Graduate School of Advanced Science and Engineering, Hiroshima University, Hiroshima, 739-8511, Japan}
\email{ykoda@hiroshima-u.ac.jp}
\address{\auAd{4}Faculty of Arts and Sciences, Komazawa University, Tokyo, 154-8525, Japan}
\email{w3c@komazawa-u.ac.jp}
\date{}

\fi


\begin{abstract}
\ifRSPA
We introduce the concept of a handlebody decompo-\\
sition of a 3-manifold,
\else
We introduce the concept of a handlebody decomposition of a 3-manifold,
\fi
a generalization of a Heegaard splitting, or a trisection.
We show that two handlebody decompositions of a closed orientable 3-manifold are stably equivalent.
As an application to materials science, we consider a mathematical model of polycontinuous patterns and discuss a topological study of microphase separation of a block copolymer melt.
\end{abstract}

\ifRSPA
\else
\maketitle
\fi

\ifRSPA
\begin{fmtext}
\fi
	
\section{Introduction}
A {\em Heegaard splitting} is a decomposition of a closed orientable 3-manifold into two handlebodies of the same genus.
It is well-known that every closed orientable 3-manifold admits a Heegaard splitting.
By Reidemeister-Singer's theorem~\cite{Reidemeister,Singer},
two Heegaard splittings of a given 3-manifold are stably equivalent, i.e., isotopic after a finite number of stabilizations.

\ifRSPA
Many generalizations of Heegaard splittings have been investigated.
G{\'o}mez-Larra{\~n}aga~\cite{Gomez87} studied orientable $3$-manifolds decomposed into three solid tori.
Coffey and Rubinstein analyzed orientable $3$-manifolds built from three $\pi_1$-injective handlebodies~\cite{Coffey}.
In~\cite{Koenig2018}, Koenig considered a \emph{trisection} of a closed orientable $3$-manifold,
\end{fmtext}

\noindent
which is an embedded branched surface decomposing the manifold into three handlebodies with connected pairwise intersections.
Koenig introduced the notion of stabilization for a trisection and showed an analogue of Reidemeister-Singer's theorem for trisections of 3-manifolds.
\maketitle

\else

Many generalizations of Heegaard splittings have been investigated.
G{\'o}mez-Larra{\~n}aga~\cite{Gomez87} studied orientable $3$-manifolds decomposed into three solid tori.
Coffey and Rubinstein analyzed orientable $3$-manifolds built from three $\pi_1$-injective handlebodies~\cite{Coffey}.
In~\cite{Koenig2018}, Koenig considered a \emph{trisection} of a closed orientable $3$-manifold,
which is an embedded branched surface decomposing the manifold into three handlebodies with connected pairwise intersections.
Koenig introduced the notion of stabilization for a trisection and showed an analogue of Reidemeister-Singer's theorem for trisections of 3-manifolds.

\fi


In this paper, we consider a generalization of all of the above.
We define a \emph{handlebody decomposition} to be a decomposition of a closed orientable $3$-manifold into a finite number of handlebodies (see Definition~\ref{dfn:handlebody_decomposition} for the detailed definition).
We will also introduce \emph{stabilizations} for handlebody decompositions and show an analogue of Reidemeister-Singer's theorem for handlebody decompositions (see Theorem~\ref{thm:stabeq}).

The primary motivation of this study comes from materials science.
We are interested in the characterization of \emph{bicontinuous patterns}, \emph{tricontinuous patterns}, and \emph{polycontinuous patterns} of microphase separation of a block copolymer melt
\ifRSPA
(see Section~\ref{sec:patts}\ref{subsec:polymer}).
\else
(see Section~\ref{subsec:polymer}).
\fi
See \cite{Hyde2,Hyde1} for related research.
A mathematical model of a bicontinuous (resp. tricontinuous or polycontinuous) pattern is a triply periodic non-compact surface (resp. tribranched surface or polyhedron) embedded in $\mathbb R^3$ that divides it into two (resp. three or a finite number of) possibly disconnected submanifolds as shown in Figure~\ref{fig:intro} (see Definition~\ref{dfn:polyconti} for more details). 
We are particularly interested in the case where the submanifolds are the open neighborhood of networks.

If a bicontinuous pattern is triply periodic, then by considering the quotient of the action, the pattern induces a Heegaard splitting of the 3-dimensional torus $T^3$ (see Remark~\ref{rem:biconti}).
If a polycontinuous pattern is triply periodic and satisfies suitable conditions, then it corresponds to a handlebody decomposition of $T^3$ (Corollary~\ref{cor:effective_pattern2decomp}).
Hence a characterization of handlebody decompositions of $T^3$ gives that of triply periodic polycontinuous patterns.
The Reidemeister-Singer-type theorem of polycontinuous patterns (Corollary~\ref{cor:pattern_stabilization_theorem}) follows from that of handlebody decompositions of $T^3$.
This point of view allows us to explain how two polycontinuous patterns are related,
\ifRSPA
which will be discussed in Section~\ref{sec:patts}\ref{subsec:polymer}.
\else
which will be discussed in Section~\ref{subsec:polymer}.
\fi

This paper is organized as follows.
In Section 2, we define a handlebody decomposition of a 3-manifold.
In Section 3, we introduce several types of stabilization operations of handlebody decompositions and prove an analogue of Reidemeister-Singer's theorem for them.
In Section 4, we particularly focus on decompositions of 3-manifolds into three handlebodies.
In Section 5, we study a mathematical model of polycontinuous patterns. We define polycontinuous patterns and, more generally, net-like patterns.
The correspondence between triply periodic net-like patterns and handlebody decompositions of $T^3$ is given.
In Section 6, we discuss stabilizations of net-like patterns. 
We also present how this research relates to the subject of materials science.
In Section 7, we give characterizations of net-like patterns.

\begin{figure}[htbp]
    \centering
    \includegraphics[width=0.30\textwidth]{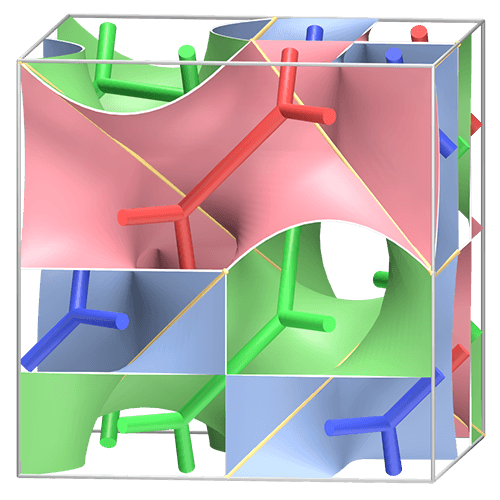}
    \includegraphics[width=0.30\textwidth]{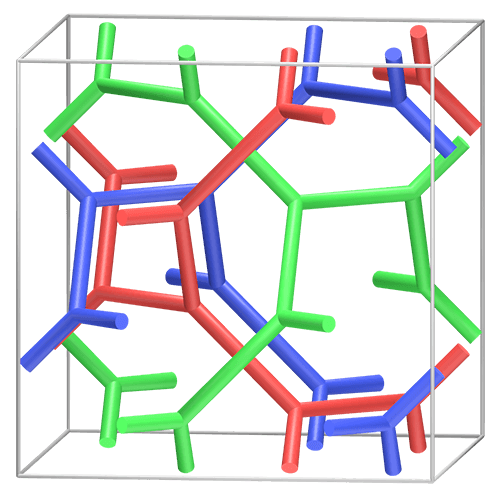}
    \caption{A tricontinuous pattern and an entangled network.}
    \label{fig:intro}
\end{figure}

\section{Handlebody decompositions of 3-manifolds}

We work in the piecewise linear category throughout this paper.

\subsection{Definition of handlebody decompositions of 3-manifolds with simple polyhedra}

By a \emph{2-dimensional polyhedron} $P$, we mean the underlying space of a non-collapsible locally finite 2-dimensional complex such that the link of each vertex contains no isolated vertices.
A connected component of the set of points of $P$ having neighborhoods homeomorphic to disks is called a \emph{sector}.
The set of all points not contained in the sectors is called its \emph{singular graph}.
A 2-dimensional polyhedron $P$ is said to be \emph{simple} if, after giving a structure of a complex in a suitable way, the link of each point in $P$ is homeomorphic to one of the three models shown in Figure~\ref{fig:simple_poly}.
A point whose link is homeomorphic to the model (c) is called a \emph{vertex} of its singular graph.
See Matveev~\cite{Matveev} for more details.

\begin{figure}[htbp]
    \centering
    \ifRSPA
    \begin{minipage}[b]{0.24\linewidth}
        \centering
        \includegraphics[width=.50\textwidth]{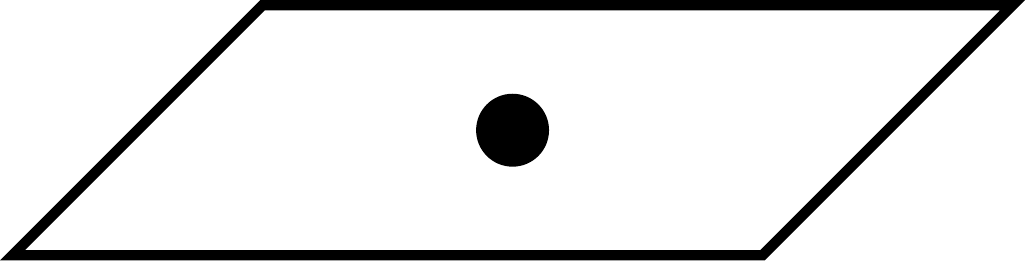}\\
        \vspace{3.0mm}

        \noindent
        (a) a nonsingular point
    \end{minipage}
    \else
    \begin{minipage}[b]{0.25\linewidth}
        \centering
        \includegraphics[width=.50\textwidth]{figure/nonsingular}\\
        \vspace{3.0mm}

        \noindent
        (a) a nonsingular point
    \end{minipage}
    \fi
    \begin{minipage}[b]{0.24\linewidth}
        \centering
        \includegraphics[width=.50\textwidth]{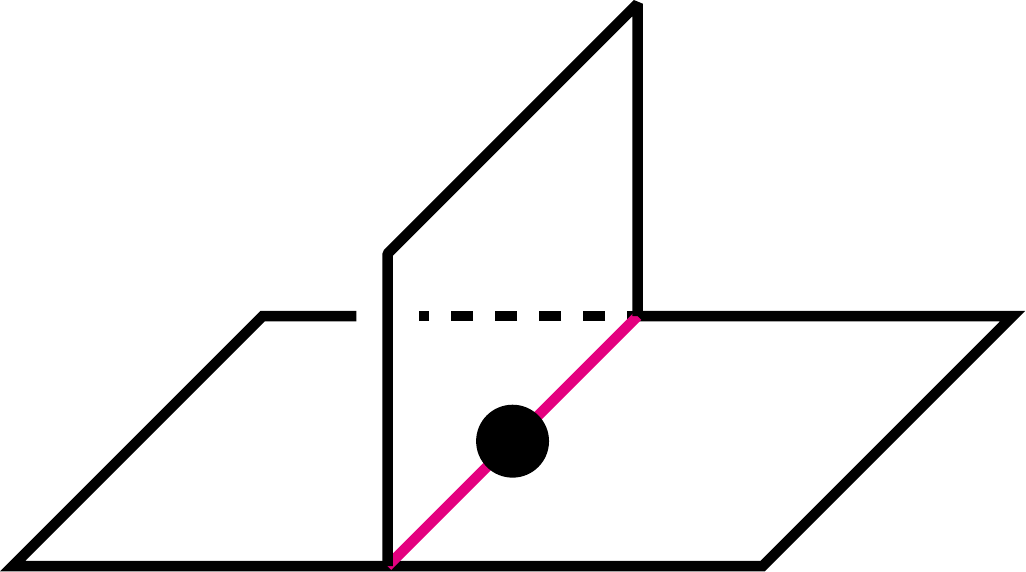}\\
        \vspace{3.0mm}

        \noindent
        (b) a triple point
    \end{minipage}
    \begin{minipage}[b]{0.24\linewidth}
        \centering
        \includegraphics[width=.50\textwidth]{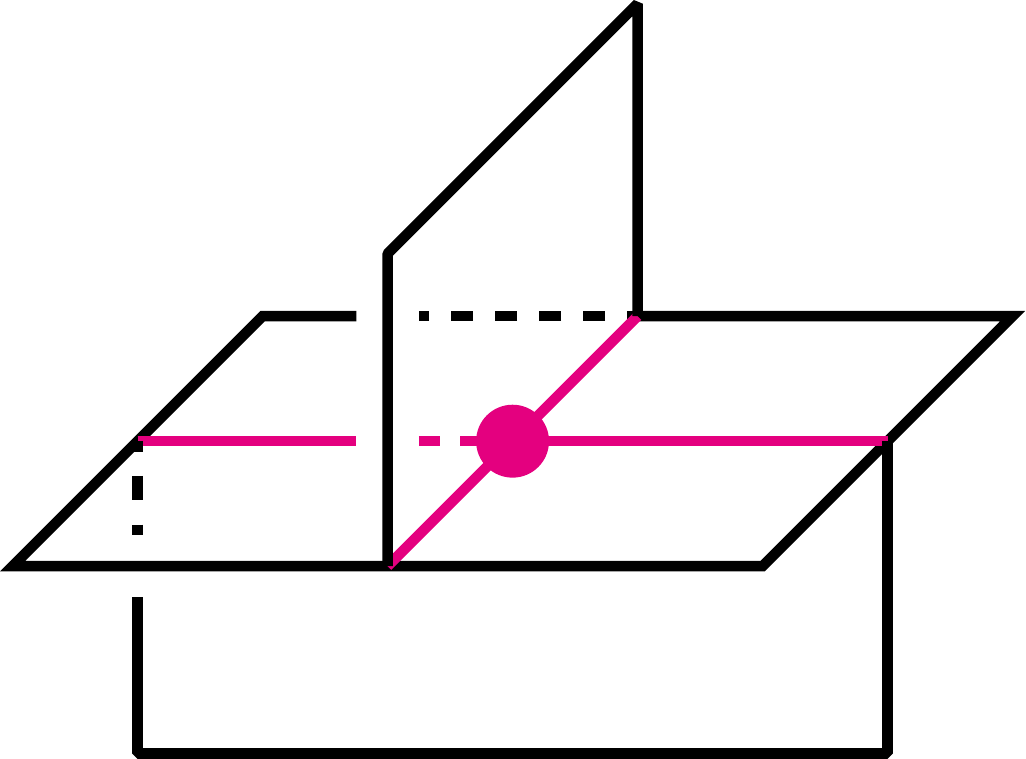}\\
        (c) a vertex
    \end{minipage}
    \caption{A neighborhood of each point of a simple polyhedron.}
    \label{fig:simple_poly}
\end{figure}

\begin{dfn}[Handlebody decomposition]\label{dfn:handlebody_decomposition}
    Let $M$ be a closed, connected, orientable $3$-manifold,
    and $P$ a connected compact $2$-dimensional polyhedron embedded in $M$.
    We call $(H_1, H_2, \ldots, H_n; P)$ a \emph{type-$(g_1, g_2, \ldots, g_n)$ handlebody decomposition} of $M$ if $M \setminus P = \bigsqcup_{i = 1}^{n} H_i$, where $H_i$ is the interior of a handlebody of genus $g_i$.
    The polyhedron $P$ is called a \emph{partition} for the decomposition.
    A handlebody decomposition is said to be \emph{proper} if there is no simple closed curve in $M \setminus B$ that intersects a sector of $P$ transversely once, where $B$ is the singular graph of $P$.
    A handlebody decomposition is said to be \emph{simple} if its partition is a simple polyhedron.
\end{dfn}

\begin{rem}\label{rem:cl_Hdecomp}
    (1) Let $(H_1, H_2, \ldots, H_n; P)$ be a type-$(g_1, g_2, \ldots, g_n)$ handlebody decomposition of $M$, and $W_i$ a handlebody of genus $g_i$ for $i = 1, 2, \ldots, n$.
    Then there exists a continuous map $\iota_i: W_i \to M$ such that the restriction of $\iota_i$ to the interior of $W_i$ is a homeomorphism to $H_i$.
    Then we have $\iota_i(W_i) \cap P = \iota_i(\partial W_i) \cap P$.
    Suppose that the handlebody decomposition is proper.
    Then for the closure $F$ of each sector, there exists a pair of handlebodies $(W_i, W_j)$ $(i \neq j)$ such that $F \subset \iota_i(\partial W_i) \cap \iota_j(\partial W_j)$.
    We denote the union of all such surfaces $F$ by $F_{ij}$.
    (Note that $F_{ij} = F_{ji}$.)

    (2) 
    In general, $\iota_i$ may not be injective on the singular graph of $P$.
    If the decomposition is simple and proper, then $\iota_i$ is a homeomorphism.
\end{rem}

The notion of handlebody decompositions generalises both Heegaard splittings~\cite{Johannson} and trisections~\cite{Koenig2018} of closed orientable 3-manifolds.
In fact, a simple proper handlebody decomposition with $n = 2$ is nothing but a Heegaard splitting, while that with $n = 3$, where each $F_{ij}$ is connected, is a trisection.
By~\cite{Casler}, any closed, connected, $3$-manifold $M$ admits a simple (non-proper) type-$(0)$ handlebody decomposition.
Therefore, it is easily seen that for any sequence $(g_1, \ldots, g_n)$ of non-negative integers, there exists a simple (possibly non-proper) type-$(g_1, \ldots, g_n)$ handlebody decomposition of $M$.

%


\section{Stable equivalence}
\label{sec:stab}

This section discusses the stable equivalence of simple proper handlebody decompositions of a $3$-manifold.
We assume that a handlebody decomposition is simple and proper throughout this section.
By Remark~\ref{rem:cl_Hdecomp}, for a handlebody decomposition $(H_1, \ldots, H_n; P)$, there exist handlebodies $W_1, \ldots, W_n$ and continuous maps $\iota_1, \ldots, \iota_n$ such that the restriction of each $\iota_i$ to the interior of $W_i$ is an embedding $\inte(W_i) \to H_i$.
For simplicity, we regard $H_i$ as $\iota_i(W_i)$ and $\partial H_i$ as $\iota_i(\partial W_i)$.
Then the intersection of $H_i$ and $H_j$ is a possibly disconnected surface with boundary.
We denote it by $F_{ij}$.

\subsection{Stabilizations and destabilizations of handlebody decompositions}

The following operations for handlebody decompositions are a generalization of the ``stabilization'' for Heegaard splittings.

\begin{dfn}\label{dfn:stabilization}
Let $(H_1,\ldots,H_n; P)$ be a simple proper type-$(g_1, \ldots, g_n)$ handlebody decomposition of a closed, connected, orientable $3$-manifold $M$.
\begin{enumerate}[label=(\arabic*),align=parleft,leftmargin=*,itemindent=14.21pt,itemsep=1mm,start=0]
\item 
    Take a properly embedded arc $\alpha$ in $H_i$, and an arc $\beta$ in $\partial H_i$ such that the endpoints of $\alpha$ lie in the interior of $F_{ij}$, and $\alpha$ is parallel to $\beta$ in $H_i$ relative to the endpoints, i.e., the endpoints of $\alpha$ is equal to that of $\beta$, and $\alpha \cup \beta$ bounds a disk in $H_i$.
Then we get a type-$(g\uprime_1, \ldots, g\uprime_n)$ handlebody decomposition $(H\uprime_1,\ldots,H\uprime_n)$ of $M$ with
\begin{align*}
    g\uprime_l & =\begin{cases}
        g_i+1 & (l=i)\\
        g_j+1 & (l=j)\\
        g_l & (l\neq i,j)
    \end{cases} & 
        H\uprime_l &= \begin{cases}
        H_i\setminus \inte(N(\alpha)) & (l=i)\\
        H_j\cup N(\alpha) & (l=j)\\
        H_l & (l\neq i,j)
    \end{cases},
\end{align*}
where $N(\alpha)$ and $\inte(N(\alpha))$ are a regular neighborhood of $\alpha$ and its interior in $H_i$, respectively.
We call this operation a \emph{type-$0$ stabilization} (along $\alpha$).
Conversely, we assume that there exist properly embedded disks $D_j$ of $H_j$ and $E$ in $H_i$ such that the boundary of $D_j$ is in $F_{ij}$, and the boundary of $D_j$ intersects that of $E$ transversely exactly one point.
Then we can perform the inverse operation of a type-$0$ stabilization. 
We call this operation a \emph{type-$0$ destabilization} (along $D_j$).
See Figure~\ref{fig:stabilizations}(a).

\item 
Take a properly embedded arc $\alpha$ on $F_{jk}$ such that the endpoints of $\alpha$ lie in the boundary of $H_{i}$ for $i\neq j, k$.
Then we get a type-$(g\uprime_1, \ldots, g\uprime_n)$ handlebody decomposition $(H\uprime_1,\ldots,H\uprime_n)$ of $M$ with
\begin{align*}
    g\uprime_l & =\begin{cases}
        g_i+1 & (l=i)\\
        g_l & (l\neq i)
    \end{cases} & 
        H\uprime_l & =\begin{cases}
        H_i\cup N(\alpha) & (l=i)\\
        H_l\setminus \inte(N(\alpha)) & (l=j, k)\\
        H_l & (l\neq i,j,k)
    \end{cases}.
\end{align*}
We call this operation a \emph{type-$1$ stabilization} (along $\alpha$).
Conversely, if there exists a non-separating disk $D_{i}$ of $H_{i}$ such that the boundary of $D_i$ intersects the singular graph of the partition $P$ transversely exactly two points, then we can perform the inverse operation of a type-$1$ stabilization. 
We call this operation a \emph{type-$1$ destabilization} (along $D_i$).
See Figure~\ref{fig:stabilizations}(b).

\item 
Take two points on the interior of $F_{ij}$ and that of $F_{ik}$ for $j \neq k$,
and we connect the points by a properly embedded arc $\alpha$ in $H_i$.
Let $\beta$ be an arc in $\partial H_i$ such that $\alpha$ is parallel to $\beta$.
Then we get a type-$(g\uprime_1, \ldots, g\uprime_n)$ handlebody decomposition $(H\uprime_1,\ldots,H\uprime_n)$ of $M$ with
\begin{align*}
    g\uprime_l &= \begin{cases}
        g_i+1&(l=i)\\
        g_l&(l\neq i)\\
    \end{cases} & 
        H\uprime_l &= \begin{cases}
        H_i\setminus \inte(N(\alpha)) & (l=i)\\
        H_j\cup N(\alpha) & (l=j)\\
        H_l & (l\neq i,j)\\
    \end{cases}.
\end{align*}
We call this operation a \emph{type-$2$ stabilization} (along $\alpha$).
Conversely, if there exists a disk component $D_{jk}$ of $F_{jk}$ whose boundary intersects a properly embedded non-separating disk in $H_{i}$ transversely once, then we can perform the inverse operation of a type-$2$ stabilization.
We call this operation a \emph{type-$2$ destabilization} (along $D_{jk}$).
See Figure~\ref{fig:stabilizations}(c).
\end{enumerate}
\end{dfn}


\begin{figure}[htbp]
    \centering
    \begin{minipage}[b]{0.32\linewidth}
        \centering
        \includegraphics[width=.70\textwidth]{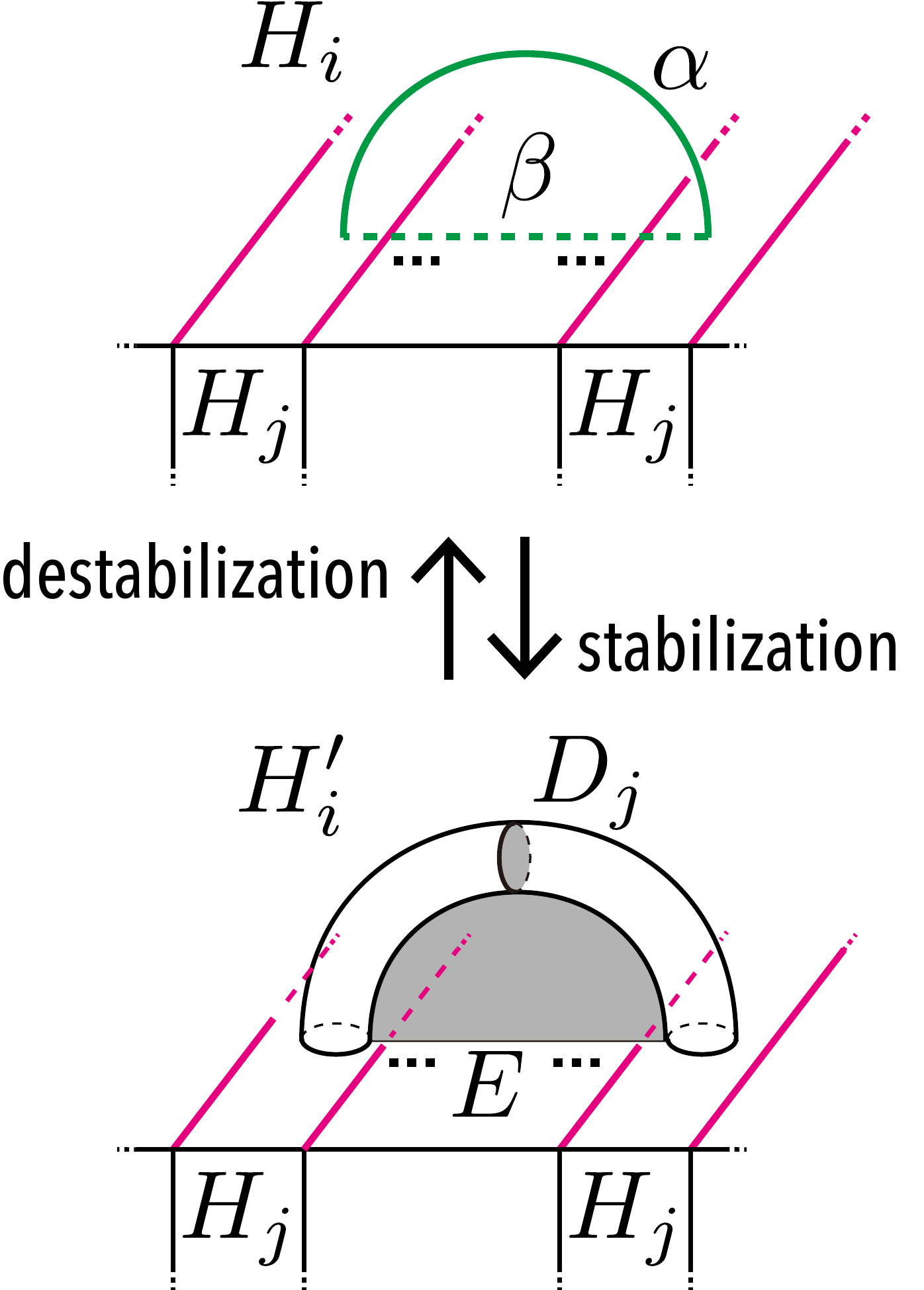}\\
        (a) Type-$0$
    \end{minipage}
    \begin{minipage}[b]{0.32\linewidth}
        \centering
        \includegraphics[width=.70\textwidth]{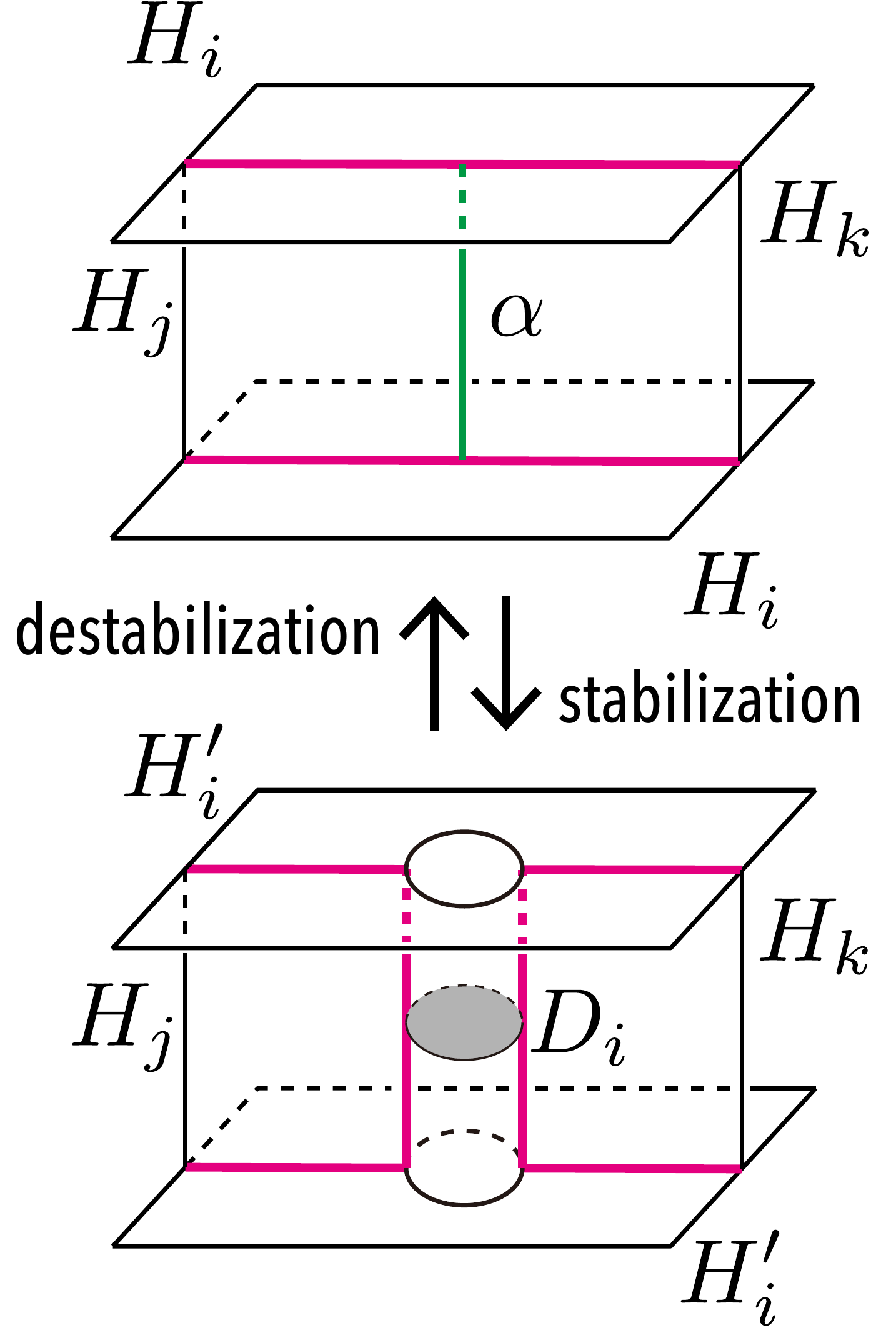}\\
        (b) Type-$1$
    \end{minipage}
    \begin{minipage}[b]{0.32\linewidth}
        \centering
        \includegraphics[width=.70\textwidth]{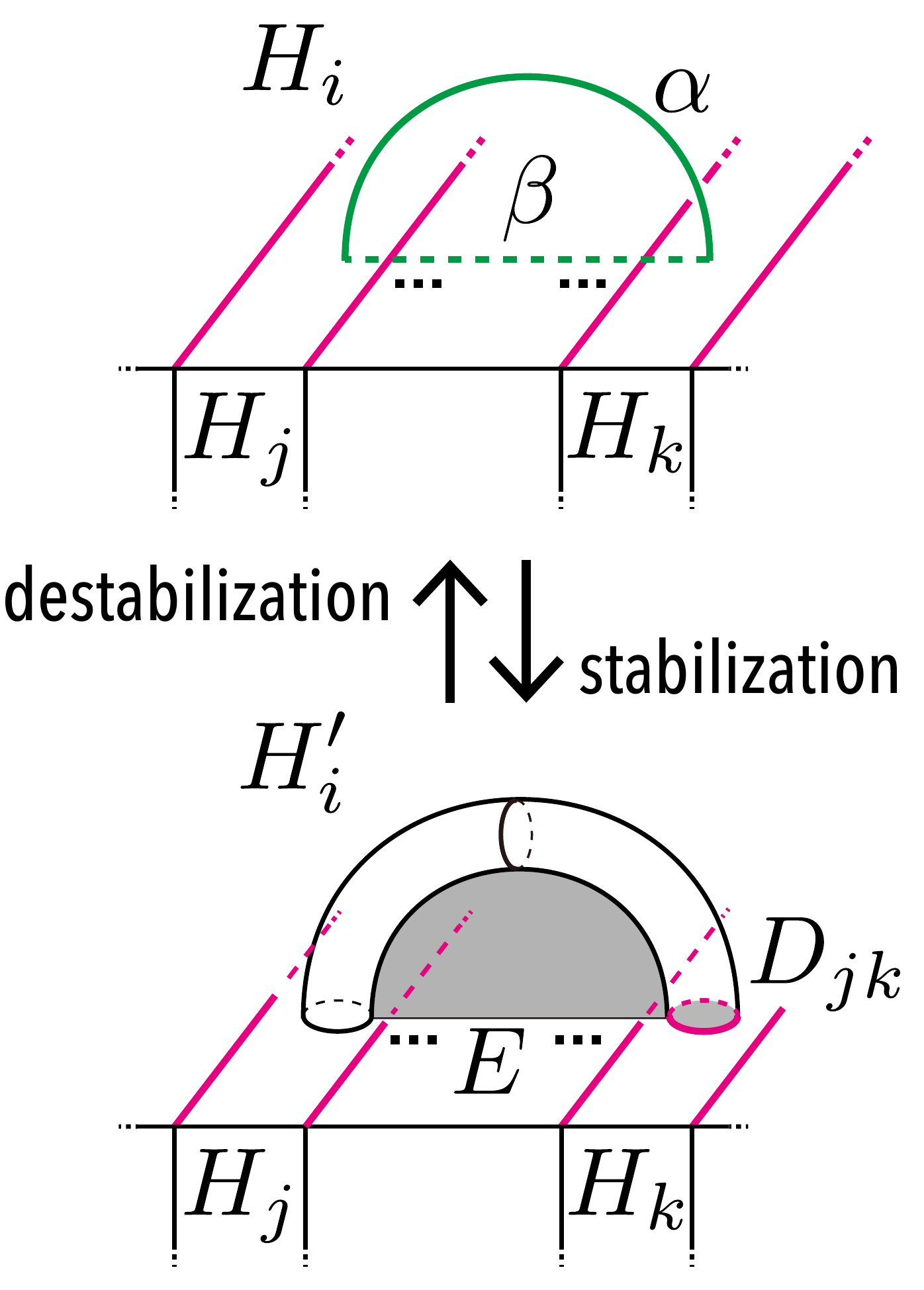}\\
        (c) Type-$2$
    \end{minipage}
\caption{Stabilizations and destabilizations. Red curves represent the singular graphs.}
\label{fig:stabilizations}
\end{figure}

\begin{rem}\label{rem:higher_type_by_type1}
    Consider a type-$(g_1, g_2, \ldots, g_n)$ handlebody decomposition of a closed, connected, orientable $3$-manifold $M$ with $3 \leq n$.
    For every $g_i \leq g\uprime_i$, we can obtain a type-$(g\uprime_1, g\uprime_2, \ldots, g\uprime_n)$ handlebody decomposition of $M$ by performing type-$1$ stabilizations repeatedly in a suitable way.
\end{rem}

\begin{dfn}
    A handlebody decomposition is said to be \emph{stabilized} if it is obtained from another handlebody decomposition by a stabilization.
\end{dfn}

When $n = 2$, a type-$0$ stabilization is nothing but a stabilization of Heegaard splittings.
\ifRSPA
In Supplementary Information, we discuss the independence of these stabilizations.
\else
In Appendix~\ref{app:independence}, we discuss the independence of these stabilizations.
\fi

\subsection{Stable equivalence theorem}

This subsection will generalize Koenig's argument~\cite{Koenig2018} on the stable equivalence of decompositions.
We first recall the following operations on simple polyhedra embedded in a closed orientable $3$-manifold introduced by Matveev~\cite{Matveev} and Piergallini~\cite{Piergallini} under our setting.

\begin{dfn}\label{dfn:moves}
Let $P$ be the partition of a handlebody decomposition of $M$.
\begin{enumerate}[label=(\arabic*),align=parleft,leftmargin=*,start=1]
\item 
Let $\alpha$ be a properly embedded arc in $F_{jk}$.
A modification of $P$ in a neighborhood of $\alpha$, as in Figure~\ref{fig:moves}(a), is called a \emph{$0$-$2$ move along $\alpha$}.
By this operation, the number of vertices of $P$ increases by two, and a new disk component appears in $F_{il}$.
Conversely, we can perform the inverse operation of 0-2 move along a disk component, $D$, of $F_{il}$.
We call the operation a \emph{$2$-$0$ move along $D$}.
By this operation, the number of vertices of $P$ decreases by two, and the disk component is removed from $F_{il}$.
\item 
Let $\alpha$ be an edge of the singular graph of $P$.
A modification of $P$ in a neighborhood of $\alpha$, as in Figure~\ref{fig:moves}(b), is called a \emph{$2$-$3$ move along $\alpha$}.
By this operation, the number of vertices of $P$ increases by one, and a new disk component appears in $F_{im}$.
Conversely, we can perform the inverse operation of 2-3 move along a disk component, $D$, of $F_{im}$.
We call the operation a \emph{$3$-$2$ move along $D$}.
By this operation, the number of vertices of $P$ decreases by one, and the disk component is removed from $F_{im}$.
\end{enumerate}
\end{dfn}
We note that the above moves do not change the topological type of each handlebody of a decomposition.


\begin{figure}[htbp]
    \centering
    \begin{minipage}[b]{0.32\linewidth}
        \centering
        \includegraphics[width=.70\textwidth]{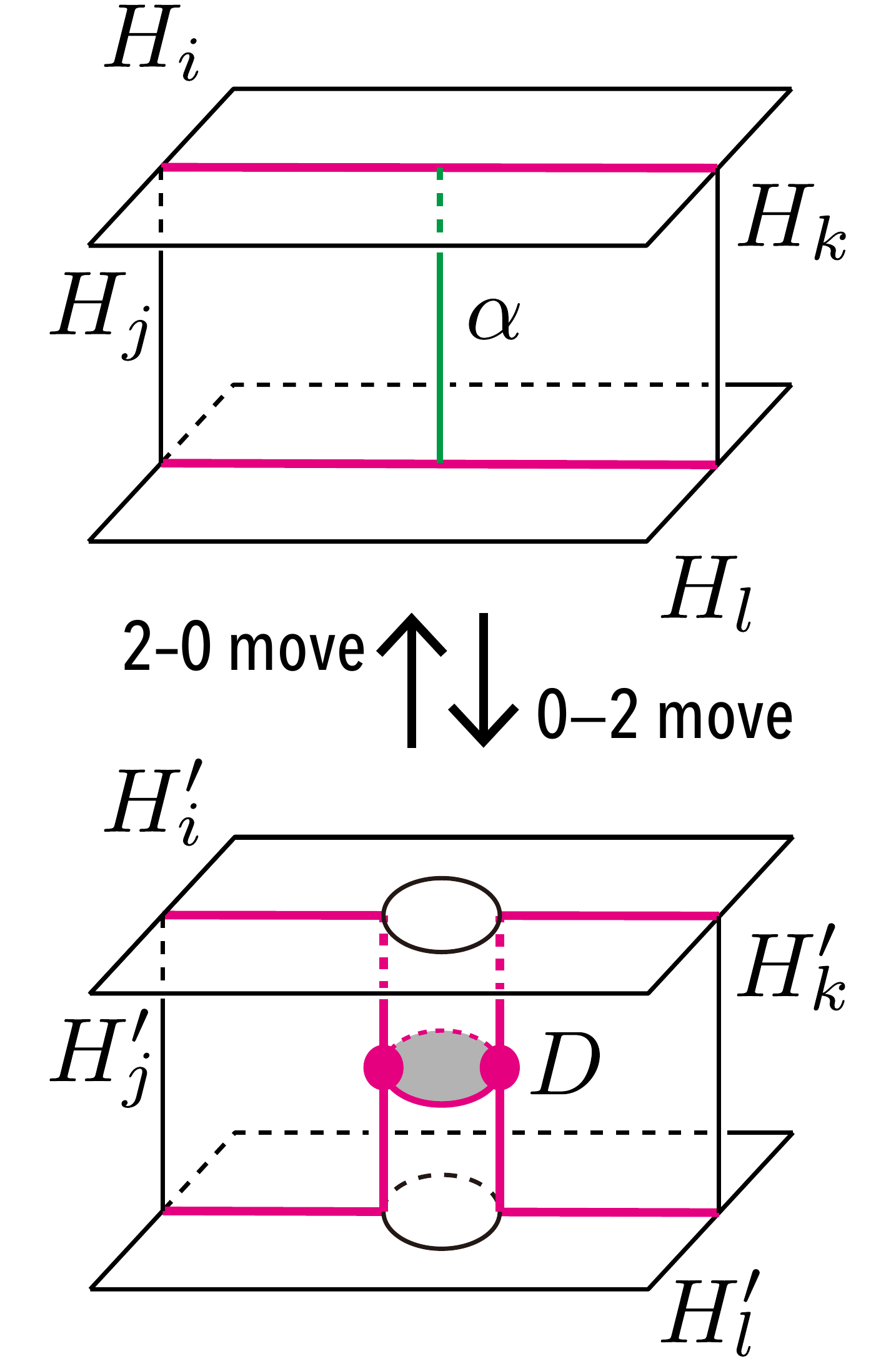}\\
        (a) $0$-$2$ move and $2$-$0$ move
    \end{minipage}
    \begin{minipage}[b]{0.32\linewidth}
        \centering
        \includegraphics[width=.70\textwidth]{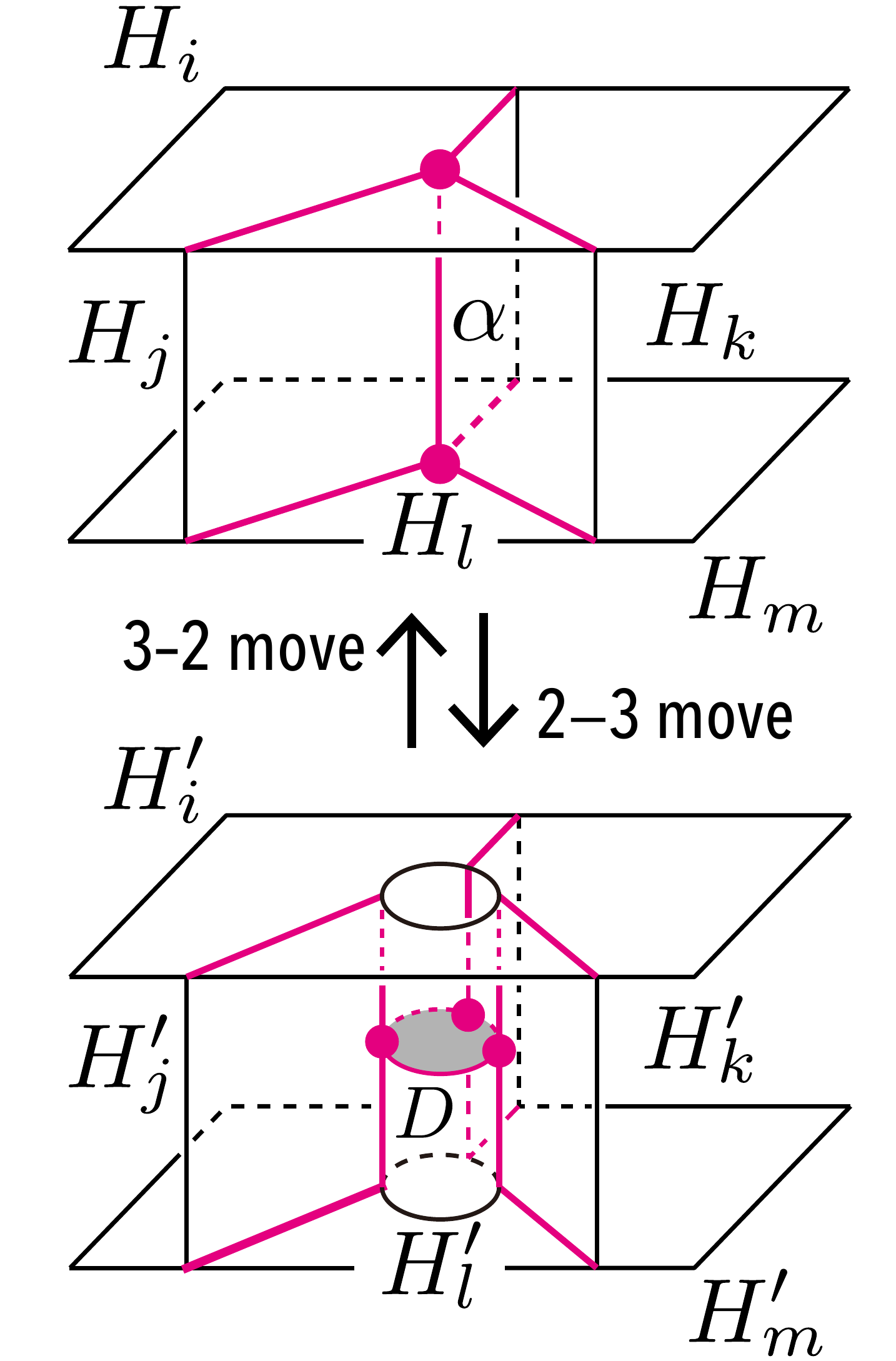}\\
        (b) $2$-$3$ move and $3$-$2$ move
    \end{minipage}
    \caption{Moves on a handlebody decomposition.}
    \label{fig:moves}
\end{figure}

We say that two handlebody decompositions $(H_1, H_2, \ldots, H_n; P)$ and $(H_1\uprime, H_2\uprime, \ldots, H_n\uprime; P\uprime)$ of a closed orientable $3$-manifold $M$ are \emph{equivalent} if there exists an ambient isotopy of $M$ that moves $P$ to $P\upr$ and each $H_i$ to $H_i\upr$ $(i = 1, \ldots, n)$ simultaneously.

\begin{thm}\label{thm:stabeq}
    Let $\mathcal{H}=(H_1, H_2, \ldots, H_n; P)$ and $\mathcal{H}\uprime= (H_1\uprime, H_2\uprime, \ldots, H_n\uprime; P\uprime)$ be simple proper handlebody decompositions of a closed orientable $3$-manifold $M$.
    Then $\mathcal{H}$ and $\mathcal{H}\uprime$ are equivalent after applying $0$-$2$, $2$-$0$, $2$-$3$ moves and types-$0$, and -$1$ stabilizations finitely many times.
\end{thm}

\begin{proof}
    Set $F_{ij}=H_i\cap H_j$ and $F\upr_{ij}=H\upr_i\cap H\upr_j$ as in Remark~\ref{rem:cl_Hdecomp}.
%
We will prove the theorem in the following steps.
\medskip

\begin{enumerate}[label=Step \arabic*.]\setcounter{enumi}{-1}
	\item  In the case of $n\ge 4$, 
            we perform $0$-$2$, $2$-$0$, $2$-$3$ moves and type-$1$ stabilization appropriately until it holds $F_{ij} = \emptyset$ for any $3 \leq i < j \leq n$.
            Then $P$ becomes a simple polyhedron without vertices. 
	\item For each $j\in\{3,\ldots,n\}$, we deform $F_{2j}$ into a disk by type-$1$ stabilizations.
            Then, $(H_1,(H_2\cup \cdots \cup H_n))$ is a Heegaard splitting. 
            By applying the same process for $\mathcal{H}^\prime$, $(H_1^\prime, (H_2^\prime\cup \cdots \cup H_n^\prime))$ becomes a Heegaard splitting, so by Reidemeister-Singer's theorem, we have $H_1 = H_1^\prime$ after applying type-$0$ stabilizations.
	\item For each $j\in\{3,\ldots,n\}$, we deform $F_{1j}$ into a disk by type-$1$ stabilizations. 
	We denote by $S_1$ the surface $F_{12}$ at this stage, and keep it throughout the steps hereafter.
	\item  We cover $H_1$ along $S_1$ with $H_3$ by type-$1$ stabilizations. 
	Then it holds that $H_3=H_3^\prime$ after handle slides.
	\item[Step $i$.] ($4\leq i\leq n$) We cover $H_{i-1}$ along $S_1$ with $H_i$ by $0$-$2$, $2$-$0$ moves, and type-$1$ stabilizations. 
	Then it holds that $H_i=H_i^\prime$ after handle slides.
\end{enumerate}

If $n = 3$, after performing the operations described in the first half of Step~0, the decompositions $\mathcal{H}$ and $\mathcal{H}\upr$ become trisections.
Then, they are equivalent by using Koenig's theorem.
Hence, in this proof, we assume that $n \geq 4$.

\medskip

\noindent		
\underline{Step 0}\quad 
Put $J=\{(i, j) \mid 3\leq i < j \leq n,\,  F_{ij}\neq\emptyset \}$.
Let $(i,j)$ be the minimum element of $J$ in the lexicographical order. 
First, we change $F_{ij}$ to be connected if it is disconnected as follows.
Take an arc properly embedded in the closure of $\partial H_i\setminus F_{ij}$ that connects different components of $F_{ij}$. 
If the arc is contained in some $F_{ik}$, a type-$1$ stabilization along the arc decreases the number of components by one.
Otherwise, we can perform a type-$1$ stabilization after $0$-$2$ moves along the arc to decrease the number of components. 
Hence, by repeatedly applying this process finitely many times, we may assume that $F_{ij}$ is connected.

Next, take mutually disjoint arcs properly embedded in $F_{ij}$ so that they cut open $F_{ij}$ into a disk. 
We perform either a type-$1$ stabilization or a $0$-$2$ move along each of the arcs according to whether both ends of the arc lie in $\partial H_k$ for $k\neq i,j$ or not.
Then $F_{ij}$ becomes a disk. 
Since $P$ gives the simple proper handlebody decomposition $\mathcal{H}$, 
the boundary $\partial F_{ij}$ has either at least two vertices or no vertex of $P$. 

Suppose $F_{ij}$ is a disk and $\partial F_{ij}$ has more than two vertices.
Let $\beta$ be a sub-arc of $\partial F_{ij}$ cut off by the vertices,
and let $H_k$ ($k = 1, 2$, or $j < k$) be the handlebody with $\beta \subset \partial H_k$ (see Figure~\ref{fig:deleting_vertex}(a)).
We perform either a \tzmove or a \tTmove along $\beta$ according to whether there exists a different handlebody $H_l$ ($l = 1, 2$, or $j < l$) from $H_k$ with $\partial \beta \subset H_k \cap H_l$ or not. 
Each operation reduce the number of vertices in $\partial F_{ij}$ by two or one (see Figure~\ref{fig:deleting_vertex}(b) and~(c)).
By continuing this process, the number of vertices in $\partial F_{ij}$ can be reduced to two.
Then we have $F_{ij}=\emptyset$ after performing a $2$-$0$ move on $F_{ij}$. 

Suppose $F_{ij}$ is a disk and $\partial F_{ij}$ has no vertices. 
There is a handlebody $H_k$ ($k = 1, 2$, or $j < k$) with $ \partial F_{ij}\subset H_k$.
In other words, $F_{ik}$ and $F_{jk}$ share their boundary components in $\partial F_{ij}$.
Since $P$ is connected, at least one of $F_{ik}$ and $F_{jk}$ has another boundary component. 
If $F_{ik}$ shares another boundary component with $F_{il}$ (and $F_{kl}$), we take an arc in $F_{ik}$ which connects $\partial F_{ij}$ and $\partial F_{il}$. 
Then we can remove $F_{ij}$ by a $2$-$0$ move after applying a $0$-$2$ move along the arc.   
A disk component of $F_{jl}$ arises in this operation.
It follows that $l > j$ or $l = 1, 2$, as $(i,l)$ is greater than $(i,j)$ in the lexicographical order by the minimality of $(i,j)$.
If $F_{jk}$ shares another boundary component with $F_{jl}$ (and $F_{kl}$), similarly, we can remove 
$F_{ij}$ by a $2$-$0$ move after applying a $0$-$2$ move. 
In this case, there is a possibility that $F_{il}$ changes to a disk from the empty set by this operation with $3\leq l < j$.
This implies that the minimal element of $J$ varies from $(i,j)$ to $(i,l)$.
In such a case, we take an oriented arc in $P$ from a point in $\partial F_{ij}$ to a point in $\partial H_1$ or $\partial H_2$.
Then we can remove $F_{ij}$, and the minimal element of $J$ increases after successively applying $0$-$2$ and $2$-$0$ moves along the arc from the start to the end.

By repeating the same process, we have $J=\emptyset$. 
Namely, $F_{ij}=\emptyset$ for $3\leq i < j\leq n$. 
Since each vertex of $P$ is contained in four different handlebodies, this condition implies that $P$ has no vertex.     
\begin{figure}[htbp]
    \centering
    \begin{minipage}[b]{0.32\textwidth}
        \centering
        \includegraphics[width=0.80\textwidth]{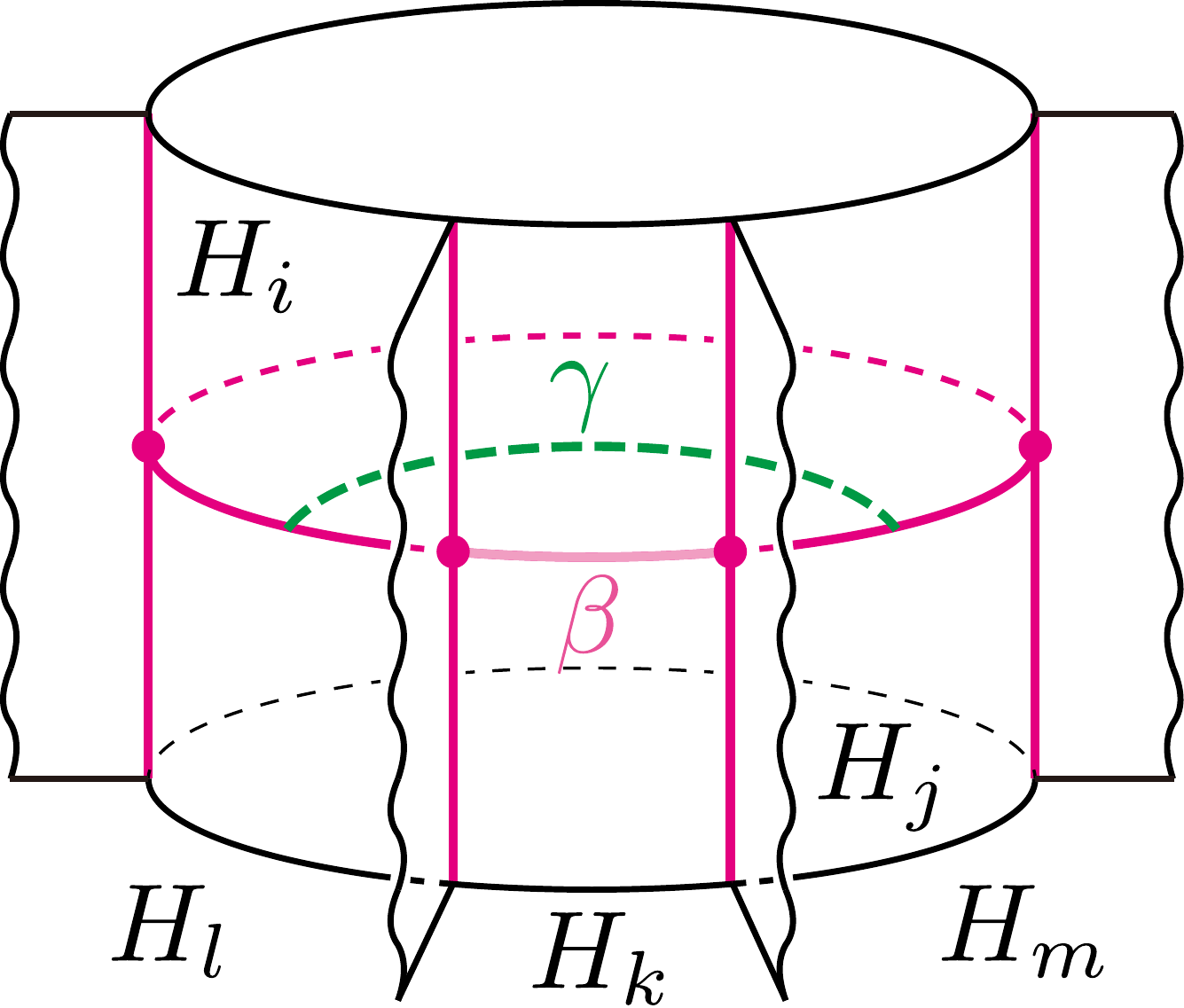}\\
        (a)
    \end{minipage}
    \begin{minipage}[b]{0.32\textwidth}
        \centering
        \includegraphics[width=0.80\textwidth]{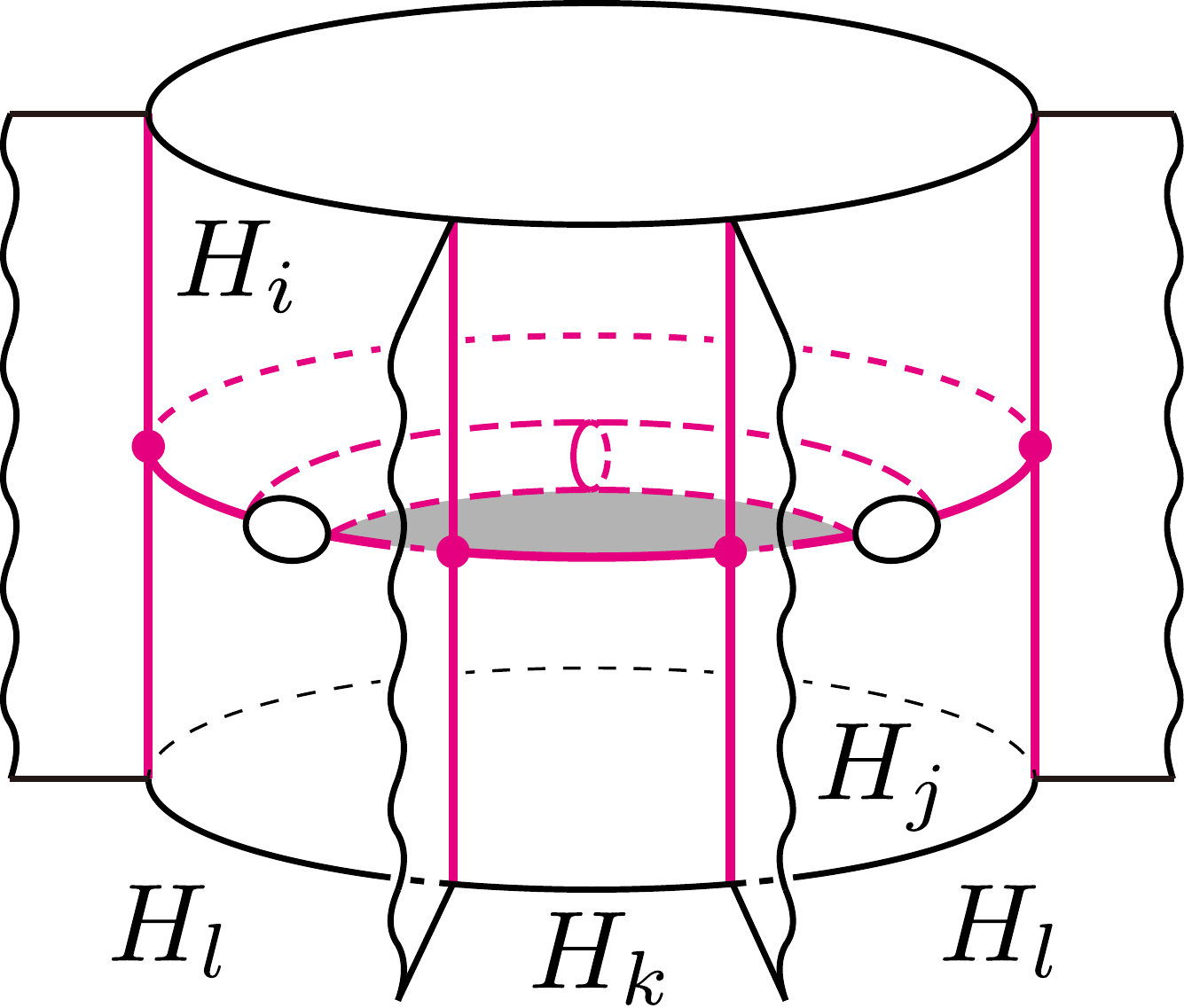}\\
        (b)
    \end{minipage}
    \begin{minipage}[b]{0.32\textwidth}
        \centering
        \includegraphics[width=0.80\textwidth]{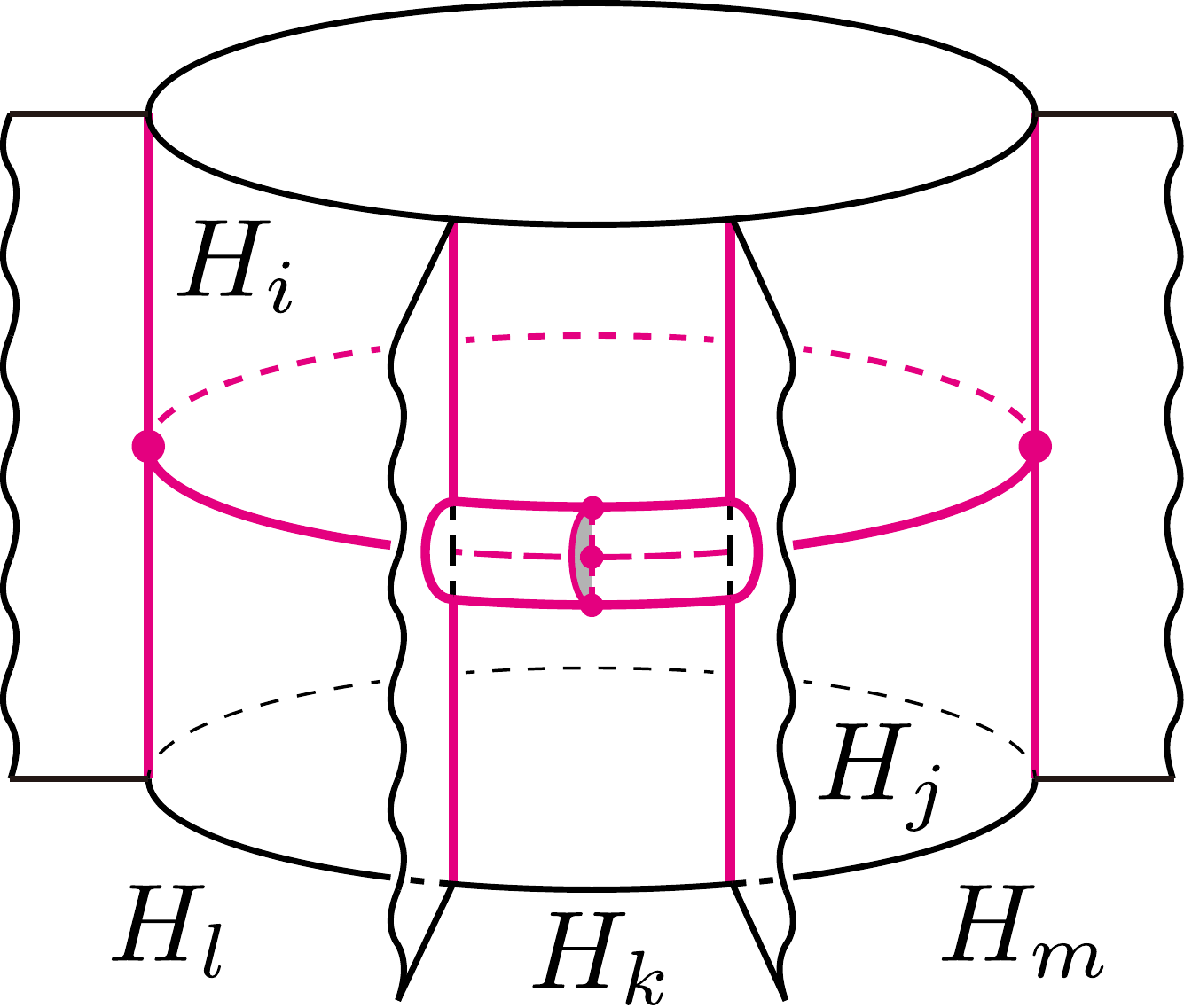}\\
        (c)
    \end{minipage}
    \caption{(a) The disk sector $F_{ij}$. A properly embedded arc $\gamma \subset F_{ij}$ is parallel to $\beta \subset \partial F_{ij}$ that is an arc contained in a handlebody $H_k$. (b) Performing a type-$1$ stabilization along $\gamma$. We can perform a \tzmove along the grayed region. (c) Performing \tTmove along $\beta$.}
    \label{fig:deleting_vertex}
\end{figure}
 
\noindent
\underline{Step 1}\quad
For each $j\geq 3$, we will deform $F_{2j}$ into a disk by applying similar operations in Step~0. 
Since $\partial H_j=F_{1j}\cup F_{2j}$, we may assume that $F_{2j}$ is connected, if necessary, by performing type-$1$ stabilizations along arcs in $F_{1j}$. 
Take a maximal set of non-separating arcs properly embedded in $F_{2j}$.
By performing type-$1$ stabilizations along the arcs, $F_{2j}$ becomes a disk. 
Then $H_2\cup \cdots \cup H_n$ is a handlebody.  
By applying the same process for $\mathcal{H}^\prime$, $H'_2\cup \cdots \cup H'_n$ becomes a handlebody.
Hence $(H_1, (H_2\cup  \cdots \cup H_n))$ and  $(H'_1, (H'_2\cup \cdots \cup H'_n))$ are Heegaard splittings of $M$.
By Reidemeister-Singer's theorem, these two Heegaard splittings become equivalent after performing a finite sequence of type-$0$ stabilizations.
In particular, we can assume $H_1=H_1\upr$.
 	
\medskip
\noindent
\underline{Step 2}\quad
Similarly to Step 1, for each $j\geq 3$, we can deform $F_{1j}$ into a disk by performing type-$1$ stabilizations along suitable arcs properly embedded in $F_{1j}$.
 \begin{claim}\label{claim:CancelingPair}
     For $i \in \{ 3, \ldots, n \}$, let $D_{i1},\ldots,D_{ig_i}$ be a complete meridian disk system of $H_i$ such that $\partial D_{ij}\subset F_{2i}$ for $j\in\{1,\ldots,g_i\}$.  
     Then there exist disjoint meridian disks $E_{ij}$ \textup{(}$i\in\{3,\ldots,n\}$, $j\in\{1,\ldots,g_i\}$\textup{)} of $H_2$ such that $\partial E_{ij}\subset F_{12}\cup F_{2i}
     $, $E_{ij}\cap D=E_{ij}\cap D_{ij}$, and $\partial E_{ij}$ intersects $\partial D_{ij}$ transversely in a single point, where $D$ denotes the union $\cup_{i,j} D_{ij}$ of all $D_{ij}$.
 \end{claim} 
 \begin{proof}[Proof of Claim~\textup{\ref{claim:CancelingPair}}]
 According to the deformation of $H_2$ at this step, there exist mutually disjoint separating disks $E_3,\ldots,E_n$ in $H_2$ such that each $E_i$ cuts off a handlebody $W_i$ from $H_2$ so that $(W_i,F_{2i})$ is homeomorphic to $(F_{2i}\times [0,1],F_{2i}\times \{0\})$. 
(The union $H_i\cup W_i$ can be regarded as the handlebody $H_i$ at the end of Step 1.)
We can take mutually disjoint arcs $\alpha_{i1},\ldots,\alpha_{ig_i}$ properly embedded in $F_{2i}$ so that $\alpha_{ij}\cap D=\alpha_{ij}\cap D_{ij}$, and $\alpha_{ij}$ intersects $D_{ij}$ transversely in a single point.
 See Figure~\ref{fig:step2}.
 Let $E_{ij}$ be a disk corresponding to $\alpha_{ij}\times [0,1]$ such that $E_{ij}\cap E_i=\emptyset$ for each $j\in\{1,\ldots,g_i\}$. 
 Then the assertion holds since $(\partial W_i\setminus E_i)\subset F_{12}\cup F_{2i}$.
 \end{proof}
 \begin{figure}[htbp]
     \centering
     \includegraphics[width=0.78\textwidth]{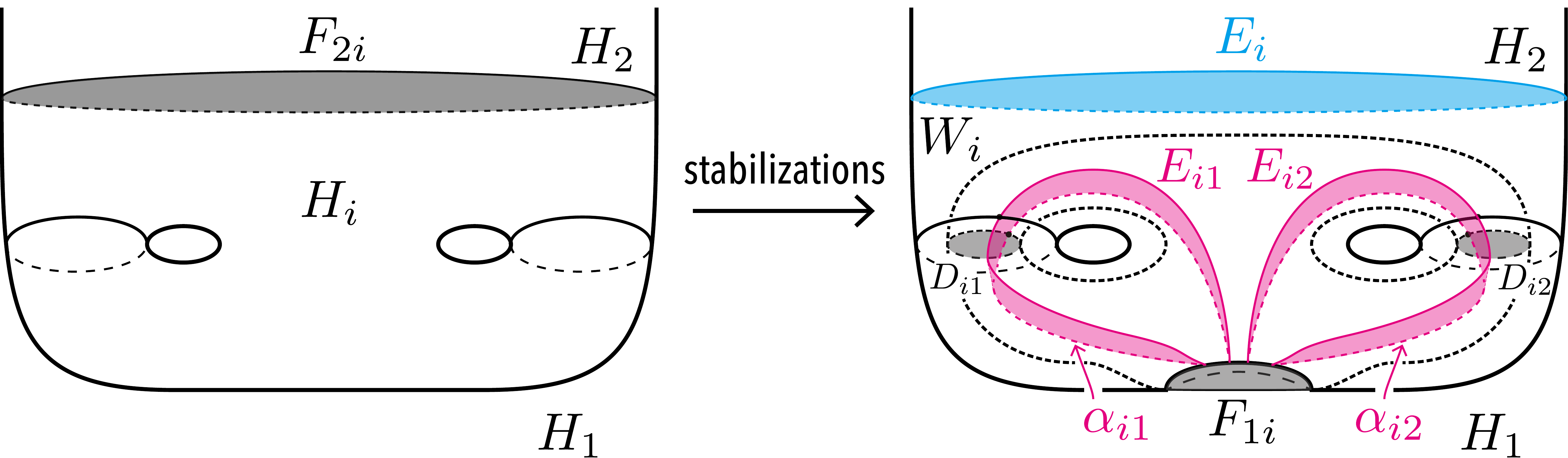}
     \caption{A deformation at Step 2}
     \label{fig:step2}
 \end{figure}

 Let $S_1$ denote the surfaces $F_{12}$ at this stage. 
 Claim~\ref{claim:CancelingPair} implies that any $1$-handle of each handlebody $H_i$ ($i\geq 3$) can be a local $1$-handle after a handle slide on $S_1$.  

\medskip
\noindent 	
\underline{Step 3}\quad 
For handle slide of $H_3$, we will cover $H_1$ along $S_1$ with $H_3$ by type-$1$ stabilizations. 
 Take a maximal set of mutually non-parallel, non-boundary parallel arcs properly embedded in $S_1$ whose endpoints lie in $\partial H_3$.
We perform type-$1$ stabilizations along those arcs. 
The surface $F_{12}=S_1\setminus F_{13}$ becomes the union of $n-3$ annuli $A_{14}, \ldots, A_{1n}$ such that $A_{1j}\cap H_j=\partial F_{1j}=\partial F_{2j}$ for each $j\in\{4,\ldots,n\}$. 
Since all spines of $S_1$ are covered by $H_3$, $H_3$ becomes a local unknotted handlebody after performing handle slides by Claim~\ref{claim:CancelingPair} (see Figure~\ref{fig:handle_slide}). 
Applying the same process for $\mathcal{H}^\prime$ and arranging genera of $H_3$ and $H\upr_3$ by performing type-$0$ stabilizations if necessary, we can assume that $H_3=H\upr_3$. 

According to the deformation of $H_3$ at this step, there exists a separating disk $D_3$ in $H_3$ that cuts off a handlebody $V_3$ from $H_3$ so that $(V_3,F_{13})$ is homeomorphic to $(F_{13}\times [0,1],F_{13}\times \{0\})$. 
($H_3\setminus V_3$ can be regarded as the previous $H_3$ at the end of Step 2.)
Let $S_3$ be the surface $\partial V_3\setminus (D_3\cup F_{13})$, which is a subsurface of $F_{23}$ and homeomorphic to $S_1$.

\begin{figure}[htbp]
    \centering
    \begin{minipage}{0.48\textwidth}
        \centering
        \includegraphics[width=0.90\linewidth]{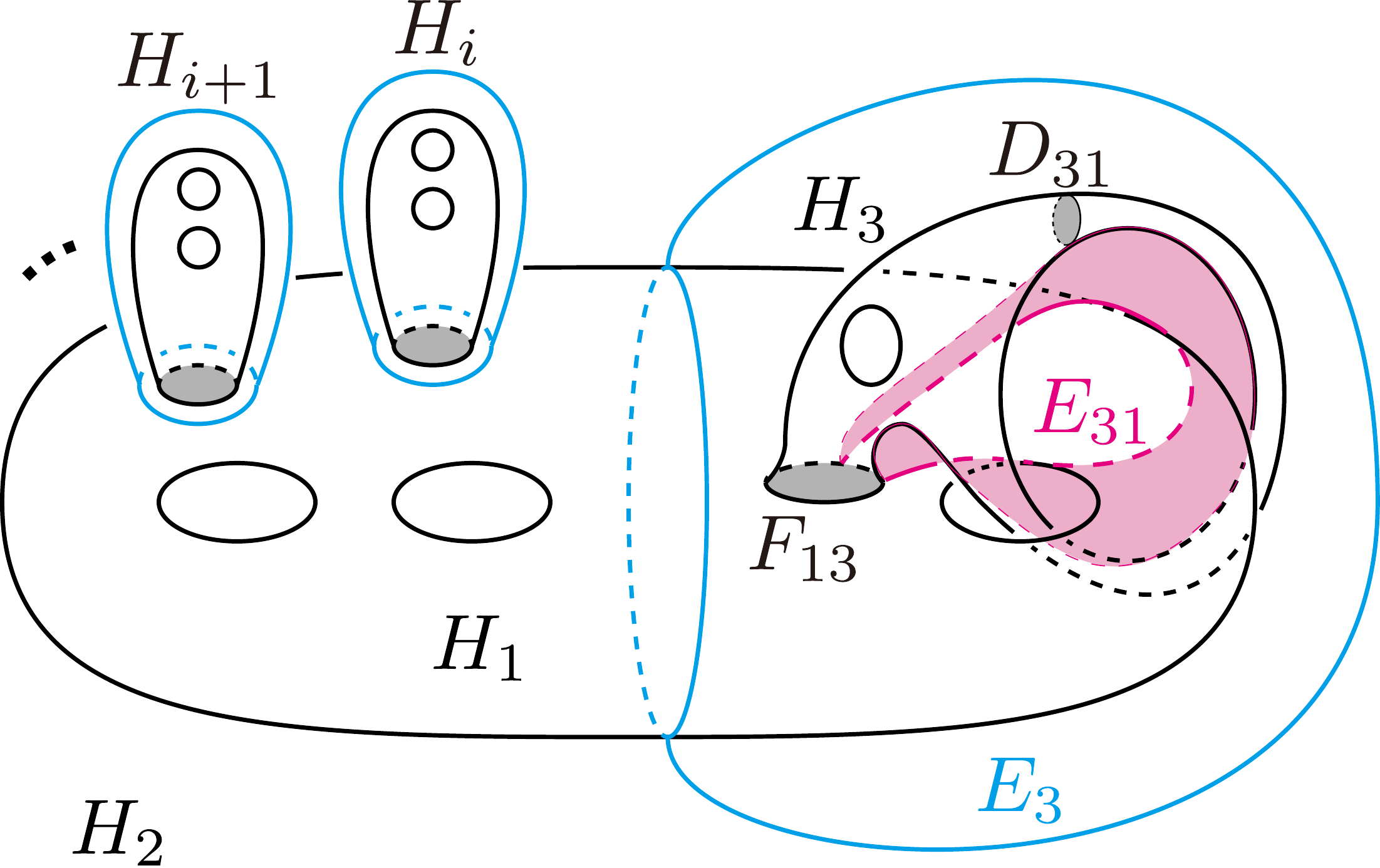}\\
        (a)
    \end{minipage}
    \begin{minipage}{0.48\textwidth}
        \centering
        \includegraphics[width=0.88\linewidth]{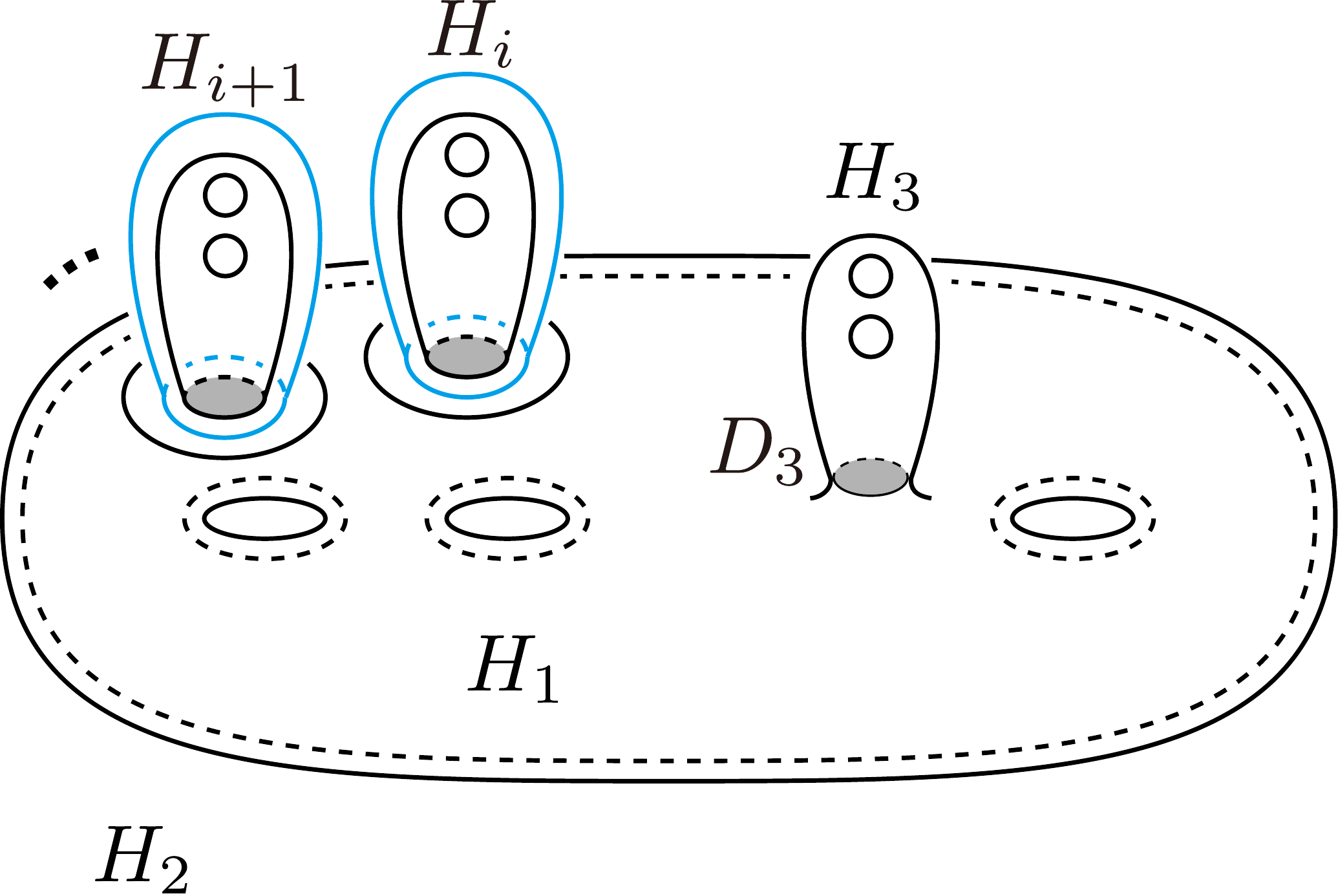}\\
        (b)
    \end{minipage}
    \caption{(a) Before performing the operation in Step~3. (b) After performing the operation. The handlebody $H_3$ is a local unknotted handlebody.}
    \label{fig:handle_slide}
\end{figure}

\medskip
\noindent 	
\underline{Step $i$} ($4\leq i\leq n$)\quad 
At the beginning of Step $i$, we may have $S_3, \ldots, S_{i-1}$ as subsurfaces of $\partial H_3, \ldots, \partial H_{n-1}$, respectively, that are homeomorphic to $S_1\subset \partial H_1$, and $(i-3)$ annuli $A_{1i},A_{3i},\ldots,A_{(i-2)i}\subset\partial H_2$ between $\partial F_{2i}$ and a component of $\partial S_{i-1}$, where $A_{ji}\subset S_j\setminus F_{j(j+1)}$ for each $j\in\{3,\ldots,i-2\}$, $A_{3i}\cap S_{1}=A_{3i}\cap A_{1i}=\partial A_{3i}\cap \partial A_{1i}$, and $A_{ji}\cap S_{j-1}=A_{ji}\cap A_{(j-1)i}=\partial A_{ji}\cap \partial A_{(j-1)i}$ for each $j\in\{4,\ldots,i-2\}$. 


Similar to Step~3, we will cover $H_{i-1}$ along $S_1$ with $H_i$ by performing handle slides of $H_i$. 
By a 0-2 move and a 2-0 move on $A_{1i}$, $F_{1i}$ extends to $A_{1i}$, and an annulus $F_{3i}$ arises. 
Continuing the same operation on $A_{3i},\ldots,A_{(i-2)i}$, $F_{ji}$ ($j\in\{1,3,\ldots i-2\}$) becomes the annulus $A_{ji}$. 
By the same operation as Step 3 on $S_{i-1}$, $S_{i-1}\setminus F_{(i-1)i}$ becomes the union of $n-4$ annuli 
including $A_{(i-1)(i+1)},\ldots, A_{(i-1)n}$. (In the case of $n=4$, $F_{34}$ includes $S_3$ at Step $4$.) 
By the same argument at Step 3, then, $H_i$ can be a local unknotted handlebody after handle slide, and we can assume that $H_i=H\upr_i$.

When we finish Step $n$, we have $H_1 = H_1\upr,\, H_3 = H_3\upr,\, \ldots,\, H_n = H_n\upr$.
Since this automatically implies that $H_2 = H_2\upr$, the proof is completed.

 \end{proof}


\section{Handlebody decomposition consisting of three handlebodies}
\label{sec:chara_decomposition}

In this section, we provide several results of handlebody decompositions consisting of three handlebodies.
We keep assuming that all handlebody decompositions are simple and proper unless otherwise specified.

\ifRSPA
Subsection~\ref{sec:chara_decomposition}\ref{subsec:unstabilization} will consider stabilizability on handlebody decompositions containing a $3$-ball.
\else
Subsection~\ref{subsec:unstabilization} will consider stabilizability on handlebody decompositions containing a $3$-ball.
\fi
In~\cite{Waldhausen}, Waldhausen showed that any genus-$g$ Heegaard splitting of $S^3$ is stabilized for $g \geq 1$.
On the other hand, Koenig found an infinite family of unstabilized type-$(1, 2, 2)$ handlebody decompositions of $S^3$ (see~\cite[Section~6]{Koenig2018}).
We will show that a closed connected orientable $3$-manifold not containing a non-separating sphere admits an unstabilized type-$(0, 0, g)$ handlebody decomposition, where $g$ is the Heegaard genus of the manifold (Proposition~\ref{prop:00l}).
Furthermore, we will see that almost all lens spaces admit a type-$(0, 1, 2)$ handlebody decomposition (Proposition~\ref{prop:012}).
\ifRSPA
In subsection~\ref{sec:chara_decomposition}\ref{subsec:3torus}, we will study handlebody decompositions of the $3$-dimensional torus $\Torus$.
\else
In subsection~\ref{subsec:3torus}, we will study handlebody decompositions of the $3$-dimensional torus $\Torus$.
\fi
These decompositions play an important role in \emph{polycontinuous patterns} (Section~\ref{sec:polyconti}).


\subsection{Handlebody decompositions containing a $3$-ball}
\label{subsec:unstabilization}

We first introduce the result of G\'omez-Larra\~naga~\cite{Gomez87}.
They gave a complete classification of all closed connected $3$-manifolds that admit handlebody decompositions with small genera.

\begin{thm}[{\cite[Propositions~1--3, Theorem~1]{Gomez87}}]\label{thm:small_genera}
    Let $(H_1, H_2, H_3)$ be a type-$(g_1, g_2, g_3)$ handlebody decomposition of a closed connected orientable $3$-manifold $M$ with $g_1 \leq g_2 \leq g_3$.
    We denote by $\mathcal{B}$ the connected sum of some copies of $S^2 \times S^1$, and denote by $\mathcal{L}$ or $\mathcal{L}_i$ a lens space with non-trivial finite fundamental group.
    Then the following hold:
    \begin{enumerate}[label={\textup{(\arabic*)}}]
        \item If all $g_i$ are equal to $0$, then $M$ is homeomorphic to $S^3$ or $\mathcal{B}$.
            Conversely, $S^3$ and $\mathcal{B}$ admit such a handlebody decomposition.
        \item If $g_1 = g_2 = 0$ and $g_3 = 1$, then $M$ is homeomorphic to $S^3$, $\mathcal{B}$, $\mathcal{L}$, or $\mathcal{B} \csum \mathcal{L}$.
            Conversely, these manifolds admit such a handlebody decomposition.
        \item If $g_1 = 0$ and $g_2 = g_3 = 1$, then $M$ is homeomorphic to $S^3$, $\mathcal{B}$, $\mathcal{L}$, $\mathcal{B} \csum \mathcal{L}$, $\mathcal{L}_1 \csum \mathcal{L}_2$, or $\mathcal{L}_1 \csum \mathcal{L}_2 \csum \mathcal{B}$.
            Conversely, these manifolds admit such a handlebody decomposition.
        \item If all $g_i$ are equal to $1$, then $M$ is homeomorphic to $S^3$, $\mathcal{B}$, $\mathcal{L}$, $\mathcal{B} \csum \mathcal{L}$, $\mathcal{L}_1 \csum \mathcal{L}_2$, $\mathcal{L}_1 \csum \mathcal{L}_2 \csum \mathcal{B}$, $\mathcal{L}_1 \csum \mathcal{L}_2 \csum \mathcal{L}_3$, $\mathcal{L}_1 \csum \mathcal{L}_2 \csum \mathcal{L}_3 \csum \mathcal{B}$, $\mathcal{S}(3)$, or $\mathcal{S}(3) \csum \mathcal{B}$,
            where $\mathcal{S}(3)$ denotes a Seifert fiber space with at most three exceptional fibers.
            Conversely, these manifolds admit such a handlebody decomposition.
    \end{enumerate}
\end{thm}


Let $M$ be a closed orientable $3$-manifold with a Heegaard splitting $(W_1, W_2)$ of genus $l$.
Then, we can take $l+1$ non-separating disks in $W_1$ so that they separate $W_1$ into two $3$-balls.
Hence, $M$ admits a type-$(0, 0, l)$ handlebody decomposition (see~\cite[Example~1.2]{Gomez94}).
The following proposition classifies such a decomposition.

\begin{prop}\label{prop:00l}
    Let $M$ be a closed, connected, orientable $3$-manifold of Heegaard genus $g$.
    Suppose that $M$ does not contain a non-separating sphere.
    Then $M$ admits a type-$(0, 0, l)$ handlebody decomposition if and only if we have $g \leq l$.
    In particular, a type-$(0, 0, g)$ handlebody decomposition of $M$ is unstabilized.
\end{prop}

To prove the above proposition, we first show the following lemma.

%

\begin{lemma}\label{lem:giving_Hsplitting}
    Let $(H_1, H_2, H_3; P)$ be a type-$(0, 0, l)$ handlebody decomposition of a closed, connected, orientable $3$-manifold $M$.
    Suppose that $M$ does not contain a non-separating sphere.
    Then $(H_1 \cup H_2, H_3)$ is a genus-$l$ Heegaard splitting of $M$.
\end{lemma}

\begin{proof}
\setcounter{claim}{0}

Let $F_{ij}$ denote a surface as in Remark~\ref{rem:cl_Hdecomp}.
We show that the surface $F_{12}$ consists of disks.
Assume that $F_{12}$ contains a non-disk component $S$.
Then there exists an essential simple loop $C$ in $S$ such that each complementary region of $C$ in $\partial H_1 \cong S^2$ contains a connected component of $F_{13}$.
Since $H_1$ and $H_2$ are $3$-balls, the simple loop $C$ bounds a disk in each of $H_1$ and $H_2$.
Then the union of the two disks is a non-separating disk, which is a contradiction.

Thus, $F_{12}$ consists of only disks.
It follows that the union of $H_1$ and $H_2$ is a handlebody, which implies the assertion.
\end{proof}

\begin{proof}[Proof of Proposition~\textup{\ref{prop:00l}}]
    Let $g$ be the Heegaard genus of a closed, connected, orientable $3$-manifold $M$.
    Then, as explained above, $M$ admits a type-$(0, 0, g)$ handlebody decomposition.
    Then, by Remark~\ref{rem:higher_type_by_type1}, we can obtain a type-$(0, 0, l)$ handlebody decomposition of $M$ for each $g \leq l$.
    Conversely, assume that $M$ admits a type-$(0, 0, l)$ handlebody decomposition.
    By Lemma~\ref{lem:giving_Hsplitting}, this handlebody decomposition induces a Heegaard splitting of genus $l$.
    Thus, we have $g \leq l$.
    This particularly implies that a type-$(0, 0, g)$ handlebody decomposition of $M$ is unstabilized.
\end{proof}


\begin{ex}\label{ex:unstabilized}
    An unstabilized type-$(0, 0, 2)$ handlebody decomposition of $S^3$ is constructed as follows.
    Let $(W_1, W_2)$ be a genus-$2$ Heegaard splitting of $S^3$.
    By using~\cite[Section~5 and Figure~4]{Cho}, we can take a non-primitive disk triple of $W_1$, which separates $W_1$ into two $3$-balls.
    We denote the 3-balls by $H_1$ and $H_2$, and put $H_3 = W_2$.
    Then $(H_1, H_2, H_3)$ forms a type-$(0, 0, 2)$ handlebody decomposition of $S^3$.
    Because each component of $F_{12}$ is a non-primitive disk in $W_1$, we can see that the boundary of any properly embedded disk in $H_3$ transversely intersects the singular graph of the partition in at least six points.
    Hence we can not perform a destabilization along any properly embedded disks in $H_3$.
    Therefore, the decomposition is unstabilized.
\end{ex}

Next, we will consider the stabilizability of type-$(0, 1, l)$ handlebody decompositions.

\begin{prop}\label{prop:01l}
    Let $M$ be a closed, connected, orientable, irreducible $3$-manifold.
    Suppose that $M$ is not a lens space with non-trivial finite fundamental group.
    Then, for each $1 \leq l$, any type-$(0,1,l)$ handlebody decomposition of $M$ is stabilized.
    In fact, such a decomposition is obtained from a type-$(0, 0, l)$ handlebody decomposition by performing a type-$1$ stabilization.
\end{prop}

\begin{proof}
    \setcounter{claim}{0}
    We first assume that $F_{12}$ consists of disks.
    Then there exists a meridian disk of $H_2$ whose boundary intersects $\partial F_{12}$ transversely exactly twice.
    Hence we can perform a type-$1$ destabilization along the meridian disk.

    In the remainder, we assume that $F_{12}$ has a non-disk component.
    \begin{claim}
        We have $\chi(S) \geq 0$ for each component $S$ of $F_{12}$.
    \end{claim}

    \begin{proof}[Proof of Claim~\textup{$1$}]
        Suppose that $\chi(S) < 0$.
        Since $S \subset \partial H_1 \cong S^2$, the boundary $\partial S$ has at least three components.
        Then there exists a component $c$ of $\partial S$ such that it is an inessential loop in $\partial H_2 \cong T^2$.
        Hence $c$ bounds a properly embedded disk in $H_2$.
        The closed curve $c$ also bounds a properly embedded disk in $H_1$ since $H_1$ is a $3$-ball.
        Because each complementary region of $c$ in $\partial H_1 \cong S^2$ contains a component of $F_{13}$, the two properly embedded disks in $H_1$ and $H_2$ form a non-separating sphere.
        This contradicts the irreducibility of $M$.
    \end{proof}

    \begin{claim}
        A core curve of each annulus component of $F_{12}$ is essential in $\partial H_2 \cong T^2$.
    \end{claim}

    \begin{proof}[Proof of Claim~\textup{$2$}]
        Assume that a core curve $C$ of an annulus component of $F_{12}$ is inessential in $\partial H_2$.
        Then $C$ bounds a properly embedded disk in $H_2$, and each complementary region of $C$ intersects $F_{23}$.
        Since $H_1$ is a $3$-ball, $C$ also bounds a properly embedded disk in $H_1$.
        Hence the two disks form a non-separating sphere.
        This is a contradiction.
    \end{proof}

    \begin{claim}
        The surface $F_{12}$ contains precisely one annulus component.
    \end{claim}
    \begin{proof}[Proof of Claim~\textup{$3$}]
        We assume that $F_{12}$ contains two annulus components.
        Then, by Claim~2, their core curves, $C_1$ and $C_2$, are parallel essential loops in $\partial H_2 \cong T^2$.
        Thus, $C_1 \cup C_2$ cobounds a properly embedded annulus in $H_2$.
        Since $H_1$ is a $3$-ball, each of $C_1$ and $C_2$ bounds a disk in $H_1$.
        Because each complementary region of $C_1 \cup C_2$ in $\partial H_2$ intersects $F_{23}$,
        the union of the annulus and the disks is a non-separating sphere.
        This is a contradiction.
    \end{proof}

    Let $C$ be a core curve of the annulus component $A$ of $F_{12}$.
    Then we have $[C] = a \mu + b \lambda \in \mathrm{H_1}(\partial H_{2})$, where $a, b \in \mathbb{Z}$, and $\mu$ and $\lambda$ denote homology classes of a meridian loop and a longitude loop, respectively.
    Since $H_1$ is a $3$-ball, each component of $\partial A$ bounds a properly embedded disk in $H_1$.
    Then the union $S$ of $\partial H_2 \setminus \inte(A)$ and the two disks is a separating sphere in $M$.
    Thus, if $b \neq 0$ and $\pm 1$, then $M$ has a lens space as a connected summand.
    However, $M$ is irreducible and not a lens space.
    Hence we have $b = 0$ or $\pm 1$.
    If $b = 0$, then $C$ bounds a properly embedded disk in $H_2$.
    The curve $C$ also bounds a properly embedded disk in $H_1$ since $H_1$ is a $3$-ball.
    So, the disks form a non-separating sphere, which is a contradiction.
    Hence we have $b = \pm 1$.
    Thus, we can take a meridian disk of $H_2$ that intersects the boundary of $F_{12}$ transversely exactly two points.
    Therefore, we can perform a type-$1$ destabilization along the meridian disk.
\end{proof}

The following proposition implies that the assumption that $M$ is not a lens space in Proposition~\ref{prop:01l} is essential.

\begin{prop}\label{prop:012}
    Any lens space with non-trivial finite fundamental group admits an unstabilized type-$(0, 1, 2)$ handlebody decomposition.
\end{prop}
\begin{proof}
    Let $(W_1, W_2)$ be a genus-$2$ Heegaard splitting of $S^3$.
    Then there exists a pair of non-primitive disks $D_1$ and $D_2$ in $W_1$ as in Example~\ref{ex:unstabilized}.
    Note that $B := N(D_1; W_1) \cup N(D_2; W_1) \cup W_2$ is a $3$-ball, where $N(D_1; W_1)$ and $N(D_2; W_1)$ are regular neighborhoods of $D_1$ and $D_2$, respectively.
    We take an unknotted arc $\delta$ in the interior of $W_1$ that joins $D_1$ and $D_2$ and intersects them at only its endpoints.
    Let $h$ be a $1$-handle attached to sides of each $N(D_1; W_1)$ and $N(D_2; W_1)$ along $\delta$.
    Then, $h \cup B$ is a solid torus (see Figure~\ref{fig:012}).
    Hence, for any lens space $M$, there exists a homeomorphism $\psi$ from $\partial (h \cup B)$ to the boundary of a solid torus $H_2$ such that $M$ is homeomorphic to the manifold pasted by $h \cup B$ and $H_2$ along $\psi$. 
    We put $H_1 = h \cup N(D_1; W_1) \cup N(D_2; W_1)$, $H_3 = W_2$.
    Thus, $(H_1, H_2, H_3)$ is a type-$(0, 1, 2)$ handlebody decomposition of $M$.
    By the construction, each meridian disk of $H_2$ and $H_3$ intersects the singular graph at least four and six times, respectively.
    Hence, the handlebody decomposition is unstabilized.
%
%
%
%
\end{proof}

\begin{figure}[htbp]
    \centering
    \includegraphics[width=.63\textwidth]{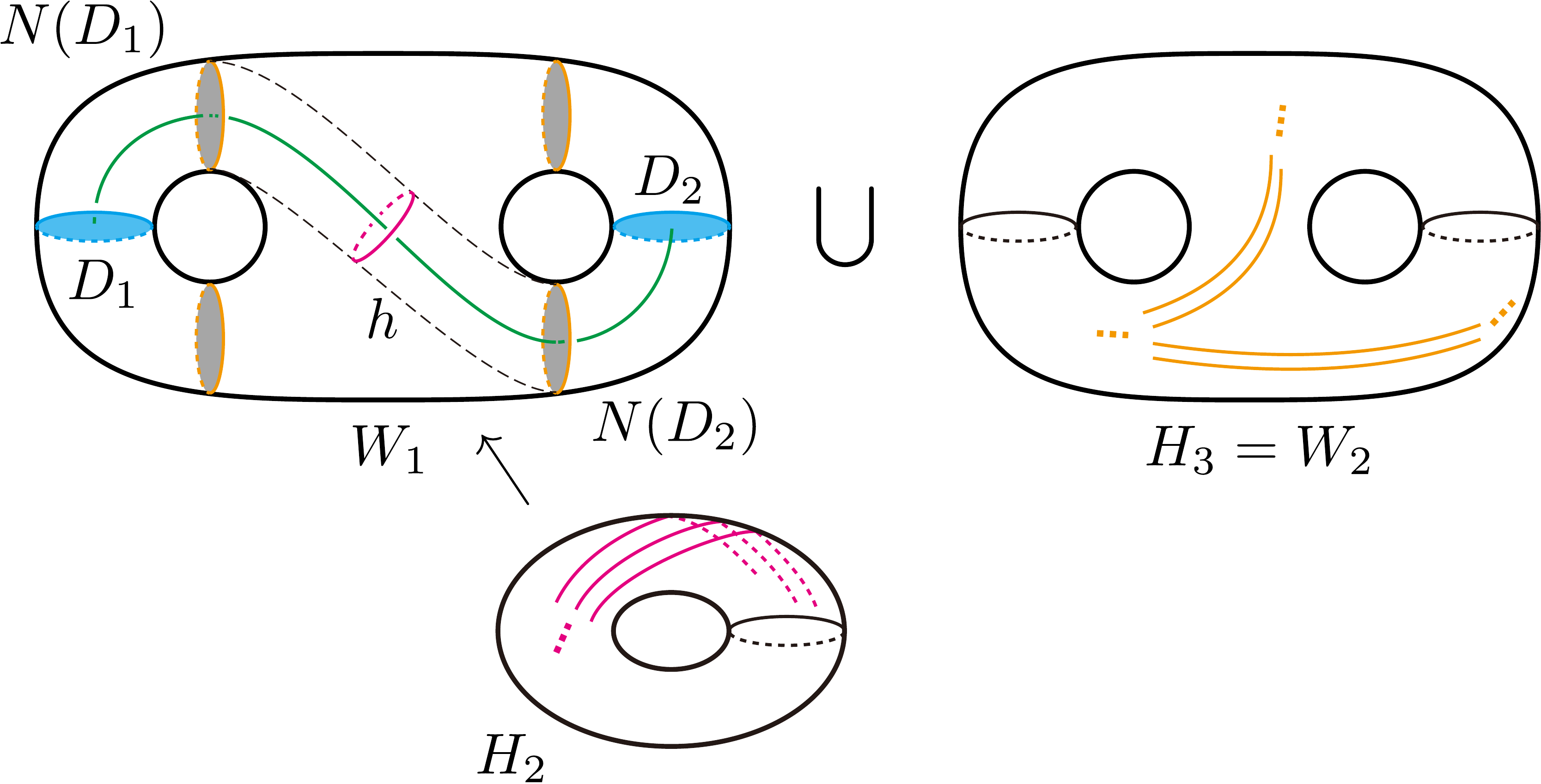}
    \caption{An unstabilized type-$(0, 1, 2)$ handlebody decomposition of a lens space with non-trivial finite fundamental group. The decomposition consists of three handlebodies $H_1 = h \cup N(D_1) \cup N(D_2)$, $H_2$, and $H_3 = W_2$.}
    \label{fig:012}
\end{figure}

\subsection{Examples: the $3$-dimensional torus}
\label{subsec:3torus}

We will show some examples of handlebody decompositions of the $3$-dimensional torus $T^3$.


%

First, we consider handlebody decompositions consisting of one ball and two handlebodies.
By Proposition~\ref{prop:00l}, $\Torus$ admits a unstabilized type-$(0, 0, 3)$ handlebody decomposition (see Figure~\ref{fig:nxx}(a)).
Thus, for $k \geq 0$ and $l \geq 3$, $\Torus$ admits a type-$(0, k, l)$ handlebody decomposition by Remark~\ref{rem:higher_type_by_type1}.
Figure~\ref{fig:nxx}(b) illustrates a type-$(0, 2, 2)$ handlebody decomposition of $\Torus$.
On the other hand, by Theorem~\ref{thm:small_genera}, $\Torus$ admits neither type-$(0, 0, 0)$, type-$(0, 0, 1)$, nor type-$(0, 1, 1)$ handlebody decompositions.
In addition, by Propositions~\ref{prop:00l} and~\ref{prop:01l}, there is no type-$(0, 1, 2)$ handlebody decomposition of $\Torus$.
Hence, any type-$(0, 2, 2)$ handlebody decomposition of $\Torus$ is unstabilized.
In summary, we have the following proposition.

\begin{prop}\label{prop:handledecomp_T3}
    Let $(k, l)$ be a pair of non-negative integers with $k \leq l$.
    The $3$-dimensional torus $T^3$ admits a type-$(0, k, l)$ handlebody decomposition if and only if the pair $(k, l)$ is not in $\{ 0, 1 \} \times \{ 0, 1, 2 \}$.
\end{prop}





\begin{figure}[htbp]
    \centering
    \begin{minipage}[b]{0.24\textwidth}
        \centering
        \includegraphics[width=0.80\textwidth]{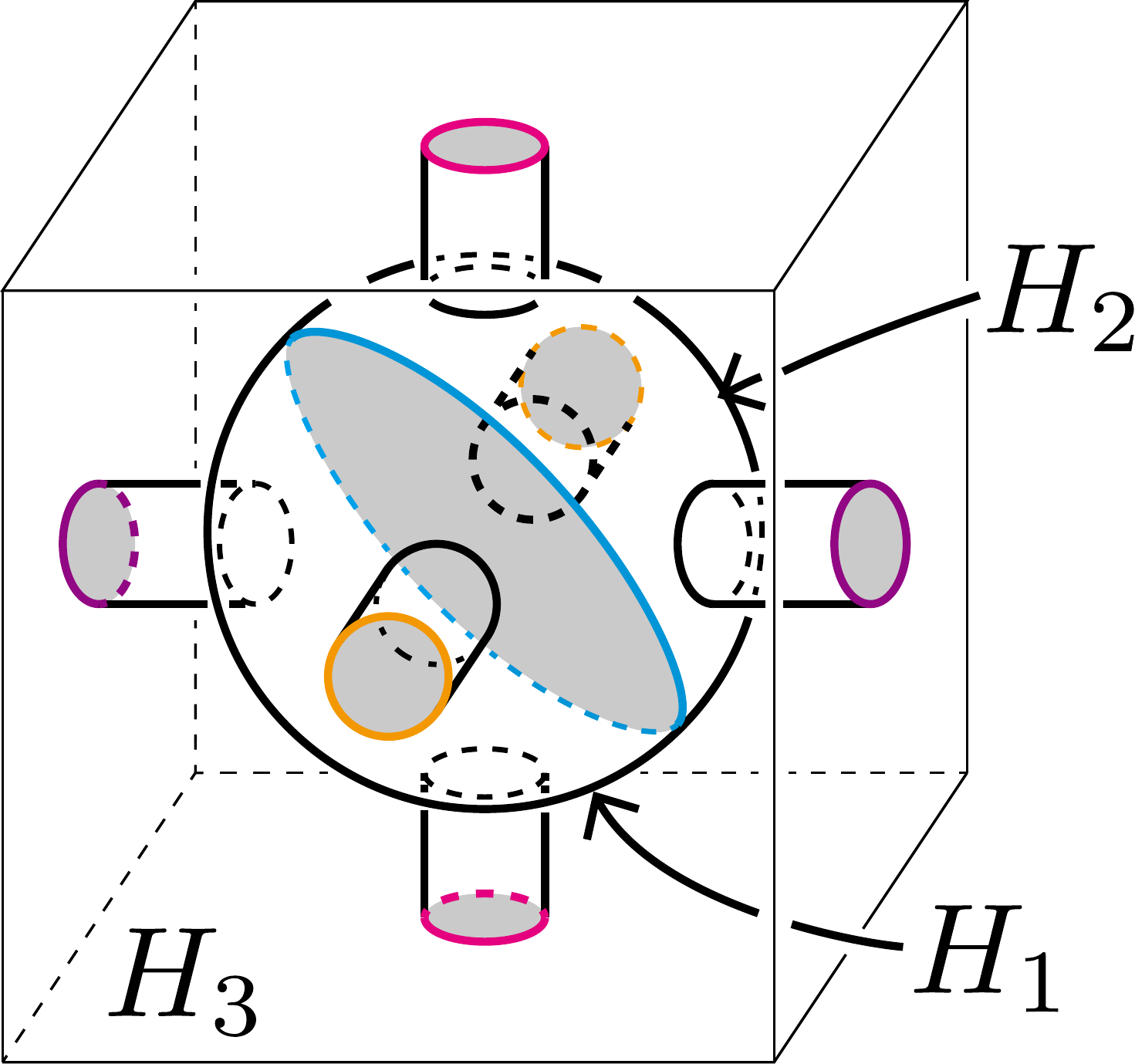}\\
        (a)
    \end{minipage}
    \begin{minipage}[b]{0.24\textwidth}
        \centering
        \includegraphics[width=0.80\textwidth]{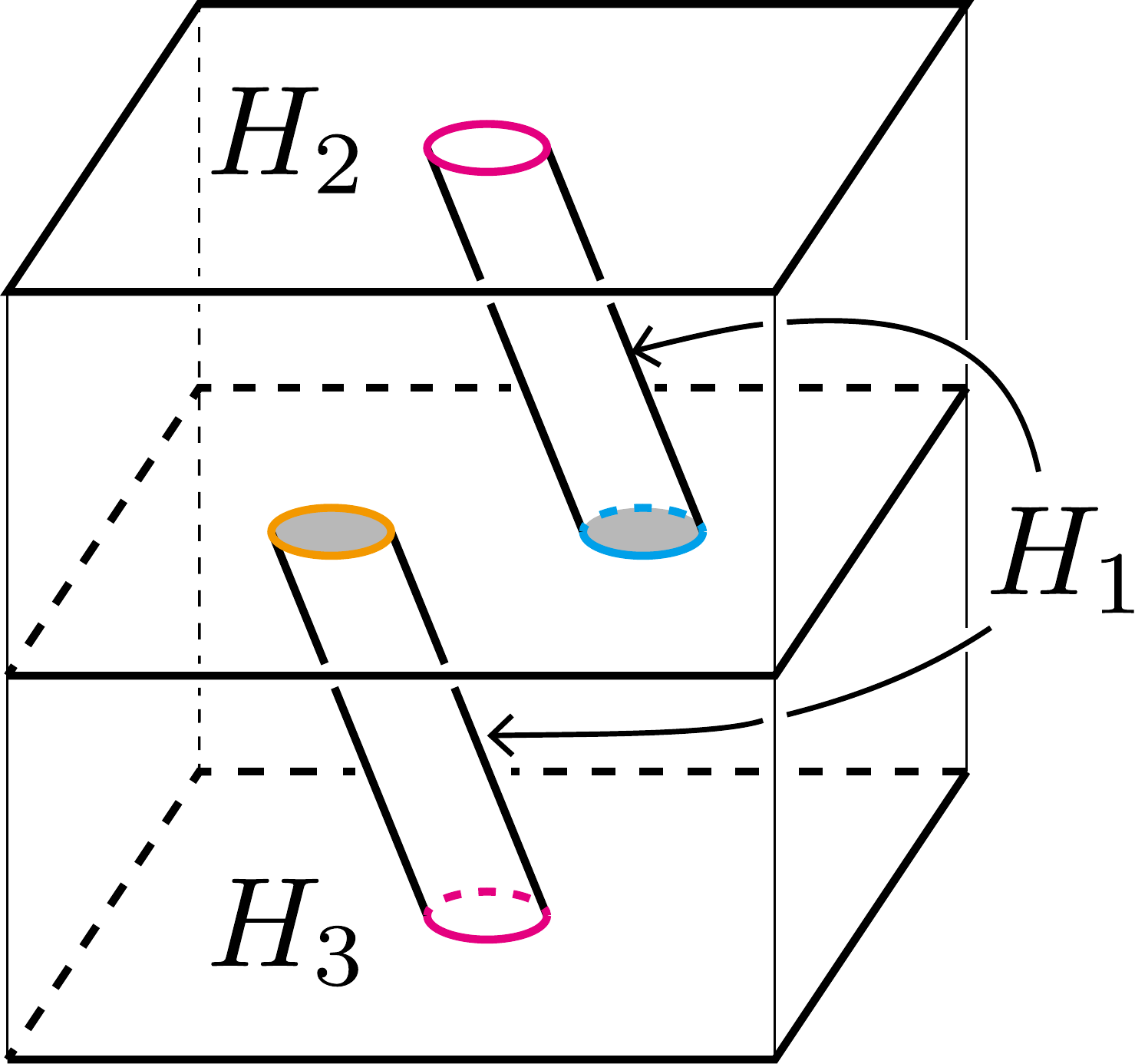}\\
        (b)
    \end{minipage}
    \begin{minipage}[b]{0.24\textwidth}
        \centering
        \includegraphics[width=0.70\textwidth]{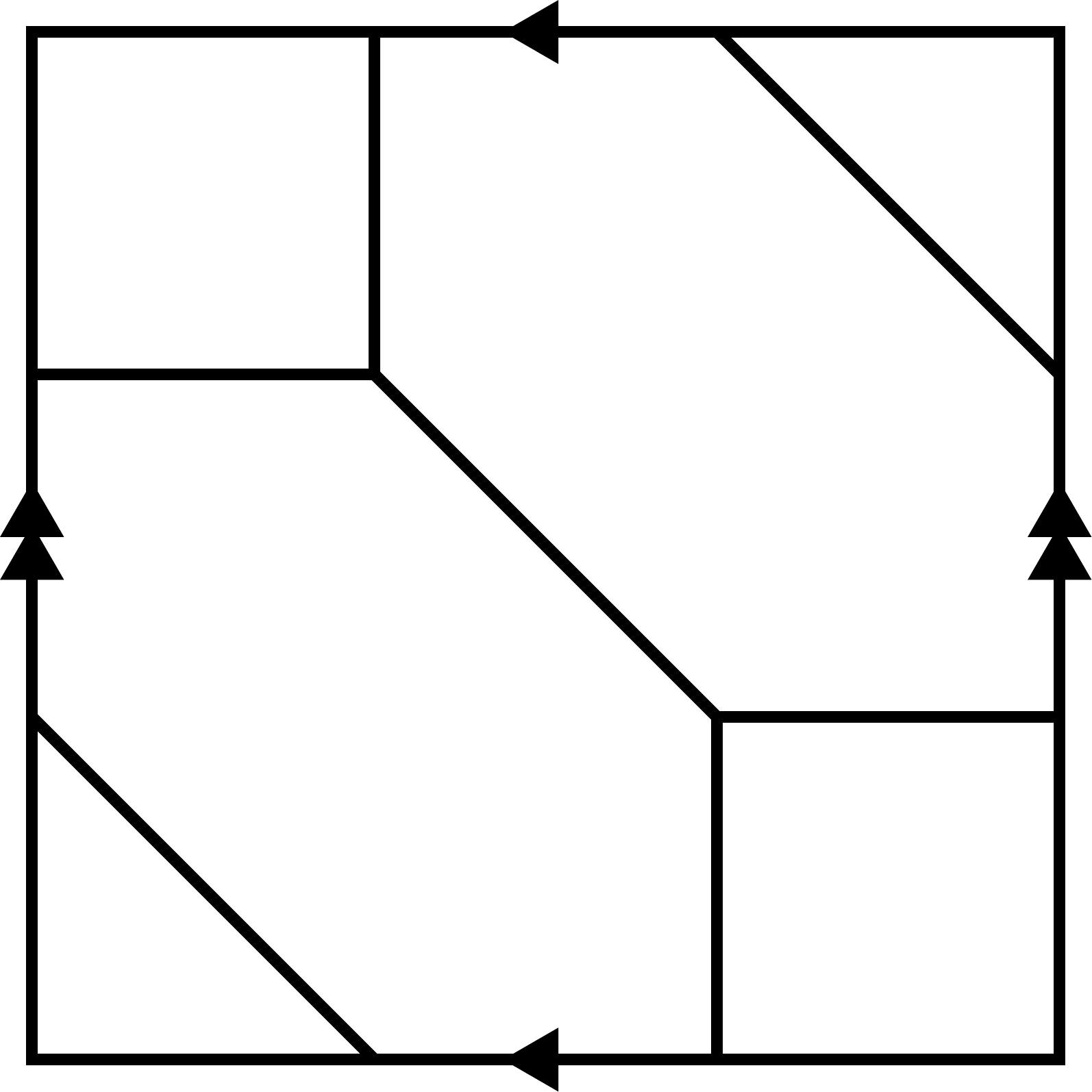}\\
        (c)
    \end{minipage}
    \begin{minipage}[b]{0.24\textwidth}
        \centering
        \includegraphics[width=0.70\textwidth]{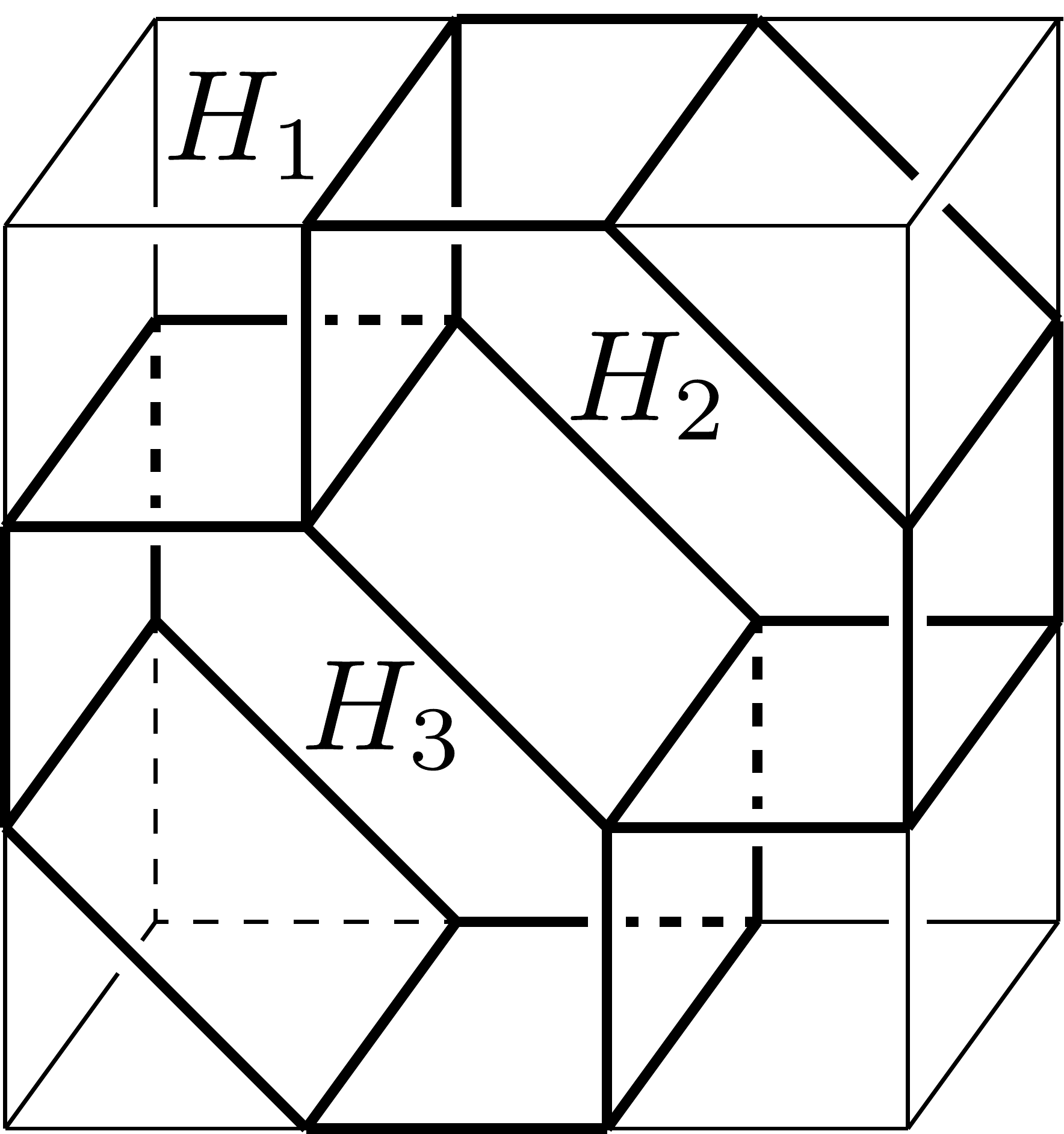}\\
        (d)
    \end{minipage}
    \caption{(a) A type-$(0,0,3)$ handlebody decomposition of $\Torus$. Colored edges illustrate the singular graph, and grayed disks are all components of $F_{12}$ as in Remark~\ref{rem:cl_Hdecomp}. %
    (b) A type-$(0, 2, 2)$ handlebody decomposition of $T^3$. %
    (c) A decomposition of $T^2$ into three hexagons. %
    (d) The hexagonal honeycomb decomposition of $\Torus$.}
\label{fig:nxx}
\end{figure}




Theorem~\ref{thm:small_genera} guarantees that $\Torus$ admits a type-$(1, 1, 1)$ handlebody decomposition.
Figure~\ref{fig:nxx}(c) shows a decomposition of $T^2$ into three hexagons.
By taking the product with $S^1$, we have a decomposition of $T^3$ into three solid tori.
We call this handlebody decomposition the \emph{hexagonal honeycomb decomposition} (see Figure~\ref{fig:nxx}(d)).
In general, a $3$-manifold admits a lot of handlebody decompositions of the same type.
The next proposition asserts that the hexagonal honeycomb decomposition is the unique type-$(1, 1, 1)$ handlebody decomposition of $\Torus$ up to self-homeomorphism of $\Torus$.

\begin{prop}\label{prop:honeycomb_decomp}
    For a simple and proper type-$(1, 1, 1)$ handlebody decomposition of $\Torus$, there exists a self-homeomorphism of $\Torus$ that maps the partition of the decomposition to that of the hexagonal honeycomb decomposition.
\end{prop}

\begin{proof}
\setcounter{claim}{0}
Let $(H_1, H_2, H_3)$ be a simple and proper type-$(1, 1, 1)$ handlebody decomposition of $\Torus$.
Let $F_{ij}$ denote a surface as in Remark~\ref{rem:cl_Hdecomp}.

\begin{claim}
For any $1 \leq i < j \leq 3$, there is no disk component in $F_{ij}=H_i\cap H_j$.
\end{claim}
\begin{proof}[Proof of Claim~\textup{$1$}]
    Suppose there is a disk component $D$ in some $F_{ij}$.
    Without loss of generality, we can assume that $D \subset F_{23}$.
    If $\partial D$ is essential in $\partial H_1$,
    $H\upr_1=H_1 \cup N(D)$ is a punctured lens space.
    Since $T^3$ is prime, $H\upr_1$ is a $3$-ball.
    Thus, a triple $(H\upr_1, \mathrm{cl}(H_2 \setminus N(D)), \mathrm{cl}(H_3 \setminus N(D)))$ gives a simple and proper type-$(0, 1, 1)$ handlebody decomposition of $T^3$.
    However, from Proposition~\ref{prop:handledecomp_T3}, $T^3$ does not admit such a decomposition, which is a contradiction.

    Suppose $\partial D$ is inessential in $\partial H_1$.
    Then we can take a disk $D\upr$ in $\partial H_1$ such that $\partial D\upr=\partial D$.
    Put $S=D\cup D\upr$.
    Since $\Torus$ does not contain non-separating spheres, $S$ is separating.
    Thus, $F_{23}$ consists of only the disk $D$.
    Hence, $H_2 \cup H_3$ is a genus-$2$ handlebody.
    This is impossible because $\partial (H_2 \cup H_3) = \partial H_1$ is a torus.
\end{proof}

By Claim~1 and~\cite[LEMMA~1]{Gomez87}, each component of $F_{ij}$ is an annulus.
Suppose the core of an annulus of $F_{ij}$ is meridional in $H_i$.
Then, $C$ is longitudinal in $H_j$; otherwise, we can find a punctured lens space or a punctured $S^2 \times S^1$ in $T^3$, which is a contradiction.
Then, by removing the neighborhood of a meridian disk in $H_i$ and attaching it to $H_j$, 
we have a decomposition of $T^3$ with two $3$-balls and a solid torus, 
i.e., a type-$(0,0,1)$ decomposition, which contradicts Proposition~\ref{prop:handledecomp_T3}.
Therefore, $H_1$, $H_2$, and $H_3$ are fiber tori of a Seifert fibration of $\Torus$.
As the Seifert fiber structure of $T^3$ is unique up to self-homeomorphism of $\Torus$,
we can assume that $H_i = D_i \times S^1$ ($i = 1, 2, 3$), where $D_1$, $D_2$, and $D_3$ are disks in $T^2$ satisfying $D_1\cup D_2\cup D_3 = T^2$.
By the Euler characteristic, the intersection of two of the disks consists of precisely three arcs.
Hence, this structure corresponds to the hexagonal honeycomb decomposition in $\Torus$.
\end{proof}

When it comes to the case of type-$(2, 2, 2)$ handlebody decompositions,
the uniqueness no longer holds, as we see in the following example.

\begin{ex}\label{ex:nonequiv}
    By performing type-$1$ stabilizations to the honeycomb decomposition three times, we have the two type-$(2, 2, 2)$ handlebody decompositions shown in Figure~\ref{fig:type222}.
    In Figure~\ref{fig:type222}(a), each $F_{ij} = H_i \cap H_j$ is homeomorphic to the disjoint union of a $2$-holed torus and a disk.
    On the other hand, in Figure~\ref{fig:type222}(b), the surface $F_{23}$ is homeomorphic to the disjoint union of a $1$-holed torus and an annulus, and each $F_{1j}$ is homeomorphic to a $3$-holed sphere.
    Hence, they are different decompositions.
\end{ex}

\begin{figure}[htbp]
    \centering
    \begin{minipage}{0.22\textwidth}
        \centering
        \includegraphics[width=0.90\textwidth]{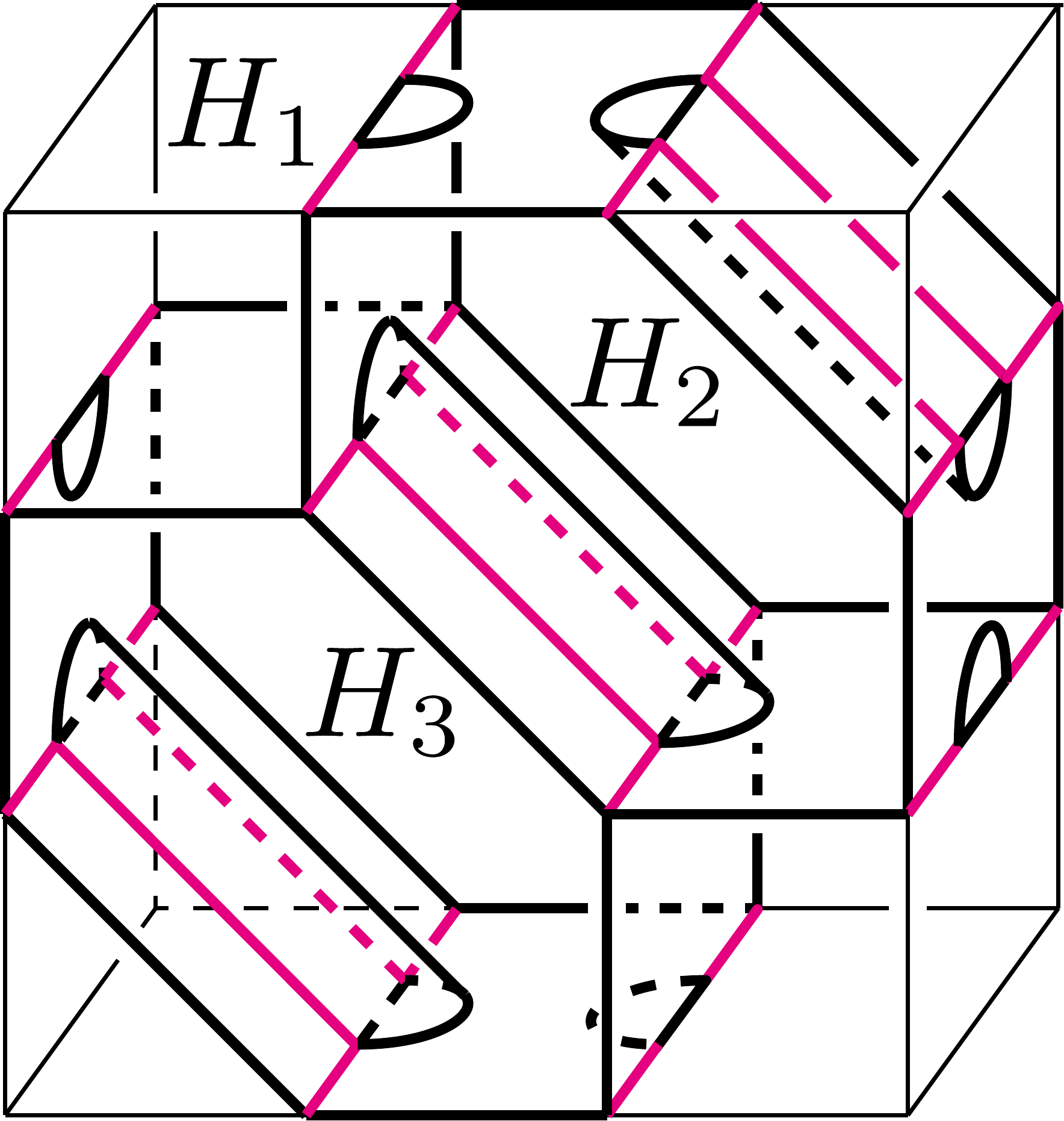}\\
        (a)
    \end{minipage}
    \begin{minipage}{0.22\textwidth}
        \centering
        \includegraphics[width=0.90\textwidth]{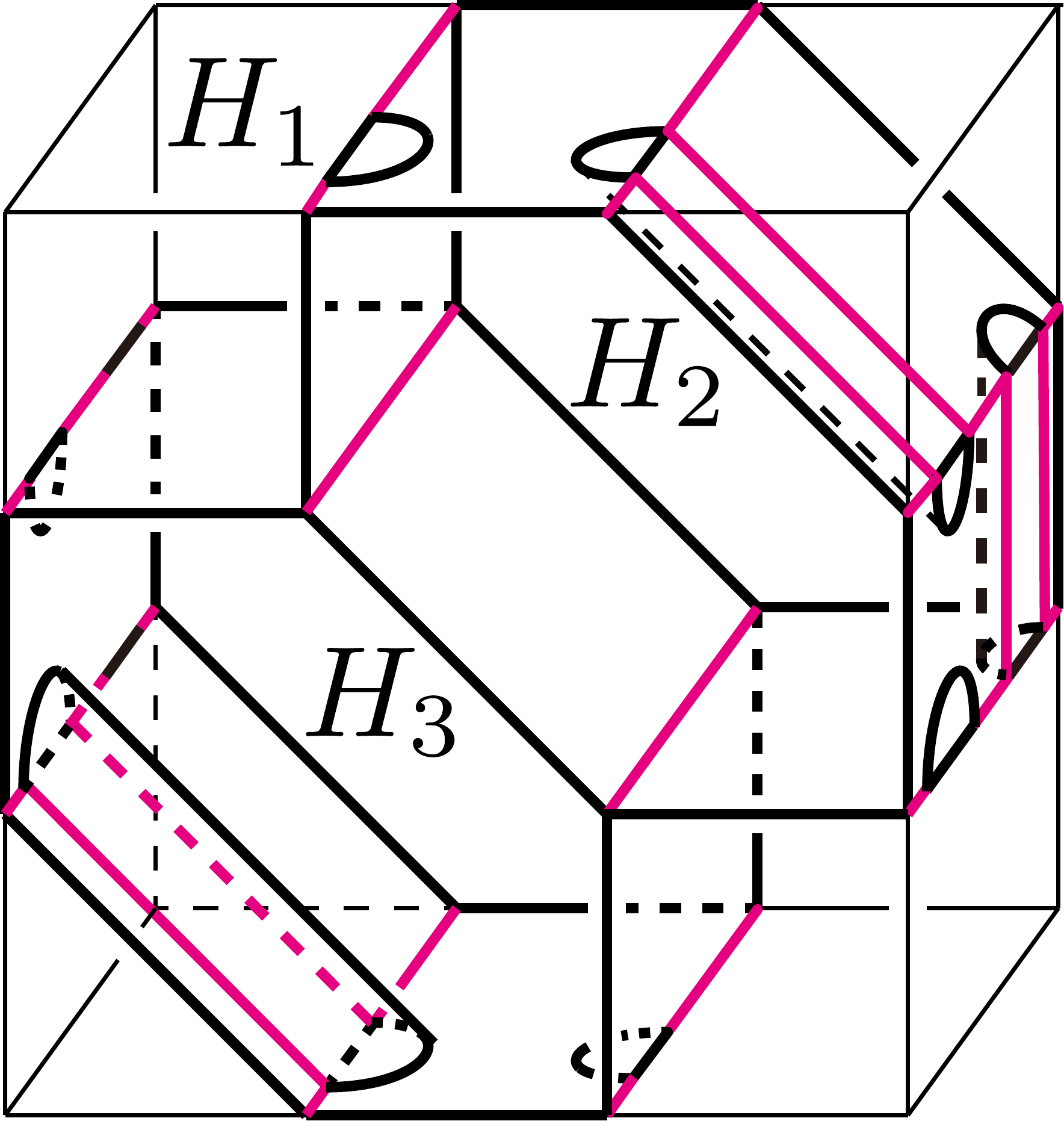}\\
        (b)
    \end{minipage}
    \caption{A pair of different type-$(2, 2, 2)$ handlebody decompositions of $\Torus$.}
\label{fig:type222}
\end{figure}

\section{Topological study of polycontinuous patterns}
\label{sec:polyconti}

In this section, we will consider ``polycontinuous patterns,'' which are roughly $3$-periodic structures assembled by polymers.
See, for example, \cite{Hyde1, Sch} for studies on polycontinuous patterns.
We will suggest a mathematical model of polycontinuous patterns (Definition~\ref{dfn:polyconti}).

\subsection{Polycontinuous patterns and net-like patterns}

First, we define ``net-like patterns'' that satisfy the essential properties of polycontinuous patterns.


\begin{dfn}
    We denote by $T^d$ the $d$-dimensional torus.
    Let $\widetilde{X}$ be a graph embedded in $\R^d$ such that each component of $\widetilde{X}$ is unbounded.
    If there exists a covering map $\pi: \R^d \to T^d$ such that all covering transformations of $\pi$ preserve $\widetilde{X}$, then $\widetilde{X}$ is called a \emph{net}.
\end{dfn}

In this paper, we mainly discuss the case where $d=3$.

\begin{rem}
    In crystal chemistry (e.g.~\cite{Delgado}), the term ``net'' means a periodic, connected, simple, abstract graph.
    In this paper, we allow a net to be disconnected.
    Furthermore, all nets are embedded in Euclidean space.
\end{rem}

\begin{dfn}\label{dfn:pattern}
    Let $\widetilde{P}$ be a non-compact connected $2$-dimensional polyhedron embedded in $\mathbb{R}^3$.
    The polyhedron $\widetilde{P}$ is called a \emph{\patt} if there exist a covering map $\pi: \R^3 \to \Torus$ and a net $\widetilde{X}$ such that the following conditions hold:
    \begin{enumerate}[label=\textup{(\arabic*)}]
        \item All covering transformations of $\pi$ preserve both $\widetilde{P}$ and $\wt{X}$.
        \item The polyhedron $\widetilde{P}$ divides $\mathbb{R}^3$ into unbounded open components $V_i$ $(i \in I)$, where $I$ is a finite or countable set.
        \item There exists a strong deformation retraction of $\mathbb{R}^3 \setminus \wtP$ onto $\widetilde{X}$.
    \end{enumerate}
    We call the pair $(\widetilde{P}, \pi)$ a \emph{framed \patt}, and $\pi$ its \emph{frame}.
    We say that a connected component of $\R^3 \setminus \wtP$ (resp. $\wt{X}$) is a \emph{labyrinthine domain} (resp. \emph{labyrinthine net}) of $\wtP$.

    A \patt $\wtP$ is said to be \emph{proper} if there is no simple closed curve in $\mathbb{R}^3$ that does not cross the singular graph and intersects a sector of $\widetilde{P}$ transversely once.
    A \patt $\wtP$ is said to be \emph{simple} if $\widetilde{P}$ is a simple polyhedron. 
\end{dfn}

More generally, we can define net-like patterns for any closed prime $3$-manifold with a (possibly non-Euclidean) crystallographic group and its covering space, but this paper will not deal with it.

\begin{rem}
    Consider two \patts that satisfy the following conditions:
    \begin{enumerate}[label=(\arabic*)]
        \item They have the same labyrinthine net.
        \item They do not have a disk sector.
        \item The singular graphs of them have no vertices.
    \end{enumerate}
    Then, they can be transformed to each other by a (possibly infinite) sequence of \emph{IX-moves}, \emph{XI-moves}, and isotopies (see~\cite[Theorem~3.1]{Nei}).

    By using~\cite{Matveev88,Piergallini}, if two \patts with the same labyrinthine net are simple, then the patterns can be transformed to each other by a (possibly infinite) sequence of 0-2 moves, 2-0 moves, 2-3 moves, 3-2 moves, and isotopies.
\end{rem}


The following two propositions state a relationship between (framed) \patts and handlebody decompositions of $\Torus$.

\begin{prop}\label{prop:decomp2pattern}
    Let $(H_1, H_2, \ldots, H_n; P)$ be a handlebody decomposition of $T^3$,
    and $\wtP$ the preimage of $P$ under the universal covering map $\pi$ of $T^3$.
    Suppose that, for each $i$, the induced homomorphism $(\iota_i)_\ast: \fund(H_i) \to \fund(T^3)$ is not trivial, where $\iota_i$ is the inclusion map.
    Then the pair $(\wtP, \pi)$ is a framed \patt.
    Furthermore, if $P$ is simple (resp. proper), then the \patt is also simple (resp. proper).
\end{prop}

\begin{proof}
    Let $\{ V^i_j \}$ be the connected components of the preimage of $H_i$ under $\pi$.
    Since the homomorphism $(\iota_i)_\ast$ is not trivial, each open component $V^i_j$ is unbounded.
    Each open handlebody $H_i$ contains a simple finite graph $X_i$ that is a strong deformation retract of $H_i$.
    Then the preimage, $\widetilde{X}_i$, of $X_i$ under $\pi$ is a net.
    Furthermore, each connected component of $\widetilde{X}_i$ is a strong deformation retract of some $V^i_j$.
    Hence, $(\wtP, \pi)$ is a framed \patt.

    Since $\pi$ is a local homeomorphism, if $P$ is simple, then the preimage $\wtP$ is also simple.
    Next, assume that the handlebody decomposition $(H_1, H_2, \ldots, H_n; P)$ of $T^3$ is proper,
    whereas the \patts $\wtP$ is \emph{not} proper.
    Then, there exists a simple loop $\widetilde{c}$ in $\mathbb{R}^3$ that transversely intersects $P$ at a single point only in a sector.
    Thus, there exists a simple loop in $\Torus$ isotopic to $\pi(\widetilde{c})$ that intersects a sector of $P$ transversely once.
    Hence the handlebody decomposition is not proper, which is a contradiction.
    Therefore the \patt $\wtP$ is proper.
\end{proof}

\begin{prop}\label{prop:pattern2decomp}
    Let $(\wtP, \pi)$ be a framed \patt.
    \begin{enumerate}[label=\textup{(\arabic*)},leftmargin=*]
        \item The image $\pi(\wtP)$ gives a handlebody decomposition of $T^3$.
            If $\wtP$ is simple, then the handlebody decomposition is also simple.
        \item 
            Let $\{V_i\}_{i \in I}$ be the set of labyrinthine domains of $\wtP$, where $I$ is a finite or countable set.
            Suppose that $\wtP$ is proper.
            Suppose further that for any $V_i$, $V_j$ with $V_i \neq V_j$, where $V_i$ is the image of $V_j$ under some covering transformation, $V_i$ and $V_j$ are not adjacent to the same sector.
            Then the handlebody decomposition given by $\pi(\wtP)$ is proper.
    \end{enumerate}
\end{prop}

\begin{proof}
    Let $\Gamma$ be the covering transformation group of $\pi$.
    Set $P = \pi(\wtP)$.

    (1)
    Since $\wtP$ is a connected $2$-dimensional polyhedron,
    its projection image $P$ is also a connected $2$-dimensional polyhedron.
    Furthermore, if $\wtP$ is simple, then $P$ is also simple.

    The complement $T^3 \setminus P$ consists of finite open components $\{ H_j \}$ because $T^3 = \mathbb{R}^3 / \Gamma$ is compact and $P$ is the underlying space of a locally finite complex.
    We show that each open component $H_j$ is an open handlebody.
    There exists a labyrinthine domain $V_i$ such that $H_j = \pi(V_i)$.
    Furthermore, since $\wtP$ is a net-like pattern, there exists a labyrinthine net $\wt{X}_i$ such that $\wt{X}_i$ is a strong deformation retract of $V_i$.
    Put $G = \pi(\wt{X}_i)$.
    Then, $G$ is an embedding of a graph in $\Torus$.
    So, the fundamental group $\fund(G)$ is free.
    Since $\pi |_{V_i}$ is a covering map, and $\wt{X}_i$ is a strong deformation retract of $V_i$,
    the inclusion map $G \to H_j$ induces an isomorphism from $\fund(G)$ to $\fund(H_j)$.
    Hence, $H_j$ is the interior of a handlebody because $\fund(H_j)$ is free.
    Therefore, $P$ gives a handlebody decomposition of $T^3$.

    (2)
    We suppose that $\wt{P}$ is proper and $P$ is not proper.
    Then there exists a simple loop $c: [0, 1] \to T^3$ such that it transversely intersects $P$ at a single point only in a sector.
    Let $\widetilde{c}$ be a lift of $c$. 
    Since $\wt{P}$ is proper, $\wt{c}$ is an arc (not a loop) whose initial point $v$ and terminal point $w$ are contained in different labyrinthine domains $V_i$ and $V_j$, respectively, of $\mathbb{R}^3 \setminus \wtP$.
    Note that $V_i$ and $V_j$ are adjacent.
    This is impossible because there exists a covering transformation that takes $v$ to $w$, so $V_i$ to $V_j$.
\end{proof}

\begin{ex}\label{ex:honeycomb_pattern}
    Figure~\ref{fig:colored_patterns}(a) illustrates a simple proper \patt that comes from the hexagonal honeycomb tessellation of $\R^2$.
    A yellow polygon illustrates a fundamental domain of its frame.
    We call the pattern the {\em hexagonal honeycomb pattern}.
    By Proposition~\ref{prop:pattern2decomp}, the hexagonal honeycomb pattern with the frame induces the hexagonal honeycomb decomposition (see Figure~\ref{fig:nxx}(d)).
    Note that a tessellation of $\R^2$ induces a \patt in general.
    The meanings of colors except yellow in Figure~\ref{fig:colored_patterns}(a) will be explained in Definition~\ref{dfn:coloring}.
\end{ex}

\begin{figure}[htbp]
    \centering
    \begin{minipage}[b]{0.48\textwidth}
        \centering
    \includegraphics[width=0.8\textwidth]{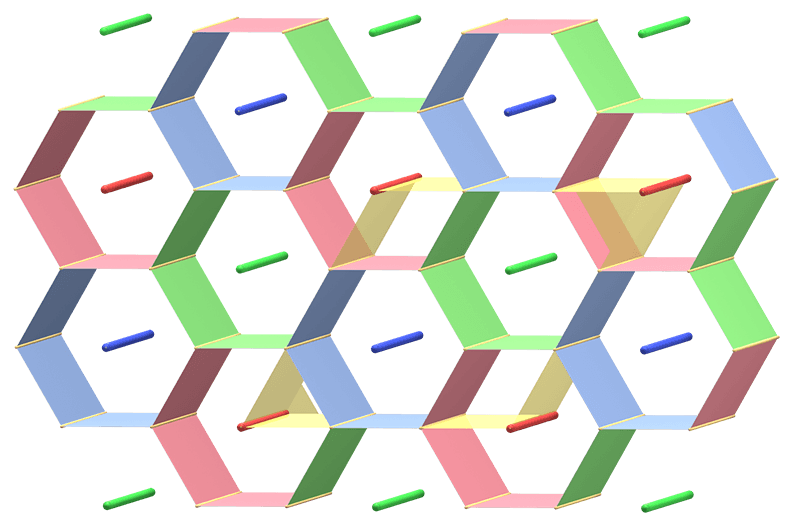}\\
        (a)
    \end{minipage}
    \begin{minipage}[b]{0.48\textwidth}
        \centering
        \includegraphics[width=0.8\textwidth]{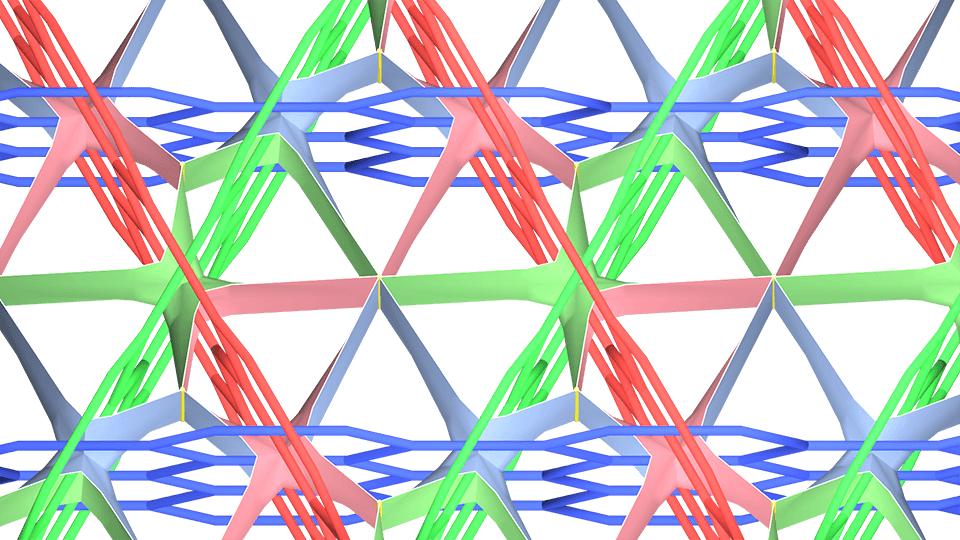}\\
        (b)
    \end{minipage}
    \caption{(a) The hexagonal honeycomb pattern. (b) A non-simple \patt.}
    \label{fig:colored_patterns}
\end{figure}

We now suggest a strict mathematical definition of polycontinuous patterns.

\begin{dfn}\label{dfn:polyconti}
    We say that a \patt $\wtP$ is an \emph{$n$-continuous pattern} (or a \emph{polycontinuous pattern}) if the following conditions hold:
    \begin{enumerate}
        \item $\wtP$ has precisely $n$ labyrinthine domains.
        \item $\wtP$ is proper.
        \item Any sector of $\wtP$ is not a disk.
    \end{enumerate}
\end{dfn}

Note that for any positive integer $n$, there exists an $n$-continuous pattern.
In the remainder, we call a $2$-continuous (resp. $3$-continuous) pattern a \emph{bicontinuous} (resp. \emph{tricontinuous}) pattern, according to the conventions of soft materials~\cite{Hyde1, Sch}.

The following corollary is a polycontinuous pattern version of Proposition~\ref{prop:decomp2pattern}.
\begin{cor}\label{cor:rel_decomp2poly}
    Let $(H_1, \ldots, H_n; P)$ be a proper handlebody decomposition of $\Torus$, and $\wtP$ the preimage of $P$ under the universal covering map of $\Torus$.
    Suppose that the following two conditions hold:
    \begin{enumerate}[label=\textup{(\arabic*)}]
        \item For each $i$, we have $(\iota_i)_\ast (\fund(H_i)) \cong Z \oplus Z \oplus Z$.
        \item Any sector of $P$ is not a disk.
    \end{enumerate}
    Then, $\wtP$ is a polycontinuous pattern.
    In particular, if $(\iota_i)_\ast (\fund(H_i)) = \fund(\Torus)$ for each $i$, then $\wtP$ is an $n$-continuous pattern.
    Furthermore, if $P$ is simple (resp. proper), then $\wtP$ is also simple (resp. proper).
\end{cor}

\subsection{Colorings of patterns}

In this subsection, we will define colorings of \patts.
Each labyrinthine domain of a \patt is a mathematical model of polymers assembled in one kind of block.
In materials science, one kind of block may form many domains of a \patt in general.
To describe such a situation, we introduce ``colors'' of \patts, of which each color corresponds to one kind of block of polymers.

\begin{dfn}\label{dfn:coloring}
    Let $\wtP$ be a \patt.
    Set $X_n = \{ 1, 2, \ldots, n \}$.
    Let $\widetilde{\mathcal{V}}$ denote the set of all labyrinthine domains of $\wtP$.
    A surjection $\varphi: \widetilde{\mathcal{V}} \to X_n$ is called an \emph{$n$-coloring of $\wtP$} if there exists a frame $\pi$ of $\wtP$ such that the following conditions hold:
    \begin{enumerate}[label=(\arabic*)]
        \item For each covering transformation $t$ of $\pi$ and for any $V \in\widetilde{\mathcal{V}}$, we have $\varphi(V) = \varphi \circ t (V)$. 
        \item Two sides of a local part of each sector have different colors.
            Namely, for each point $p$ in a sector $C$, 
            the two labyrinthine domains, $V$ and $V\upr$, that have non-trivial intersection with $N(p)$ satisfy $\varphi(V) \neq \varphi(V\upr)$.
    \end{enumerate}
    The image $\varphi(V)$ is called \emph{the color of $V$}.
    The frame $\pi$ of $\wtP$ is said to be \emph{compatible with the coloring $\varphi$}.
    A \patt together with a fixed ($n$-)coloring is called an \emph{($n$-)colored} \patt (see Figure~\ref{fig:colored_patterns}(b)).
    We say that two colored \patts, $\wtP$ and $\wt{Q}$, are \emph{equivalent} if there exists an ambient isotopy of $\Rth$ that moves $\wtP$ to $\wt{Q}$, and each pair of corresponding labyrinthine domains has the same color after permuting the colors.
    If a surjection $\varphi$ satisfies only the condition (1), 
    then we call $\varphi$ a \emph{non-effective $n$-coloring}, and $\wtP$ is said to be \emph{non-effectively $n$-colored}.

    Let $\wtP$ be a (possibly non-effectively) $n$-colored \patt with a coloring $\varphi: \widetilde{\mathcal{V}} \to X_n$, where $X_n = \{ 1, 2, \ldots, n \}$, and $\pi$ a frame of $\wtP$ compatible with $\varphi$.
    By Proposition~\ref{prop:pattern2decomp}, the image $\pi(\wtP)$ gives a handlebody decomposition of $\Torus$.
    Denote by $\mathcal{V}$ the set of all handlebodies of the decomposition.
    Then we say that $(\wtP, \pi)$ is \emph{of type $(\mathfrak{g}_1, \ldots, \mathfrak{g}_n)$}, where $\mathfrak{g}_i = [g_{i_1}, \ldots, g_{i_k}]$ is a sequence of the genera of the handlebodies in $\mathcal{V}$ colored by $i \in X_n$.
    (For simplicity, if the length of $\mathfrak{g}_i$ is equal to $1$, then we put $\mathfrak{g}_i = g_{i_1}$.)
\end{dfn}

Note that, as opposed to colorings of graphs on surfaces, for any integers $m$, $n$ with $n \geq 2$ and $m < n$, there is an $n$-colored framed \patt that does not admit $m$-coloring.

\begin{rem}\label{rem:effective}
    If a \patt admits a coloring, then it is necessarily proper.
\end{rem}

In fact, a colored \patt $\wtP$ with its frame $\pi$ compatible with the coloring satisfies the condition that $\pi(\wtP)$ induces a proper handlebody decomposition of $\Torus$ since any two labyrinthine domains sharing a sector have different colors (see Proposition~\ref{prop:pattern2decomp}).
Hence, the following holds.

\begin{cor}\label{cor:effective_pattern2decomp}
    Let $(\wtP, \pi)$ be a framed \patt and let $\mathfrak{g}_i = [g\ssn{i}_{1}, \ldots, g\ssn{i}_{k_i}]$ be a sequence of positive integers for $i \in \{ 1, \ldots, n \}$.
    Suppose that $(\wtP, \pi)$ admits an $n$-coloring and is of type $(\mathfrak{g}_1, \ldots, \mathfrak{g}_n)$.
    Then $\pi(\wtP)$ gives a proper type-$(g\ssn{1}_{1}, \ldots, g\ssn{1}_{k_1}, \ldots, g\ssn{n}_{1}, \ldots, g\ssn{n}_{k_n})$ handlebody decomposition $(H\ssn{1}_{1}, \ldots, H\ssn{1}_{k_1}, \ldots, H\ssn{n}_1, \ldots, H\ssn{n}_{k_n})$ such that $H\ssn{i}_{j_1} \cap H\ssn{i}_{j_2} = \emptyset$ for $j_1, j_2 \in \{ 1, \ldots, k_i \}$.
\end{cor}

The converse of the above corollary is clear by Proposition~\ref{prop:decomp2pattern}.

\begin{cor}\label{cor:decomp2effective_pattern}
    Let $(H\ssn{1}_{1}, \ldots, H\ssn{1}_{k_1}, \ldots, H\ssn{n}_{1}, \ldots, H\ssn{n}_{k_n}; P)$ be a proper type-$(g\ssn{1}_{1},$ $\ldots,$ $g\ssn{1}_{k_1},$ $\ldots,$ $g\ssn{n}_{1},$ $\ldots,$ $g\ssn{n}_{k_n})$ handlebody decomposition of $\Torus$ and let $\pi$ be the universal covering map $\Rth \to \Torus$.
    We assume that $H\ssn{i}_{j_1} \cap H\ssn{i}_{j_2} = \emptyset$ for $j_1, j_2 \in \{ 1, \ldots, k_i \}$.
    We further assume that, for each handlebody $H\ssn{i}_{j}$, the induced homomorphism $(\iota\ssn{i}_j)_\ast : \pi_1(H\ssn{i}_j) \to \pi_1(\Torus)$ is not trivial, where $\iota\ssn{i}_j$ is the inclusion map.
    Then $(\pi^{-1}(P), \pi)$ is a colored \patt of type $(\mathfrak{g}_1, \ldots, \mathfrak{g}_n)$, where $\mathfrak{g}_i = [g\ssn{i}_{1}, \ldots, g\ssn{i}_{k_i}]$.
\end{cor}

\subsection{A sufficient condition for the equivalence of patterns}

Corollaries~\ref{cor:effective_pattern2decomp} and~\ref{cor:decomp2effective_pattern} say there is a nice relationship between colored \patts and proper handlebody decompositions.
This subsection gives a sufficient condition for two colored \patts to be equivalent.

To this end, we first consider adjusting a framed \patt to another frame.
Let $(\wtP, \pi)$ be a framed \patt, and let $\rho$ be a covering map $\Rth \to \Torus$.
Since the two covering maps are equivalent, there exists a self-homeomorphism $f$ of $\Rth$ such that $\pi = \rho \circ f$.
If $f$ is orientation-preserving, we say that $\pi$ and $\rho$ \emph{have the same orientation}.
Otherwise, we say that $\pi$ and $\rho$ \emph{have different orientations}.

If $\pi$ and $\rho$ have the same orientation, then $\wtP$ is isotopic to $f(\wtP)$, and $(f(\wtP), \rho)$ is a framed \patt.
Consider the case that $\pi$ and $\rho$ have different orientations.
Let $r$ be an orientation-reversing self-homeomorphism of $\Torus$.
Then, $\pi\upr := r \circ \pi$ is also a covering map $\Rth \to \Torus$.
Hence, there exists an orientation-preserving homeomorphism $g: \Rth \to \Rth$ such that $\pi\upr = \rho \circ g$.
So, $\wtP$ is isotopic to $g(\wtP)$, and we have $r(\pi(\wtP)) = \rho(g(\wtP))$.
Furthermore, $(g(\wtP), \rho)$ is a framed \patt because each covering transformation of $\pi\upr$ is also that of $\pi$.

To summarize, we obtain the following lemma.

\begin{lem}\label{lem:adjast_frame}
    Let $(\wtP, \pi)$ be a framed \patt, and let $\rho$ be a covering map $\Rth \to \Torus$.
    Then, there exists a \patt $\wt{Q}$ such that the following three conditions hold:
    \begin{enumerate}[label=\textup{(\arabic*)}]
        \item The covering map $\rho$ is a frame of $\wt{Q}$.
        \item The pattern $\wtP$ is isotopic to $\wt{Q}$.
        \item Either $\pi(\wtP) = \rho(\wt{Q})$ or there exists an orientation-reversing self-homeomorphism $r$ of $\Torus$ with $r(\pi(\wtP)) = \rho(\wt{Q})$.
    \end{enumerate}
    In particular, if $\wtP$ is (non-effectively) $n$-colored, then so is $\wt{Q}$.
    Furthermore, $\wtP$ and $\wt{Q}$ have the same type.
\end{lem}

By Lemma~\ref{lem:adjast_frame}, we can assume that any two \patts have the same frame.
Proposition~\ref{prop:pattern2decomp} says the two framed \patts induce two handlebody decompositions of $\Torus$, respectively.
If the two decompositions are isotopic, then the two patterns are also isotopic.
In fact, we can say more as follows.

\begin{lemma}\label{lem:change_lattice}
    Let $(\wtP, \pi)$ and $(\wt{Q}, \pi)$ be framed \patts.
    Suppose that there exists an orientation-preserving self-homeomorphism $f$ of $\Torus$ that maps $\pi(\wtP)$ to $\pi(\wt{Q})$.
    Then, $\wtP$ is isotopic to $\wt{Q}$.
\end{lemma}

\begin{proof}
    By the assumption, we have an orientation-preserving self-homeomorphism $f$ of $\Torus$ that maps $\pi(\wtP)$ to $\pi(\wt{Q})$.
    Let $\wt{f}$ be the unique lift of $f \circ \pi$.
    Then $\wt{f}$ is a homeomorphism of $\R^3$, and we have $\wt{f}(\wtP) = \wt{Q}$. 
    Therefore, $\wtP$ and $\wt{Q}$ are isotopic.
\end{proof}

By the above lemmas, we have the following proposition.

\begin{prop}\label{prop:frame}
    Let $(\wtP, \pi)$ and $(\wt{Q}, \rho)$ be $n$-colored framed \patts.
    We assume that $\pi(\wtP)$ and $\rho(\wt{Q})$ are homeomorphic under an orientation-preserving or orientation-reversing self-homeomorphism $f$ of $\Torus$ according to whether the covering maps $\pi$ and $\rho$ have the same orientation or different orientations.
    Suppose that any two corresponding handlebodies under $f$ are the images of labyrinthine domains with the same color (after permuting the colors).
    Then $(\wtP, \pi)$ and $(\wt{Q}, \rho)$ are equivalent.
\end{prop}

\section{Stabilizations on \patts}\label{sec:patts}

In this section, we will discuss (de)stabilizations of \patts and introduce some examples.

\subsection{Stabilization theorem for \patts}

Section~\ref{sec:stab} showed an analogue of Reidemeister-Singer's theorem for handlebody decompositions (Theorem~\ref{thm:stabeq}).
This subsection shows a \patt version of the theorem.
To do so, we define (de)stabilizations of \patts.
First, we will use an example to explain how to define it.

In Example~\ref{ex:honeycomb_pattern}, we introduced the hexagonal honeycomb pattern that is a $3$-colored framed \patt of type $(1, 1, 1)$.
We consider a type-$1$ stabilization on the pattern.
Let $\wt{\alpha}$ be a properly embedded arc in a sector of the pattern (see Figure~\ref{fig:stabilization_pattern}(a)).
We assume that $\wt\alpha$ is \emph{lifted by $\pi$}, i.e., the restriction of $\pi$ to $\wt\alpha$ is injective, where $\pi$ is the frame of the pattern.
By Corollary~\ref{cor:effective_pattern2decomp}, $\pi(\wtP)$ gives a simple proper handlebody decomposition.
As noted in Example~\ref{ex:honeycomb_pattern}, $P := \pi(\wtP)$ is a simple proper type-$(1, 1, 1)$ handlebody decomposition (see Figure~\ref{fig:nxx}(d)).
Since $\wt\alpha$ is lifted and contained in a sector, the image $\alpha := \pi(\wt\alpha)$ is a properly embedded arc in a sector of $P$. 
Furthermore, $\wt\alpha$ connects two labyrinthine domains mapped to the same handlebody by $\pi$.
So, the arc $\alpha$ connects the same handlebody.
Thus, we can perform a type-$1$ stabilization along $\alpha$.
Hence, a type-$(1, 1, 2)$ handlebody decomposition $P\upr$ is obtained by performing a type-$1$ stabilization along $\alpha$. 
By Corollary~\ref{cor:decomp2effective_pattern}, the preimage of $P\upr$ under $\pi$ gives a $3$-colored framed \patt $\wtP\upr$ of type $(1, 1, 2)$ (see Figure~\ref{fig:stabilization_pattern}(b)).
Hence, we obtain the new \patt $\wtP\upr$ from $\wtP$.
We will call such an operation a \emph{type-$1$ stabilization for \patts}.

\begin{figure}[htbp]
    \centering
    \begin{minipage}{0.48\textwidth}
        \centering
        \includegraphics[width=0.80\textwidth]{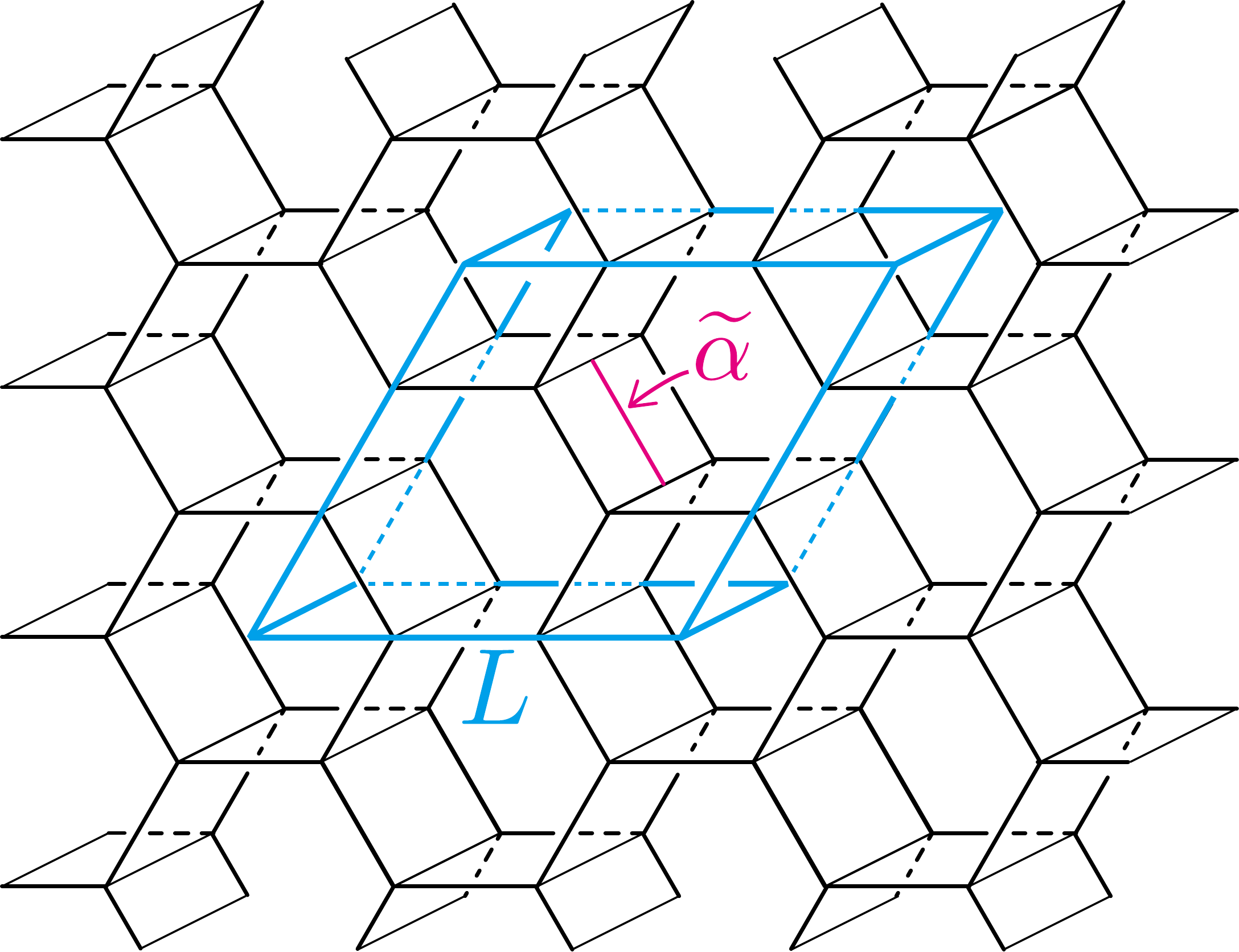}\\
        (a)
    \end{minipage}
    \begin{minipage}{0.48\textwidth}
        \centering
        \includegraphics[width=0.80\textwidth]{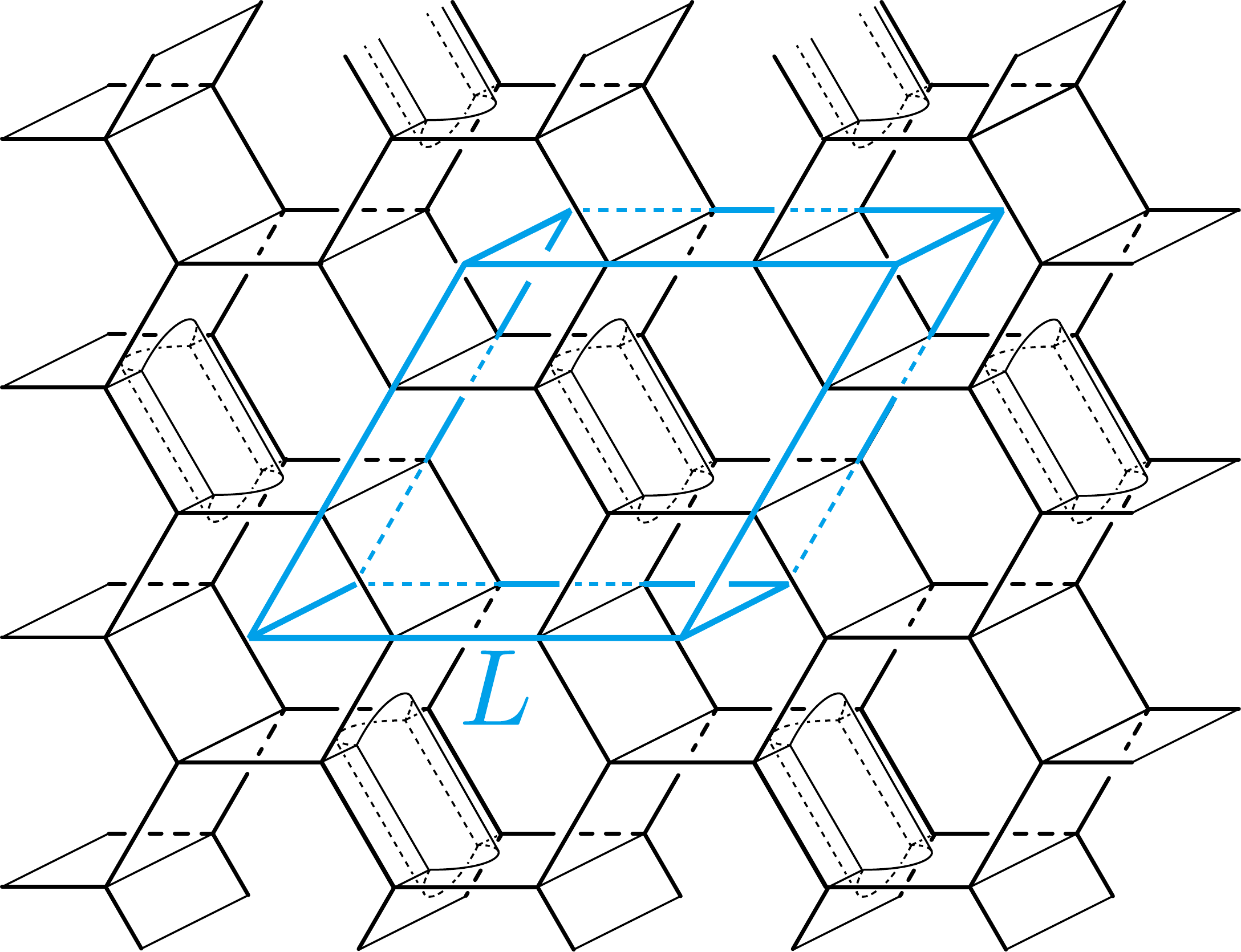}\\
        (b)
    \end{minipage}
    \caption{(a) The hexagonal honeycomb pattern introduced in Example~\ref{ex:honeycomb_pattern}. The parallelepiped $L$ illustrates a fundamental domain of the pattern. (b) The \patt obtained by performing a type-$1$ stabilization on the honeycomb pattern along the properly embedded lifted arc $\wt{\alpha}$.}
    \label{fig:stabilization_pattern}
\end{figure}

Based on the above example, we give a strict definition of stabilizations for patterns as follows.
\begin{dfn}\label{dfn:pattern_stabilization}
    Let $(\wtP, \pi)$ be a simple, $n$-colored, framed \patt of type $(\mathfrak{g}_1, \ldots, \mathfrak{g}_n)$, where $\mathfrak{g}_i$ is a sequence of positive integers $[g\ssn{i}_{1}, \ldots, g\ssn{i}_{m_i}]$ for $1 \leq i \leq n$.
    Put $P = \pi(\wtP)$.
    By Corollary~\ref{cor:effective_pattern2decomp}, $P$ gives a simple proper handlebody decomposition $\mathcal{H} = (H\ssn{1}_1, \ldots, H\ssn{1}_{m_1}, H\ssn{2}_{1}, \ldots, H\ssn{2}_{m_2}, \ldots, H\ssn{n}_1, \ldots, H\ssn{n}_{m_n}; P)$ of $\Torus$ such that $H\ssn{i}_j$ is a genus-$g\ssn{i}_j$ handlebody colored by $i$.
    Let $U$, $U\upr$, $V$, $V\upr$, and $W$ be labyrinthine domains.
    We assume that, for each pair of the labyrinthine domains except for $(U, U\upr)$, $(V, V\upr)$, and $(U, W)$, the two domains are different and share a sector.
    (There is a possibility that $U = U\upr$ or $V = V\upr$.)
    We further assume that $\pi(U) = \pi(U\upr) = H\ssn{i_U}_{j_U}$, $\pi(V) = \pi(V\upr) = H\ssn{i_V}_{j_V}$, and $\pi(W) = H\ssn{i_W}_{j_W}$.
    Here, $H\ssn{i_U}_{j_U}$, $H\ssn{i_V}_{j_V}$, and $H\ssn{i_W}_{j_W}$ are distinct handlebodies, and $i_U \neq i_V$, $i_V \neq i_W$, and $i_W \neq i_U$.

    Depending on the type of stabilization, we take an arc $\wt\alpha$ as follows.
    \begin{enumerate}[label=(type-\arabic*),align=parleft,leftmargin=*,itemindent=35.01pt,itemsep=1mm,start=0]
        \item The arc $\wt\alpha$ is a properly embedded lifted arc in $V$ that connects $U$ and $U\upr$.
            We assume that a lifted disk in $V$ contains $\wt\alpha$ as a part of its boundary, and the other part is contained in $\partial V$.
        \item The arc $\wt\alpha$ is a properly embedded lifted arc in a sector of $\wtP$ that connects $V$ and $V\upr$.
        \item The arc $\wt\alpha$ is a properly embedded lifted arc in $V$ that connects $U$ and $W$.
            We assume that a lifted disk in $V$ contains $\wt\alpha$ as a part of its boundary, and the other part is contained in $\partial V$.
    \end{enumerate}
    Then, we can obtain a new handlebody decomposition $\mathcal{H}\upr$ performed by a suitable stabilization on $\mathcal{H}$ along $\pi(\wt\alpha)$.
    We can see by Corollary~\ref{cor:decomp2effective_pattern} that the preimage of the partition of $\mathcal{H}\upr$ is a simple, colored, framed \patt of type $(\mathfrak{g}_1\upr, \ldots, \mathfrak{g}_n\upr)$.
    Here, each $\mathfrak{g}_i\upr$ is equal to $\mathfrak{g}_i$ except for the following sequences:
    \vspace{1.5mm}
    
    \noindent
    (type-$0$)\;
    $\mathfrak{g}_{i_U}\upr = [g\ssn{i_U}_1, \ldots, g\ssn{i_U}_{j_U} + 1, \ldots, g\ssn{i_U}_{m_{i_U}}]$,\, $\mathfrak{g}_{i_V}\upr = [g\ssn{i_V}_1, \ldots, g\ssn{i_V}_{j_V} + 1, \ldots, g\ssn{i_V}_{m_{i_V}}]$.

    \noindent
    (type-$1$ and type-$2$)\;
    $\mathfrak{g}_{i_V}\upr = [g\ssn{i_V}_1, \ldots, g\ssn{i_V}_{j_V} + 1, \ldots, g\ssn{i_V}_{m_{i_V}}]$.
    \vspace{2.5mm}

    \noindent
    We call this operation a \emph{type-$k$ stabilization (along $\wt{\alpha}$ with respect to $\pi$)} and its inverse operation a \emph{type-$k$ destabilization} for each $k$.
    \ifRSPA
    In Supplementary Information, we discuss sufficient conditions for performing a destabilization.
    \else
    In Appendix~\ref{app:destab}, we discuss sufficient conditions for performing a destabilization.
    \fi
\end{dfn}

Note that the result of a (de)stabilization of a polycontinuous pattern is not necessarily a polycontinuous pattern.
Further note that in a type-$2$ stabilization along an arc for \patts, even if the arc connects different labyrinthine domains, we cannot perform the operation if they are the same color.

Definition~\ref{dfn:moves} introduced some operations for handlebody decompositions called \emph{moves}.
We next consider a \patt version of them.
Of course, we can perform the original operations on simple colored \patts, but they generally lose periodicity after performing them.
Thus, we give adapted ``moves'' to \patts in a similar way to the stabilizations.

\begin{dfn}\label{dfn:pattern_moves}
    Let $(\wtP, \pi)$ be a simple, $n$-colored, framed \patt.
    Take a properly embedded lifted arc $\wt{\alpha}$ in a sector (resp. an edge $\wt{\alpha}$ of the singular graph of $\wtP$) so that it connects labyrinthine domains $V$ and $V\upr$ of different colors.
    By Corollary~\ref{cor:effective_pattern2decomp}, $P := \pi(\wtP)$ gives a simple proper handlebody decomposition $\mathcal{H}$.
    Then, we can obtain a new handlebody decomposition $\mathcal{H}\upr$ performed by a 0-2 (resp. 2-3) move on $P$ along $\pi(\wt{\alpha})$.
    By Corollary~\ref{cor:decomp2effective_pattern} the preimage of the partition of $\mathcal{H}\upr$ is a simple, colored, framed \patt.
    We call such an operation a \emph{0-2 (resp. 2-3) move (along $\wt{\alpha}$ with respect to $\pi$)} and its inverse operation a \emph{2-0 (resp. 3-2) move}.%
%
%
\end{dfn}

Note that, similar to type-$2$ stabilizations of \patts, even if we can perform a move on a handlebody decomposition corresponding to a pattern, it does not necessarily mean that we can perform the corresponding move on the pattern.

An analogue of Reidemeister-Singer's theorem for \patts is as follows.

\begin{cor}\label{cor:pattern_stabilization_theorem}
    Let $(\wtP, \pi)$ and $(\wt{Q}, \rho)$ be simple, $n$-colored, framed \patts of type $(g_1, \ldots, g_n)$ and $(g\upr_1, \ldots, g\upr_n)$, respectively, where $g_i$ and $g\upr_i$ are positive integers.
    Then $\wtP$ and $\wt{Q}$ are equivalent after applying $0$-$2$, $2$-$0$, and $2$-$3$ moves, and type-$0$ and type-$1$ stabilizations finitely many times.
\end{cor}

\begin{proof}
    We assume that $\pi$ and $\rho$ have the same orientation.
    The proof of the other case is similar.
    By Corollary~\ref{cor:effective_pattern2decomp}, the images $P := \pi(\wtP)$ and $Q := \rho(\wt{Q})$ give type-$(g_1, \ldots, g_n)$ and type-$(g_1\upr, \ldots, g_n\upr)$ simple proper handlebody decompositions of $\Torus$, respectively. 
    Hence, by Theorem~\ref{thm:stabeq}, there exists a simple proper handlebody decomposition such that $\pi(\wt{P})$ and $\rho(\wt{Q})$ are isotopic to the partition $R$ of the decomposition after applying 0-2, 2-0, and 2-3 moves, and type-$0$ and type-$1$ stabilizations to them finitely many times.
    By Corollary~\ref{cor:decomp2effective_pattern}, $\wt{R} := \pi^{-1}(R)$ is a simple $n$-colored \patt.
    Therefore, by Proposition~\ref{prop:frame}, each of $\wt{P}$ and $\wt{Q}$ is equivalent to $\wt{R}$ after applying 0-2, 2-0, and 2-3 moves, and type-$0$ and type-$1$ stabilizations finitely many times.
\end{proof}

In the above corollary, we assume that, for each color, all labyrinthine domains colored by it are mapped to the same handlebody because moves performed in the proof of Theorem~\ref{thm:stabeq} generally do not preserve the coloring.
On the concept of colorings, we can regard single-colored domains as composed of the same kind of blocks, so connecting these parts is a natural idea.

\begin{dfn}\label{dfn:domain_connection}
    Let $(\wtP, \pi)$ be a simple, $n$-colored, framed \patt.
    Take a properly embedded lifted arc $\wt{\alpha}$ in a sector so that it connects labyrinthine domains, $V$ and $V\upr$, of the same color.
    We assume that $H := \pi(V)$ and $H\upr := \pi(V\upr)$ are different handlebodies of the simple proper handlebody decomposition induced by $P := \pi(\wtP)$.
    By performing a 0-2 move on $P$ along $\pi(\wt{\alpha})$,
    the modified $H$ and $H\upr$ are intersected, and by Corollary~\ref{cor:effective_pattern2decomp}, their intersection consists of only the disk created by the operation.
    So, $H\dupr := H \cup H\upr$ is also a handlebody.
    Hence, we have a new handlebody decomposition by replacing $H$ and $H\upr$ with $H\dupr$.
    By Corollary~\ref{cor:decomp2effective_pattern}, the decomposition induces a new simple, colored, framed \patt $(\wtP\upr, \pi)$.
    We call such an operation a \emph{domain-connection (along $\wt{\alpha}$ with respect to $\pi$)} and its inverse operation a \emph{domain-disconnection}.
\end{dfn}

\begin{rem}\label{rem:domain_connection}
    We can obtain the type of $(\wtP\upr, \pi)$ in Definition~\ref{dfn:domain_connection} as follows.
    Let $\mathfrak{g}_i$ be the sequence of positive integers in the type of $(\wtP, \pi)$ corresponding to the color of the labyrinthine domains $V$ and $V\upr$.
    We remove the integers corresponding to $V$ and $V\upr$ from $\mathfrak{g}_i$ and append their sum.
    Then, we denote a new sequence by $\mathfrak{g}_i\upr$.
    By replacing $\mathfrak{g}_i$ with $\mathfrak{g}\upr_i$, we obtain the type of $(\wtP\upr, \pi)$.
\end{rem}

By applying the following to a colored \patt, it satisfies the assumption of Corollary~\ref{cor:pattern_stabilization_theorem}.

\begin{lem}\label{lem:connect_labyrinth}
    Let $(\wtP, \pi)$ be a simple, $n$-colored, framed \patt of type $(\mathfrak{g}_1, \ldots, \mathfrak{g}_n)$, where $\mathfrak{g}_i$ is a sequence of positive integers $[g\ssn{i}_{1}, \ldots, g\ssn{i}_{m_i}]$ for $1 \leq i \leq n$.
    Set $\widehat{g}_i = \sum_{k = 1}^{m_i} g\ssn{i}_k$.
    Then, we have a simple, $n$-colored, framed \patt of type $(\widehat{g}_1, \ldots, \widehat{g}_n)$ by applying 0-2 moves and domain-connections with respect to $\pi$ finitely many times to $\wtP$.
\end{lem}

\begin{proof}
    Since $\wt{P}$ is connected, for each color $i$, there exist labyrinthine domains, $V$ and $V\upr$, with color $i$ and an embedded lifted arc $\wt{\delta}$ joining $V$ and $V\upr$ in $\wtP$ such that the following hold:
    \begin{enumerate}[label=(\arabic*)]
        \item The images $\pi(V)$ and $\pi(V\upr)$ are distinct handlebodies.
        \item The arc $\wt\delta$ does not cross any labyrinthine domains with color $i$.
        \item The arc $\wt\delta$ intersects the singular graph of $\wtP$ transversely.
    \end{enumerate}
    Then, by cutting $\wt\delta$ at its intersection with the singular graph, we obtain the sequence of sub-arcs $\wt\delta_1, \ldots, \wt\delta_k$.
    Thus, we can perform 0-2 moves along $\wt\delta_1$, $\ldots$, $\wt\delta_{k-1}$, and we can finally apply domain-connection along $\wt\delta_k$.
    By repeating the above process, all labyrinthine domains with color $i$ are joined.
    Then, the type corresponding to color $i$ is $\widehat{g}_i$ by Remark~\ref{rem:domain_connection}.
    Therefore, we have a \patt of type $(\widehat{g}_1, \ldots, \widehat{g}_n)$.
\end{proof}

\subsection{Microphase separation of a block copolymer melt}\label{subsec:polymer}

One motivation for this research comes from materials science.
We are interested in the characterization of \emph{polycontinuous patterns} that appear as microphase separation of a block copolymer melt \cite{Hyde2,Hyde1}.


 In this subsection, we discuss block copolymers and phase separation of a block copolymer melt. 
 One reference of this subject is \cite{Fredrickson}. 
 A {\em polymer} is a molecule of high molecular weight created by chemically coupling large numbers of small reactive molecules, called {\em monomers}. 
 If a polymer is made of one type of monomers, it is called a {\em homopolymer}. 
 A polymer containing two or more chemically distinct monomers is referred to as a {\em copolymer}. 
 A {\em block copolymer} is an important type of copolymers, in which monomers of a given type form polymerized sequences called {\em blocks}. 
 If a block copolymer contains two (respectively three) blocks, it is called a {\em diblock} (resp. {\em triblock}) copolymer. 
 If a linear diblock copolymer is made of blocks of monomers A and B, it is called an {\em AB diblock copolymer}. 
 An {\em ABA triblock copolymer} is a linear triblock copolymer consisting of a sequence of a block of monomer A, a block of monomer B, and a block of monomer A. 
 See Figure~\ref{fig:blockcopolymer}(a).
 SBS (styrene-butadiene-styrene) triblock and SIS (styrene-isoprene-styrene) triblock copolymers are examples of linear triblock copolymers. 
 Polymers with more complex architecture have been synthesized. 
 For example, a {\em star polymer} has one branched point connecting several linear polymers. 
 An {\em ABC triblock-arm star-shaped molecule} ({\em 3-star polymer}) as in Figure~\ref{fig:blockcopolymer}(a) is an example of triblock copolymer with a star architecture, where A, B, and C blocks are mutually immiscible.

 A {\em block copolymer melt} is a solvent-free viscoelastic liquid composed of block copolymers.
 Due to the chemical distinction of monomers, we can observe phase separation in a block copolymer melt. 
 A {\em domain} of phase separation consists of monomers of one type.
 {\em Microphase separation} of a block copolymer melt is phase separation with domains of the mesoscopic size scale. 
 Sphere, cylinder, bicontinuous, and lamellar structures appear as microphase separation of AB diblock or ABA triblock copolymers \cite{Fredrickson, Matsen}.

An example of \emph{bicontinuous patterns} is the Gyroid surface.
In materials science, in the bicontinuous pattern of an AB diblock copolymer melt, the domain of the A monomer is the neighborhood of the partition surface, and that of the B monomer forms two labyrinths (see Figure~\ref{fig:blockcopolymer}(b)).

A {\em tricontinuous pattern} is a mathematical model of microphase separation of an ABC star-shaped block copolymer melt.
The branch line of a tricontinuous pattern consists of the connection points of the A, B, and C blocks in the block copolymers \cite{Hyde1}.
A sector of a tricontinuous pattern is the interface of two domains.

Next, we discuss a mathematical model of microphase separation with four phases.
Let A, B, C, D be four chemically distinct monomers.
We consider the polycontinuous pattern of melts of 4 types of 3-star block copolymers of ABC, ABD, ACD, and BCD.
In this case, four different branched lines appear.
The interface of domains and these branched lines form a simple polyhedron.
The vertex of the simple polyhedron of the polycontinuous pattern corresponds to the point where four domains A, B, C, and D meet.
The four edges corresponding to the connecting points of the ABC, ABD, ACD, and BCD triblock star polymers are placed around a vertex.
Also, ABCD 4-star polymers are synthesized~\cite{Iatrou,Mavroudis}, and their morphologies have been discussed in~\cite{Gupta,Wang}.
The joining point of 4 blocks of the block copolymer corresponds to a vertex of the simple polycontinuous pattern.

We want to analyze the property of materials with this structure via a topological study of these polycontinuous patterns.
We hope the characterization and the classification of polycontinuous patterns will lead to the design of polymeric materials with the desired properties.

As an application of Corollary~\ref{cor:pattern_stabilization_theorem}, we can discuss the relation between two microphase-separated structures of the same type.
Here we discuss the polymer science implications of stabilization and destabilization operation of patterns.

\begin{obs}
The type-$0$ destabilization for a bicontinuous pattern can be considered as the model of the canceling of an unstable local 1-handle structure of the pattern of the microphase separation.
\end{obs}

\begin{obs}
The type-$1$ destabilization (resp. stabilization) for a polycontinuous pattern can be considered as the model of the separation (resp. amalgamation) of the domains during the uniaxial elongation of polymeric materials.
\end{obs}

\begin{figure}[htbp]
    \centering
    \begin{minipage}[b]{0.55\textwidth}
        \centering
        \includegraphics[width=0.80\textwidth]{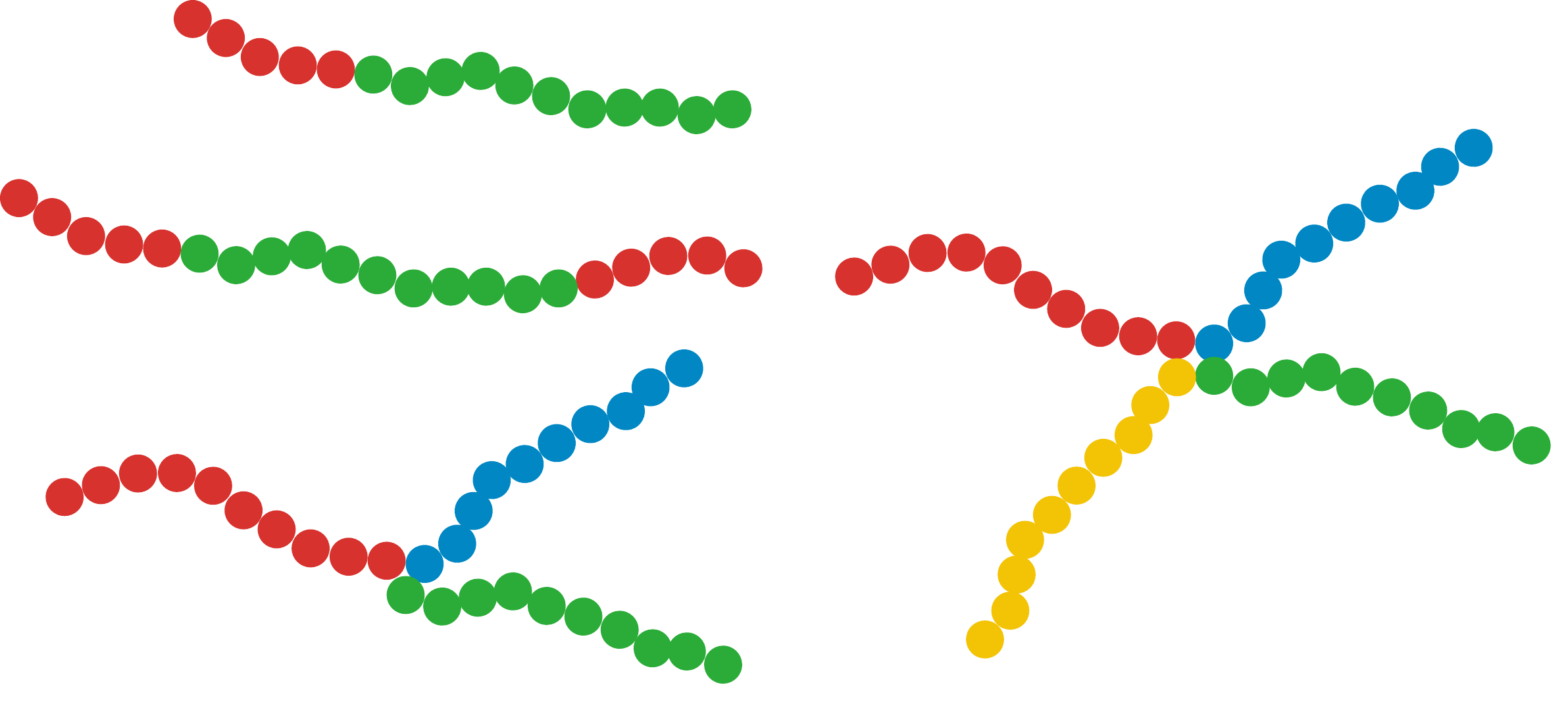}\\
        (a)
    \end{minipage}
    \begin{minipage}[b]{0.40\textwidth}
        \centering
        \includegraphics[width=0.60\textwidth]{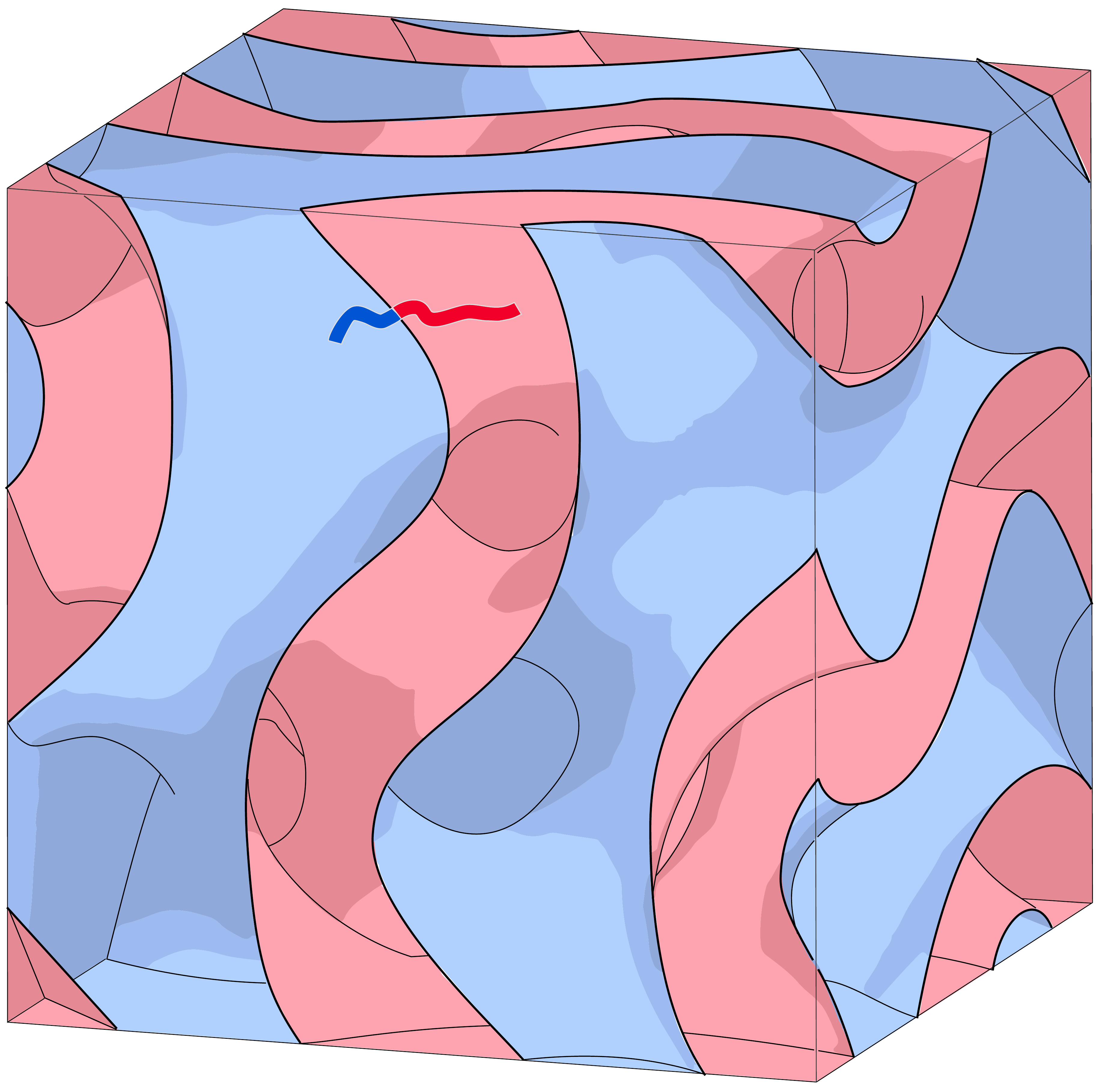}\\
        (b)
    \end{minipage}
    \caption{(a) A red dot indicates the monomer A, blue monomer B, green monomer C, and yellow monomer D. Left: AB diblock copolymer, ABA triblock copolymer, and ABC star-shaped block copolymer. Right: ABCD star-shaped block copolymer. (b) A double gyroid that is a famous bicontinuous pattern.}
    \label{fig:blockcopolymer}
\end{figure}

\subsection{Example: a 3srs pattern}

A 3srs pattern is an example of a tricontinuous pattern.
In this subsection, we will show the pattern can be destabilized to the hexagonal honeycomb pattern.

First, we introduce a 3srs net.
An \emph{srs} net is a $3$-periodic ``minimal'' net in $\R^3$ (see~\cite{Beukemann} and Figure~\ref{fig:srs}(a)).
Figure~\ref{fig:srs}(a) illustrates an srs net with a cubical fundamental domain, of which the length of each edge is $8$.
The net is an infinite trivalent graph, and the \emph{space group} of it is $P4_332$.
Note that a $2\pi/3$ rotation around the cube diagonal (shown in Figure~\ref{fig:srs}(a)) generates an action of order $3$ and preserves the cube.
A \emph{$3$srs} net is the union of the images of the srs net under the action (see Figure~\ref{fig:srs}(b)).

Figure~\ref{fig:triconti-3srs}(a) illustrates a branched surface in $\R^3$ with a cubical fundamental domain.
The branched surface is the union of precisely three surfaces with the boundary (see Figure~\ref{fig:triconti-3srs}(b--d)).
It is clear that the branched surface is a simple $3$-colored tricontinuous pattern, and each component of the $3$srs net is a labyrinthine net of the pattern.
We call the tricontinuous pattern the \emph{$3$srs pattern}.
The $3$srs pattern is of type $(5, 5, 5)$ as illustrated in Figures~\ref{fig:srs} and \ref{fig:triconti-3srs}.

\begin{figure}[htbp]
    \centering
    \begin{minipage}[b]{0.38\textwidth}
        \centering
        \includegraphics[width=0.95\textwidth]{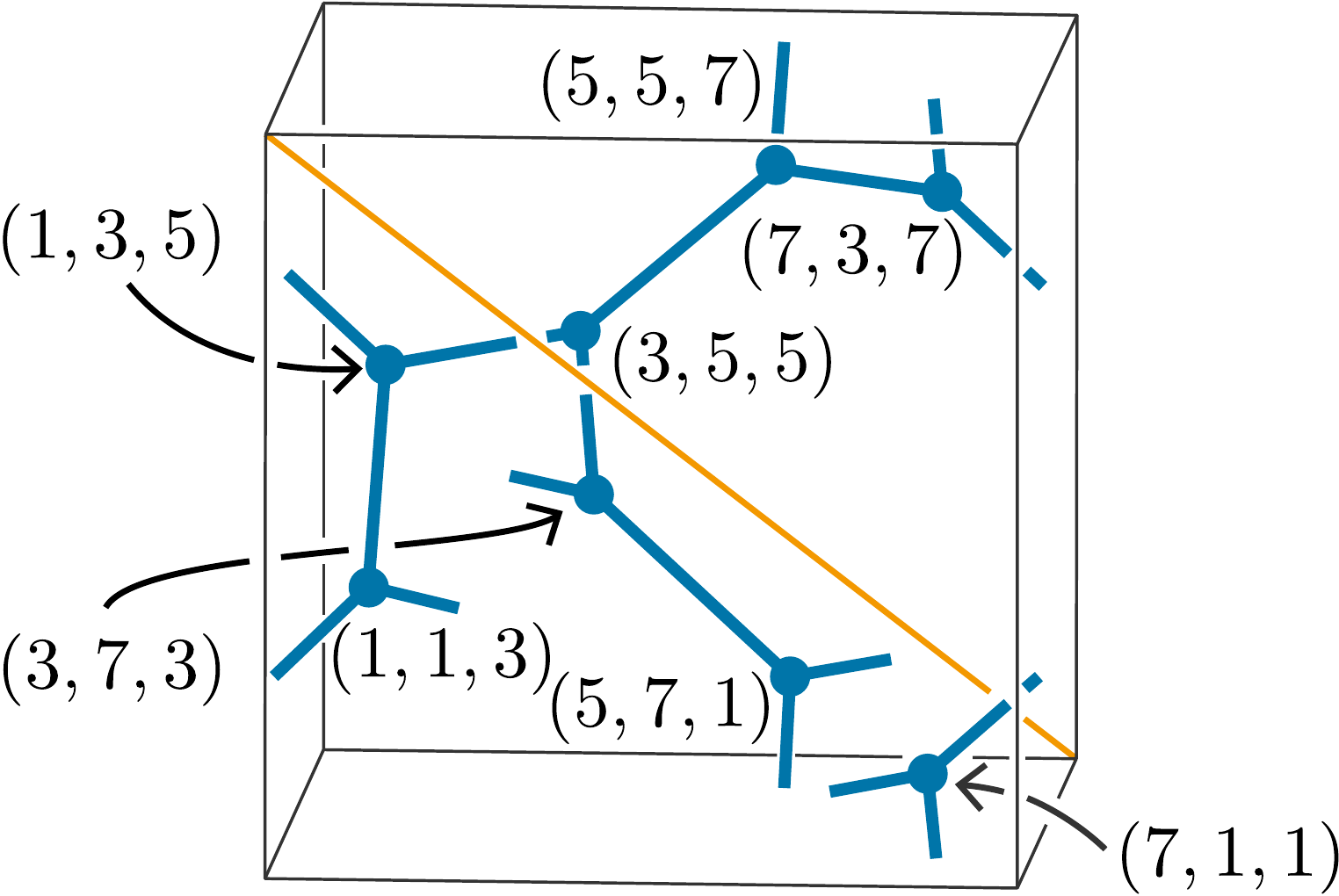}\\
        (a)
    \end{minipage}
    \begin{minipage}[b]{0.38\textwidth}
        \centering
        \includegraphics[width=0.65\textwidth]{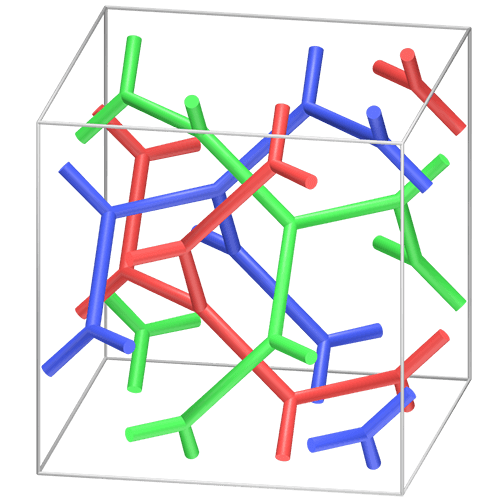}\\
        (b)
    \end{minipage}
    \caption{(a) A srs net. The orange line $l$ shows the line passing through the points $(0, 0, 8)$ and $(8, 8, 0)$. Note that this net is topologically the same as the \emph{srs-b} net (see~\cite{RCSR}). (b) A 3srs net. The $2\pi/3$ rotation around $l$ preserves the net.}
    \label{fig:srs}
\end{figure}

\begin{figure}[tbp]
    \centering
    \begin{minipage}{0.24\textwidth}
        \centering
        \includegraphics[width=1.00\textwidth]{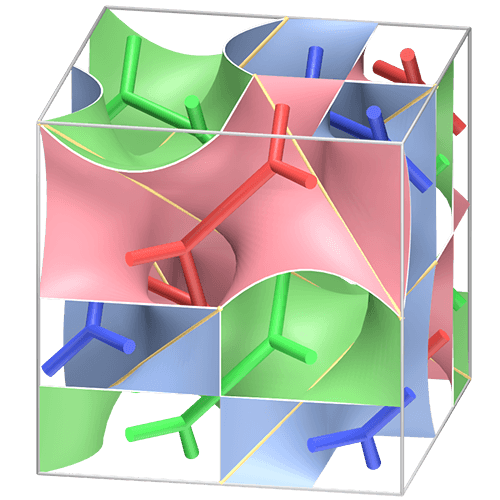}\\
        (a)
    \end{minipage}
    \begin{minipage}{0.24\textwidth}
        \centering
        \includegraphics[width=1.00\textwidth]{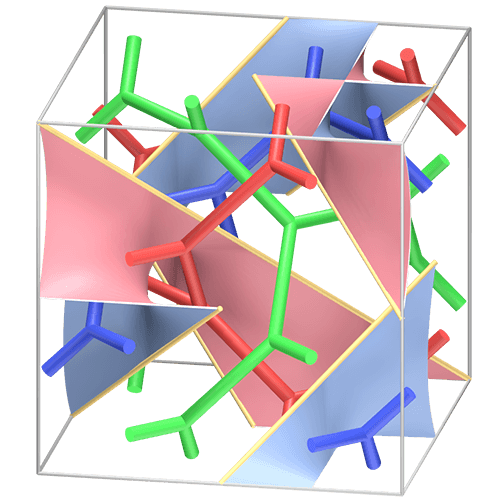}\\
        (b)
    \end{minipage}
    \begin{minipage}{0.24\textwidth}
        \centering
        \includegraphics[width=1.00\textwidth]{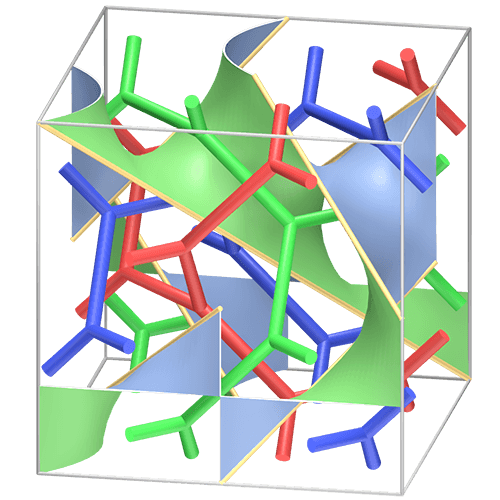}\\
        (c)
    \end{minipage}
    \begin{minipage}{0.24\textwidth}
        \centering
        \includegraphics[width=1.00\textwidth]{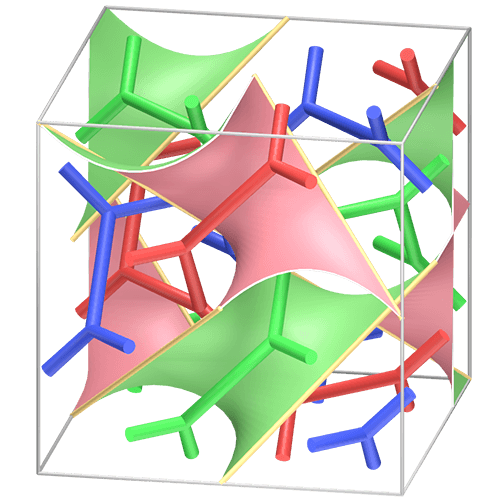}\\
        (d)
    \end{minipage}
    \caption{(a) A tricontinuous pattern corresponding to the 3srs net with a cubical fundamental domain. (b--d) Surfaces with boundary, each of which is shared by exactly two labyrinthine domains.}
    \label{fig:triconti-3srs}
\end{figure}

\begin{figure}[tbp]
    \centering
    \begin{minipage}[b]{0.3\linewidth}
        \centering
        \includegraphics[width=.95\textwidth]{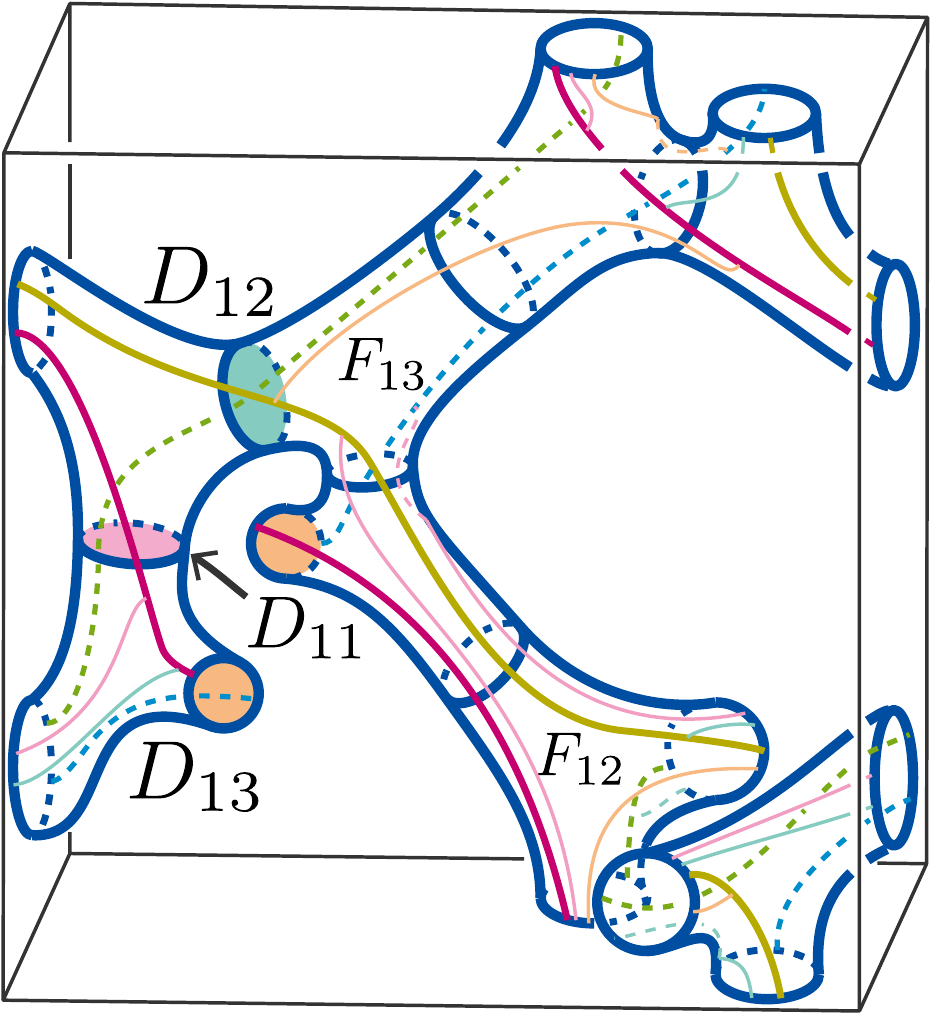}
        \subcaption{$H_1$}
    \end{minipage}
    \begin{minipage}[b]{0.3\linewidth}
        \centering
        \includegraphics[width=.95\textwidth]{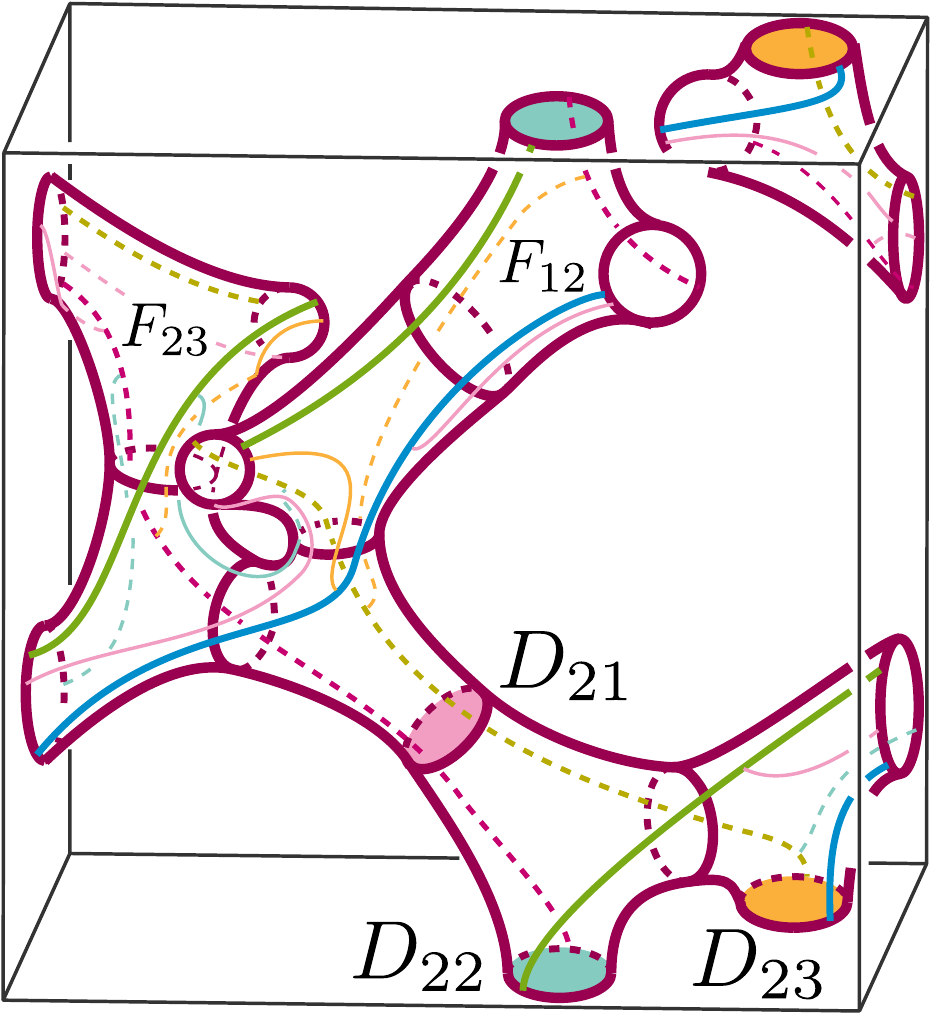}
        \subcaption{$H_2$}
    \end{minipage}
    \begin{minipage}[b]{0.3\linewidth}
        \centering
        \includegraphics[width=.95\textwidth]{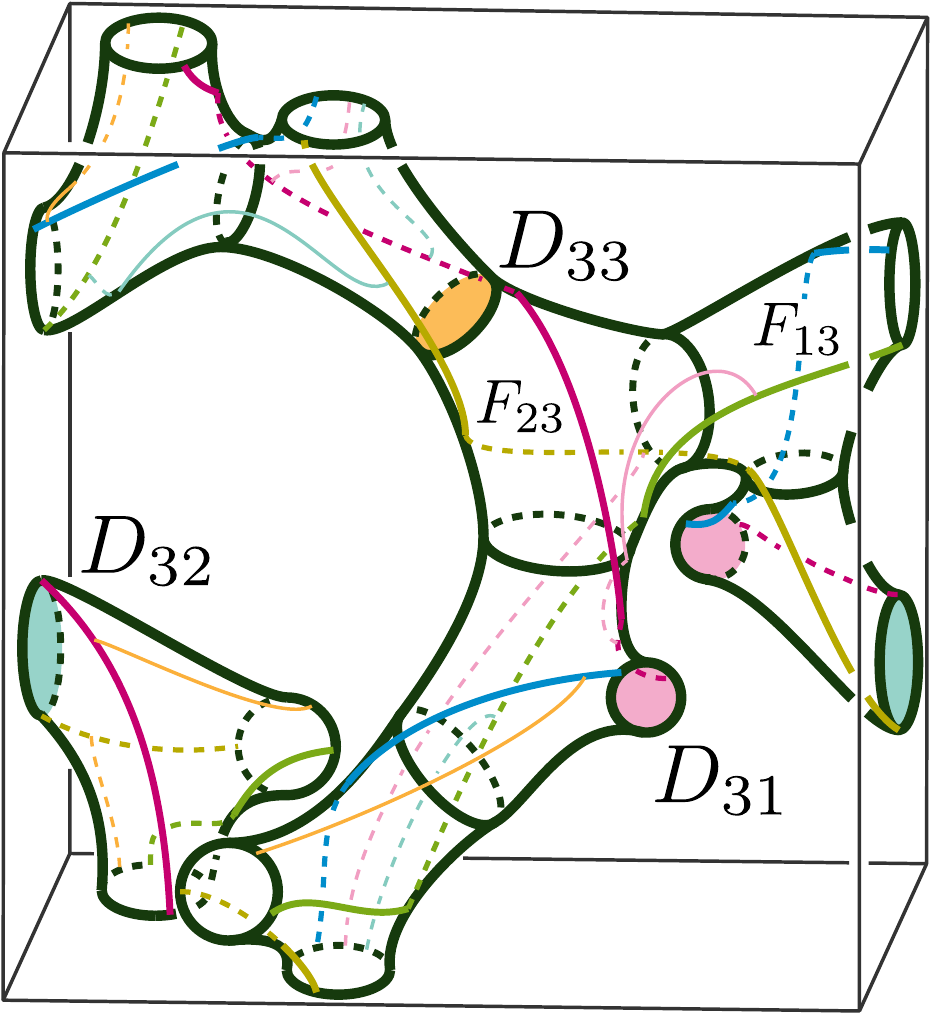}
        \subcaption{$H_3$}
    \end{minipage}
    \caption{A handlebody decomposition of $\Torus$ induced by the $3$srs pattern. The ``cores'' of handlebodies are the quotient of the 3srs net.
    The bold curves on the boundaries of handlebodies are the singular graph.
    Each $F_{ij}$ denotes a surface as in Remark~\ref{rem:cl_Hdecomp}.}
    \label{fig:3srs_decomp}
\end{figure}

\begin{thm}
    The $3$srs pattern can be destabilized to the hexagonal honeycomb pattern, 
    i.e., the $3$srs pattern can be obtained from the hexagonal honeycomb pattern by a finite sequence of type-$1$ stabilizations.
\end{thm}

\begin{proof}
    Let $\wtP$ be the $3$srs pattern, and $\pi$ its frame obtained from a cubic fundamental domain as shown in Figure~\ref{fig:triconti-3srs}.
    Put $P = \pi(\wtP)$.
    Figure~\ref{fig:3srs_decomp} shows a simple proper type-$(5, 5, 5)$ handlebody decomposition $(H_1, H_2, H_3; P)$ of $\Torus$ induced by $\wtP$.
    We denote by $F_{12}$, $F_{13}$, and $F_{23}$ surfaces with boundary as in Remark~\ref{rem:cl_Hdecomp}.
    By Definition~\ref{dfn:pattern_stabilization}, if we destabilize the decomposition to the hexagonal honeycomb decomposition by performing a finite sequence of type-$1$ destabilizations, then we can also destabilize $\wtP$ to the hexagonal honeycomb pattern by corresponding destabilizations.

    First, for each $i$, we take three meridian disks $D_{i1}$, $D_{i2}$, and $D_{i3}$ of the handlebody $H_i$ as shown in Figure~\ref{fig:3srs_decomp}(a--c).
    Each disk intersects the singular graph of $P$ transversely exactly two points.
    Furthermore, any two different disks are disjoint.
    Hence, we can perform type-$1$ destabilizations along them.
    By this operation, we obtain a type-$(2,2,2)$ handlebody decomposition of $\Torus$ (see Figure~\ref{fig:222}(a)).
    For simplicity, we denote each handlebody and the partition of the destabilized handlebody decomposition by the same symbol $H_1$, $H_2$, $H_3$, and $P$, respectively.
    Note that the preimage of the union of spines of $H_1$, $H_2$, and $H_3$ is isotopic to a \emph{$3$hcb} net as shown in Figure~\ref{fig:222}(b).
    The destabilized \patt is also a simple colored tricontinuous pattern.

    For the type-$(2,2,2)$ handlebody decomposition, we can perform a type-$1$ destabilization along a meridian disk $D_4$ of $H_3$ (see Figure~\ref{fig:222}(a) and~\ref{fig:seq_to_honeycomb} (a)).
    The type of resulting decomposition is $(2, 2, 1)$.
    Figures~\ref{fig:seq_to_honeycomb} (a--e) illustrate a destabilization to the type-$(2,2,1)$ handlebody decomposition, which produces a type-$(1,1,1)$ handlebody decomposition.
    The type-$(1, 1, 1)$ handlebody decomposition illustrated in Figure~\ref{fig:3srs_decomp}(f) is the hexagonal honeycomb decomposition (see Figure~\ref{fig:nxx}(d)).
\end{proof}

\begin{figure}[htbp]
    \centering
    \begin{minipage}[b]{0.32\linewidth}
        \centering
        \includegraphics[width=0.90\textwidth]{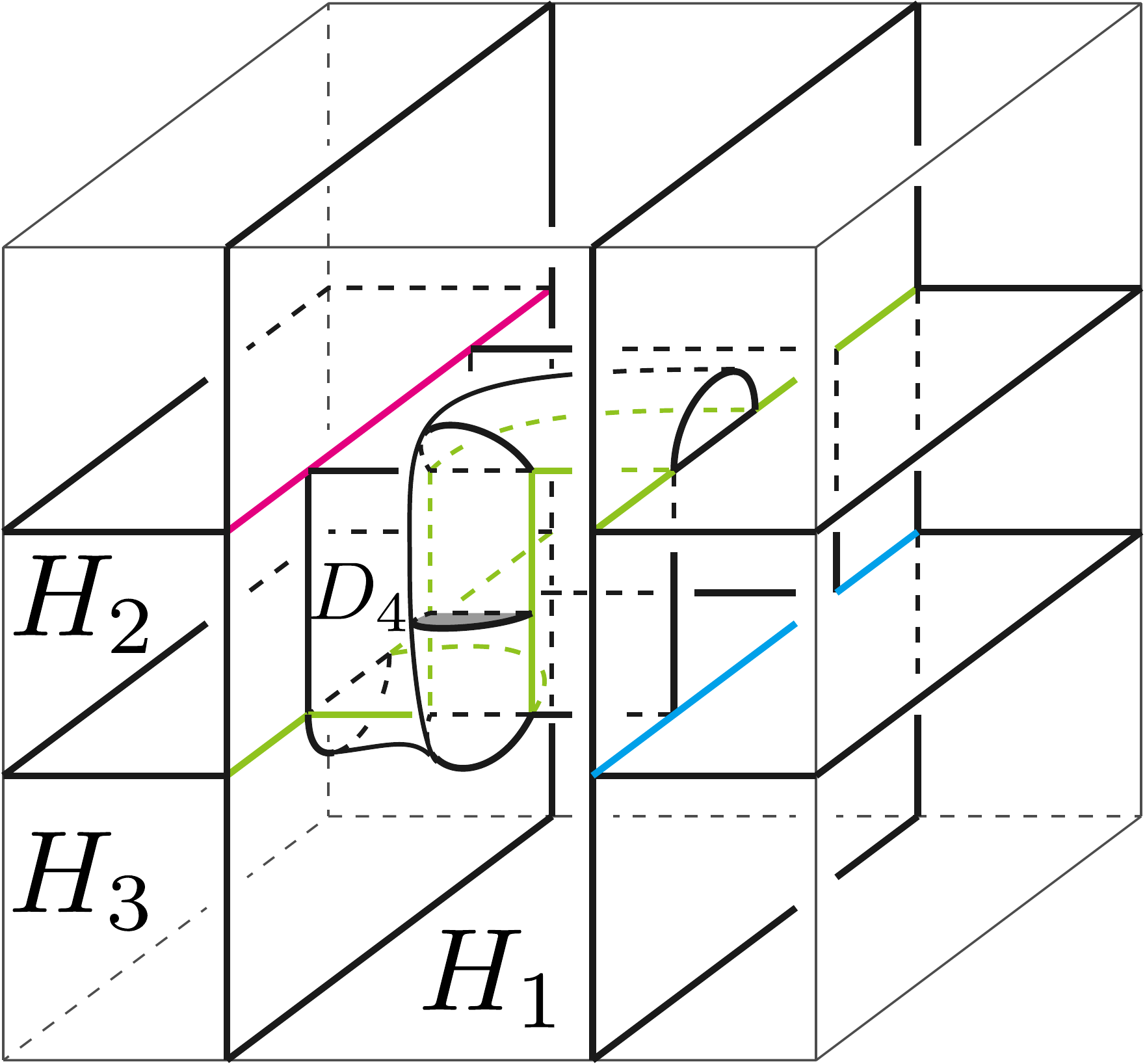}\\
        (a)
    \end{minipage}
    \begin{minipage}[b]{0.32\linewidth}
        \centering
        \includegraphics[width=0.80\textwidth]{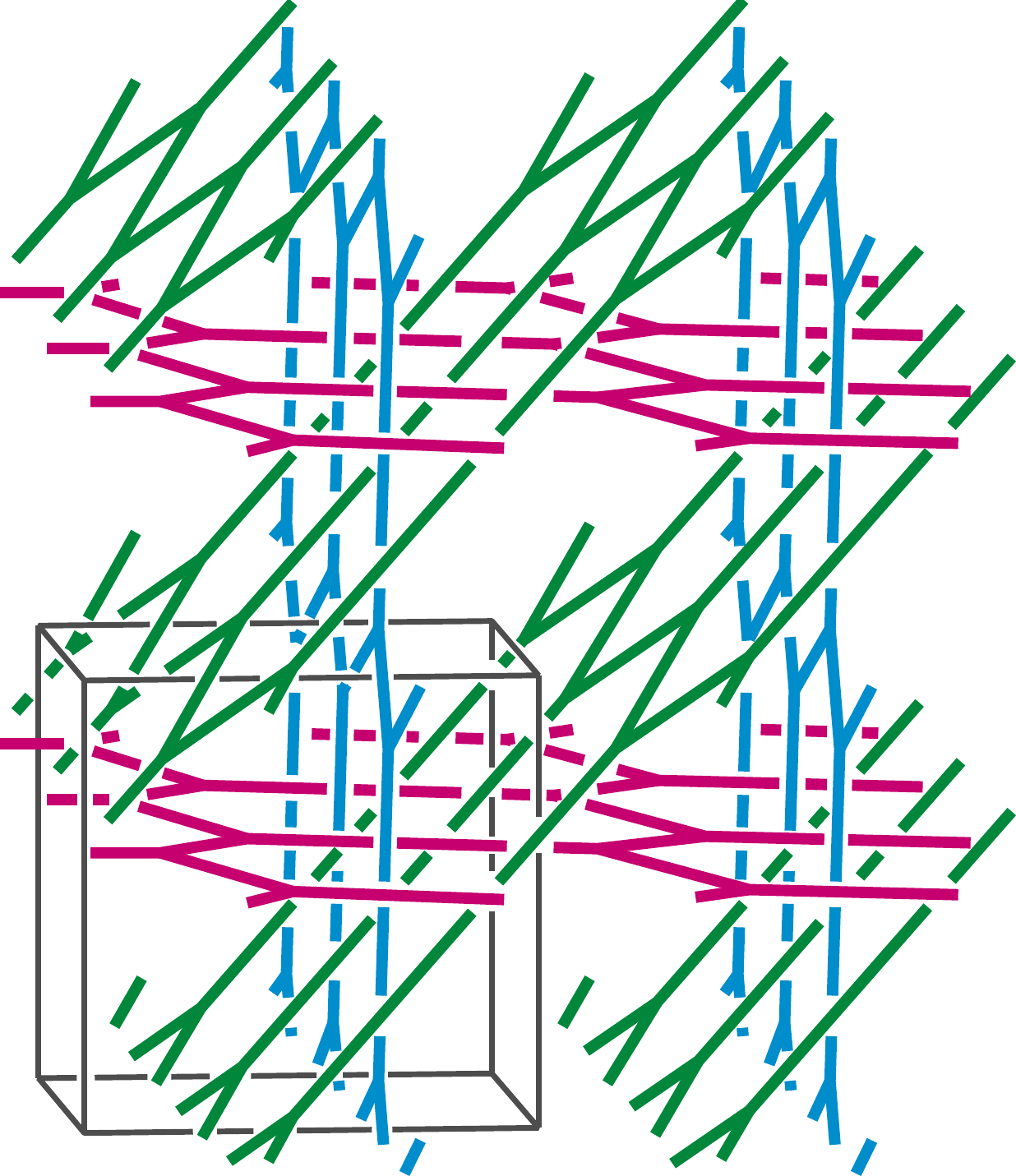}\\
        (b)
    \end{minipage}
    \caption{(a) A type-$(2, 2, 2)$ handlebody decomposition of $\Torus$. (b) A 3hcb net that is the preimage by the universal covering map of the core of the handlebodies $H_1$, $H_2$, and $H_3$.}
    \label{fig:222}
\end{figure}

\begin{figure}[htbp]
    \centering
    \begin{minipage}[b]{0.32\linewidth}
        \centering
        \includegraphics[width=.90\textwidth]{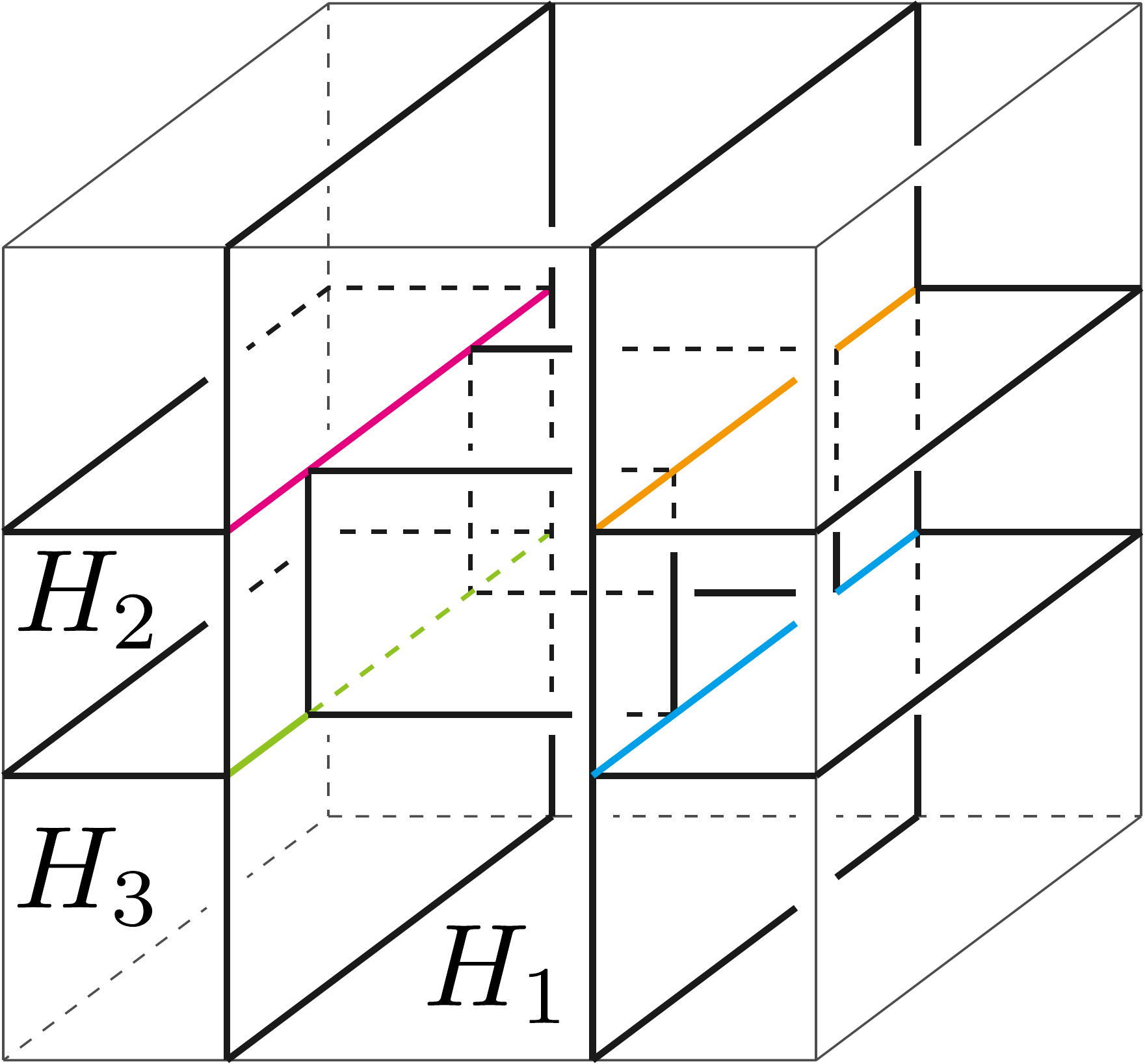}
        \subcaption{Type-$(2,2,1)$}
    \end{minipage}
    \begin{minipage}[b]{0.32\linewidth}
        \centering
        \includegraphics[width=.90\textwidth]{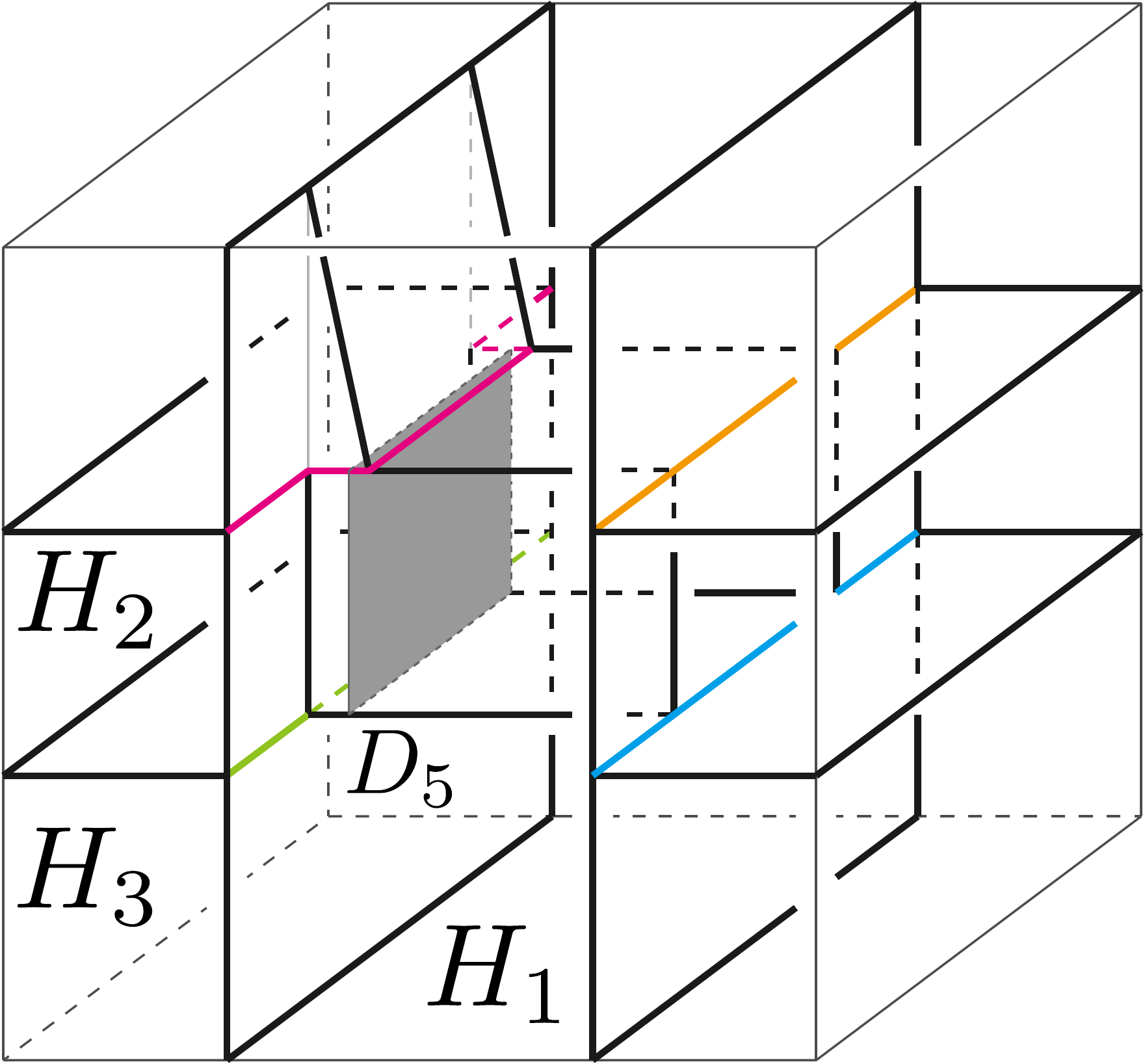}
        \subcaption{A meridian disk $D_5$ of $H_2$}
    \end{minipage}
    \begin{minipage}[b]{0.32\linewidth}
        \centering
        \includegraphics[width=.90\textwidth]{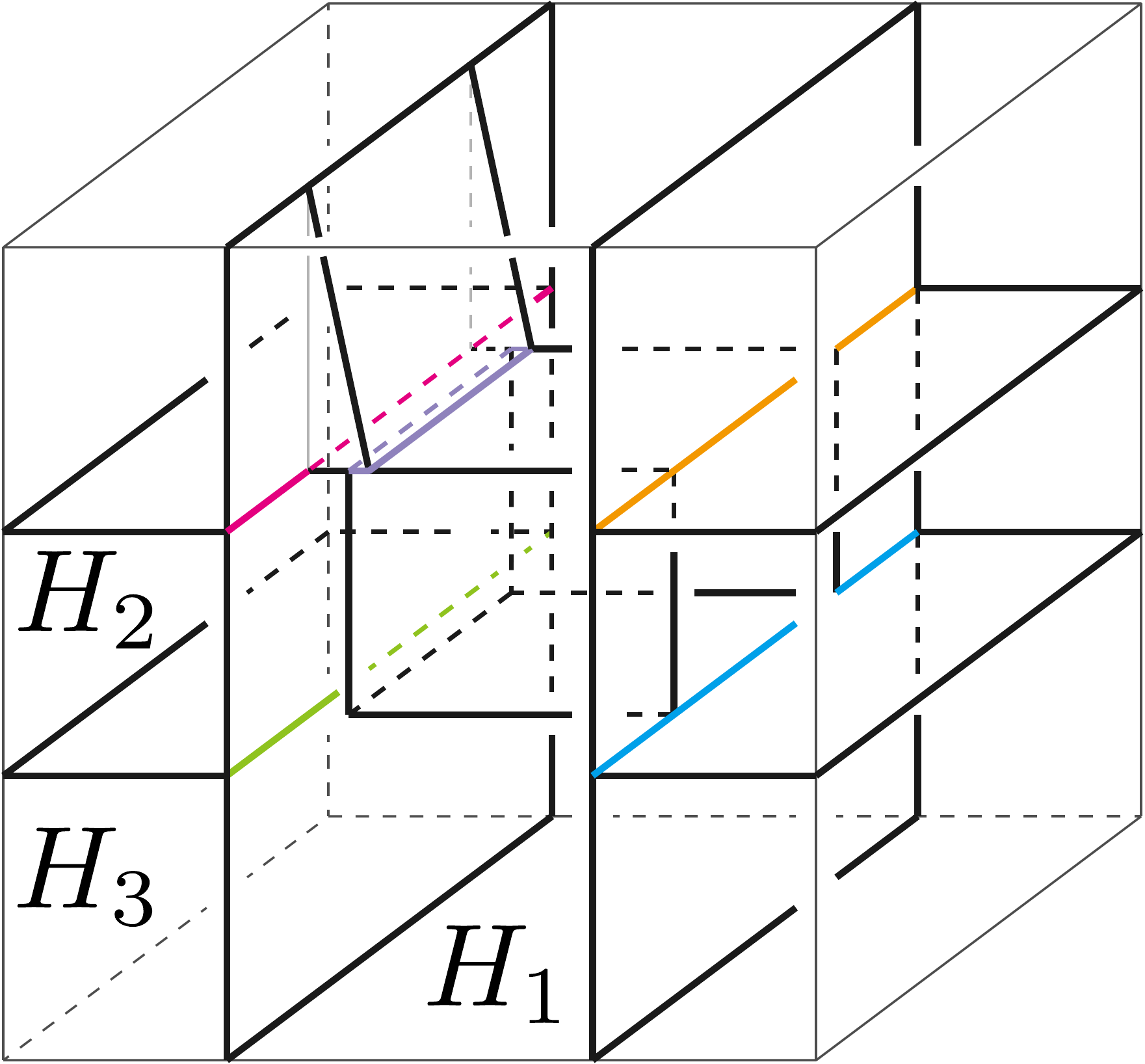}
        \subcaption{Type-$(2,1,1)$}
    \end{minipage}

    \begin{minipage}[t]{0.32\linewidth}
        \centering
        \includegraphics[width=.90\textwidth]{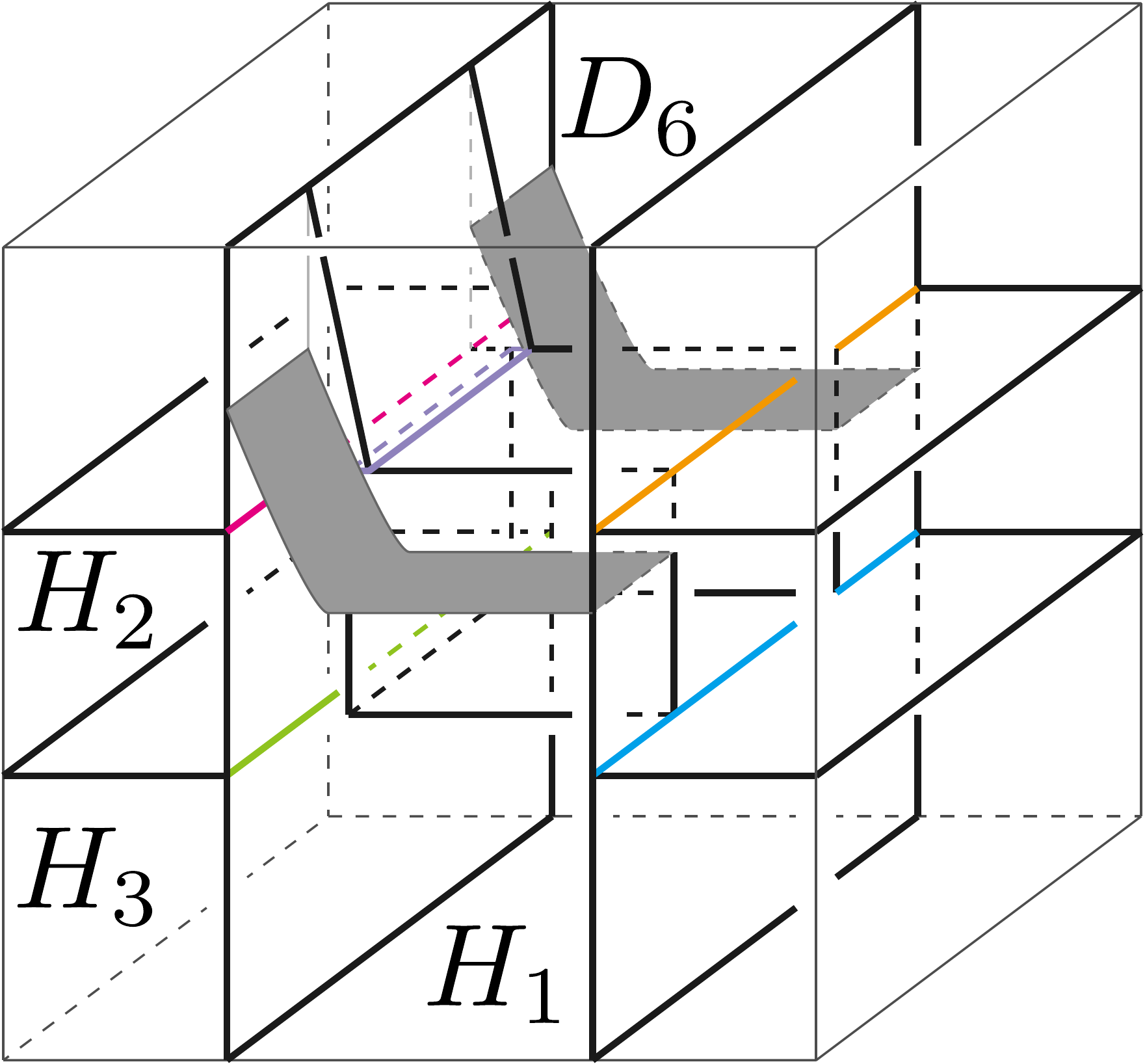}
        \subcaption{A meridian disk $D_6$ of $H_1$}
    \end{minipage}
    \begin{minipage}[t]{0.32\linewidth}
        \centering
        \includegraphics[width=.90\textwidth]{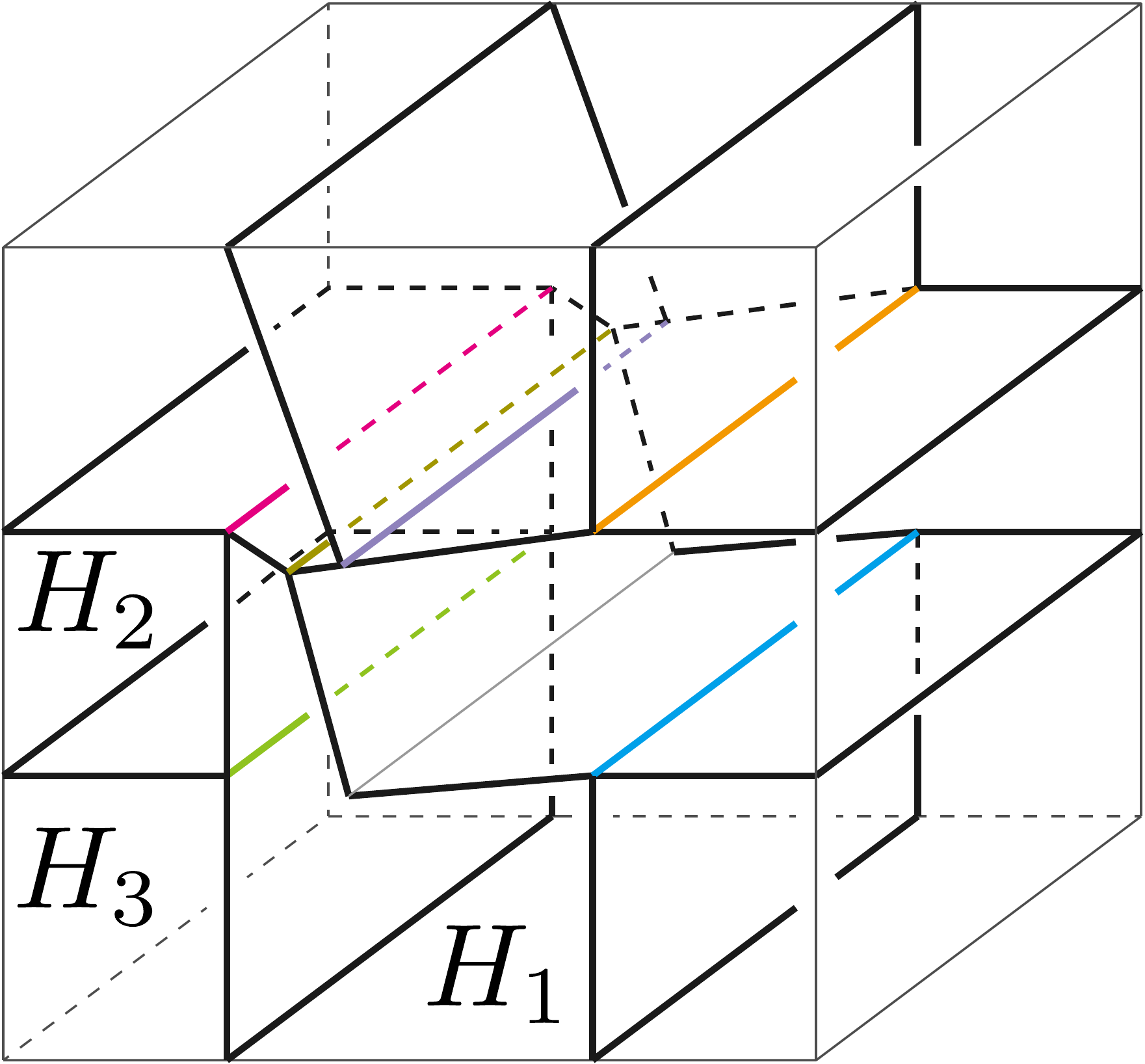}
        \subcaption{Type-$(1,1,1)$}
    \end{minipage}
    \begin{minipage}[t]{0.32\linewidth}
        \centering
        \includegraphics[width=.90\textwidth]{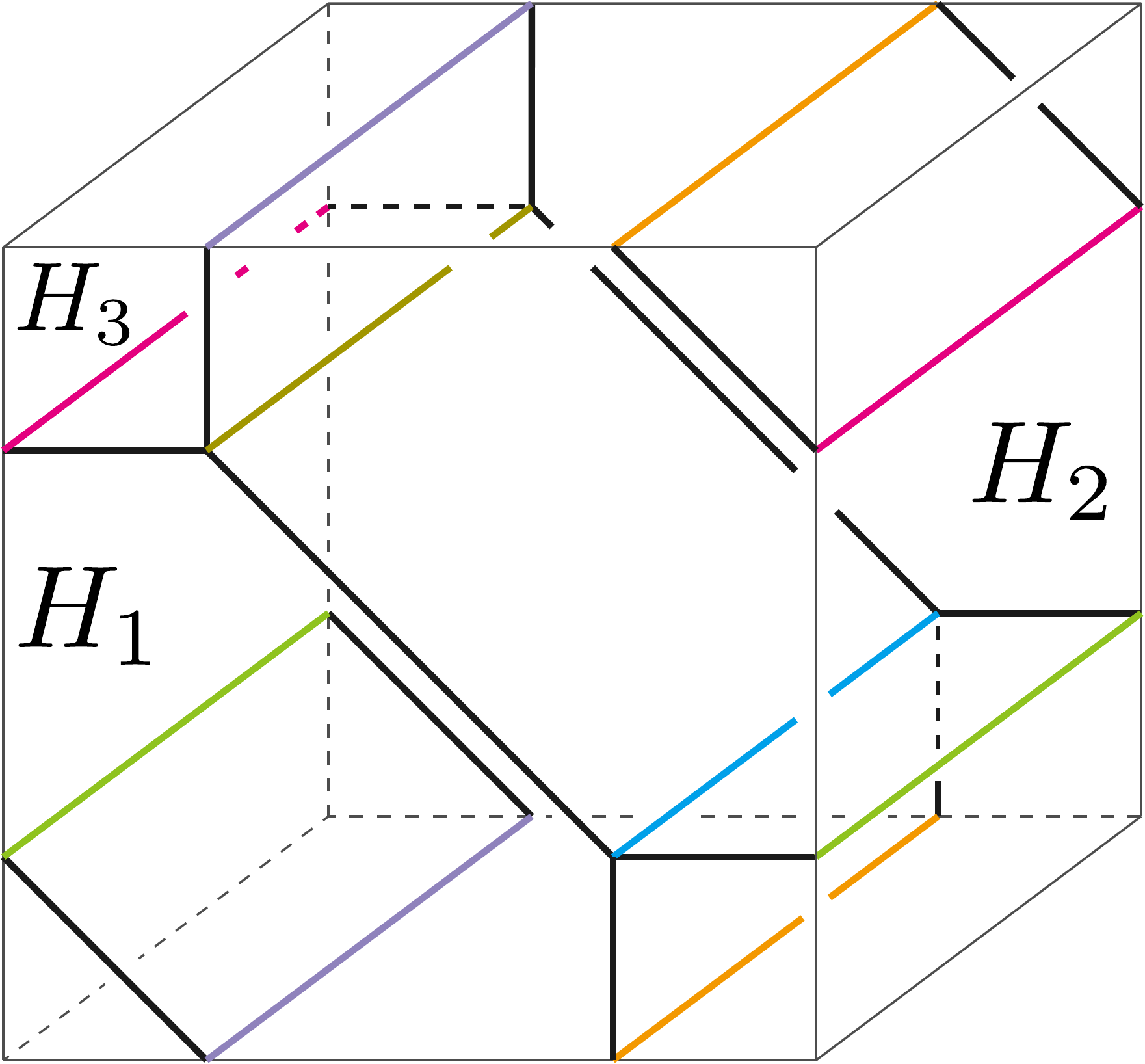}
        \subcaption{The hexagonal honeycomb decomposition}
    \end{minipage}
    \caption{A sequence of type-$1$ destabilizations from the type-$(2,2,1)$ handlebody decomposition to the type-$(1,1,1)$ handlebody decomposition.}
    \label{fig:seq_to_honeycomb}
\end{figure}

\section{Characterization of patterns}
\label{sec:classification_patterns}

In this section, we will prove that bicontinuous patterns are unique.
We will also show that simple, colored, framed \patts of type $(1, 1, 1)$ are unique.
On the other hand, we will provide two different simple colored \patts of type $(1, 1, 1, 1)$.

\subsection{Bicontinuous patterns and Heegaard splittings of $\Torus$}


By definition, an $n$-continuous pattern consists of precisely $n$ labyrinthine domains, and it is proper.
Hence, by assigning a different color to each domain, the pattern admits an $n$-coloring.
In general, a frame of the pattern is not compatible with the coloring.
However, by expanding the fundamental domain, we can obtain a frame compatible with the coloring.
Then, by Corollary~\ref{cor:effective_pattern2decomp}, the pattern with the frame gives a proper type-$(g_1, \ldots, g_n)$ handlebody decomposition of $\Torus$.
Hence, the pattern is a framed \patt of type $(g_1, \ldots, g_n)$.
In particular, we note the following for each simple bicontinuous pattern and such a frame.

\begin{rem}\label{rem:biconti}
    Any simple bicontinuous pattern and its frame compatible with a coloring induce a Heegaard splitting of $\Torus$.
\end{rem}

By~\cite{Boileau,FH1989}, Heegaard splittings of $\Torus$ are determined by their Heegaard genera.
Hence, we can prove the uniqueness of bicontinuous patterns.

%
%

\begin{thm}
\label{thm:bicontinuous}
    Any two simple bicontinuous patterns are equivalent.
\end{thm}
\begin{proof}
    Let $(\wtP, \pi)$ and $(\wtP\upr, \pi\upr)$ be bicontinuous patterns of types $(g, g)$ and $(g\upr, g\upr)$, respectively.
    For the frame $\pi$, there exists a basis $\langle \mathbf{a}_1, \mathbf{a}_2, \mathbf{a}_3 \rangle$ of $\R^3$ such that the translations $t_i$ defined by the vectors $\mathbf{a}_i$ generate the covering transformation group.
    We denote by $T$ a group generated by translations $t_1^{g\upr - 1}$, $t_2$, and $t_3$.
    Hence we have a covering map $\rho: \R^3 \to \R^3/T \cong \Torus$.
    Since the Euler characteristic of $\wtP/\rho$ is $(g\upr - 1)$-times that of $\wtP/\pi$, the surface $\wtP/\rho$ gives a Heegaard splitting of $\Torus$ of genus $(g - 1)(g\upr - 1) + 1$.
    Similarly, we can take a covering map $\rho\upr: \R^3 \to \Torus$ so that $\wtP\upr/ \rho\upr$ also gives a Heegaard splitting of genus $(g - 1)(g\upr - 1) + 1$.
    Therefore, by using~\cite[Th{\'e}or{\`e}me]{Boileau} and Proposition~\ref{prop:frame}, the two simple bicontinuous patterns are equivalent.
\end{proof}


The Gyroid, the Schwartz D surface, and the Schwartz P surface are famous triply periodic minimal surfaces that decompose $\R^3$ into precisely two open components (see~\cite{Schoen}), i.e., the surfaces are simple bicontinuous patterns.
In~\cite[APPENDIX]{Squires2005}, Squires, Templer, Seddon, and Woenkhaus gave an isotopy from the Gyroid to the Schwartz D surface and the Schwartz D surface to the Schwartz P surface by an explicit formula.
Note that Theorem~\ref{thm:bicontinuous} is a generalization of the result but does not give a formula for transformation between patterns.

\subsection{The uniqueness of framed patterns of type $(1, 1, 1)$}

We consider the hexagonal honeycomb pattern introduced in Example~\ref{ex:honeycomb_pattern}.
Recall that its pattern admits a coloring and a frame compatible with it, as in Figure~\ref{fig:colored_patterns}(a).
The pattern induces the hexagonal honeycomb decomposition of $\Torus$.
Hence, the hexagonal honeycomb pattern with the frame is of type $(1, 1, 1)$.
By Propositions~\ref{prop:honeycomb_decomp} and~\ref{prop:frame}, the hexagonal honeycomb pattern is a canonical model of simple colored \patts of type $(1, 1, 1)$.
Therefore, we have the following.

\begin{thm}\label{thm:honeycomb}
    Any simple, colored, framed \patt of type $(1, 1, 1)$ is equivalent to the hexagonal honeycomb pattern.
\end{thm}

Note that a simple $3$-colored \patt whose labyrinthine nets consist of lines is not necessarily equivalent to the hexagonal honeycomb pattern in general (see Example~\ref{ex:non_honeycomb}).
Also, there are distinct simple colored \patts of type $(1, 1, 1, 1)$ (see Example~\ref{ex:1111}).

\begin{ex}\label{ex:non_honeycomb}
    We consider a tessellation of the plane $\R^2$ by three kinds of tiles: square, hexagon, and 8-sided polygon.
    Figure~\ref{fig:11_11_1_1} shows a \patt induced by the tessellation.
    The left side illustrates a framed \patt of type $(1,1,1)$ that is not colored since 8-sided components are assigned to the same color, and they are adjacent.
    On the other hand, the pattern admits a $4$-coloring (Figure~\ref{fig:11_11_1_1}(b)).
    However, it is no longer type $(1, 1, 1)$.
\end{ex}

\begin{figure}[htbp]
    \centering
    \begin{minipage}[b]{0.49\textwidth}
        \centering
        \includegraphics[width=.90\textwidth]{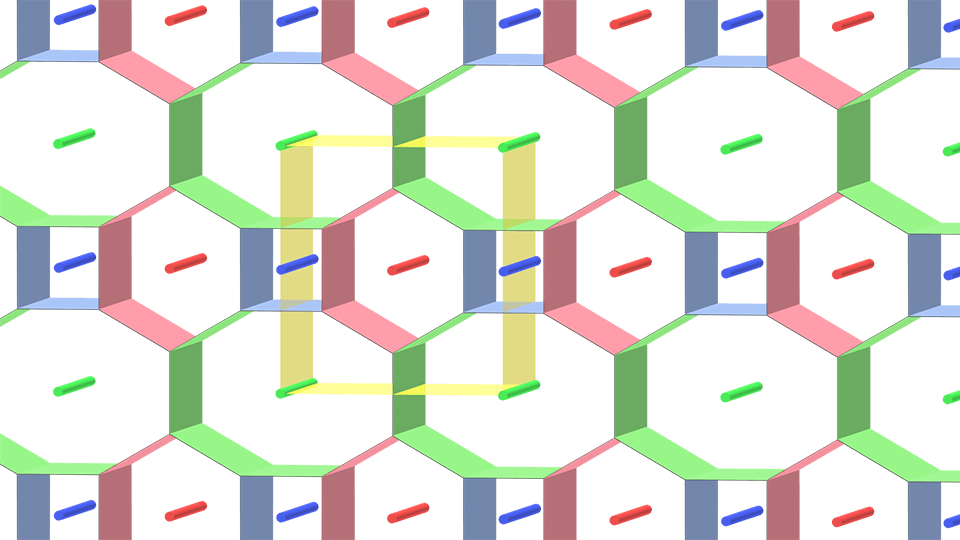}\\
        (a)
    \end{minipage}
    \begin{minipage}[b]{0.49\textwidth}
        \centering
        \includegraphics[width=.90\textwidth]{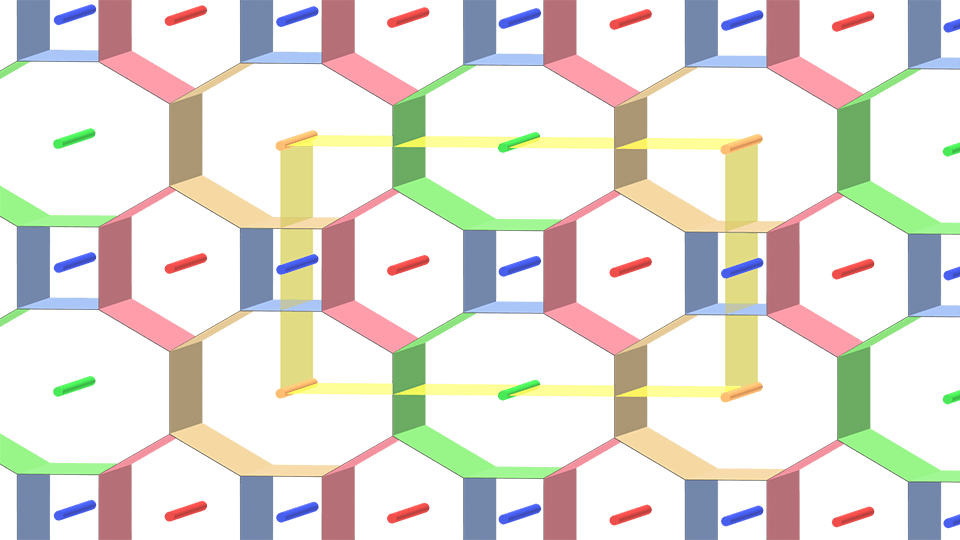}\\
        (b)
    \end{minipage}
    \caption{(a) A simple non-effectively colored \patt of type $(1, 1, 1)$. (b) A simple colored \patt  of type $([1, 1], [1, 1], 1, 1)$. The handlebody decomposition induced by the pattern contains two blue solid tori and two red solid tori.}
    \label{fig:11_11_1_1}
\end{figure}

\begin{ex}\label{ex:1111}
    Figure~\ref{fig:1111} illustrates two simple colored \patts, $(\wtP_a, \rho_a)$ and $(\wtP_b, \rho_b)$, of type $(1, 1, 1, 1)$ with a cubical fundamental domain, where $\rho_a$ and $\rho_b$ denote their frames compatible with the colorings, respectively.
    We can see the two patterns are not equivalent as follows.
    Let $\widetilde{X}_a$ and $\widetilde{X}_b$ be nets associated with $\wtP_a$ and $\wtP_b$.
    We consider the image $(\iota_a)_\ast(\fund(\rho_a(\widetilde{X}_a)))$ and $(\iota_b)_\ast(\fund(\rho_b(\widetilde{X}_b)))$, where $\iota_a$ and $\iota_b$ are the inclusion maps, respectively.
    By Figure~\ref{fig:1111}(a) $(\iota_a)_\ast(\fund(\rho_a(\widetilde{X}_a)))$ is isomorphic to $\mathbb{Z} \oplus \mathbb{Z}$.
    On the other hand, $(\iota_b)_\ast(\fund(\rho_b(\widetilde{X}_b)))$ is isomorphic to $\mathbb{Z} \oplus \mathbb{Z} \oplus \mathbb{Z}$ by Figure~\ref{fig:1111}(b).
    Hence, $\wtP_a$ is not equivalent to $\wtP_b$.

    By Theorem~\ref{thm:honeycomb}, any two simple colored framed \patts of type $(1, 1, 1)$ are equivalent.
    However, simple colored \patts of type $(1, 1, 1, 1)$ are not unique.
\end{ex}

\begin{figure}[htbp]
    \begin{minipage}[b]{0.3\linewidth}
        \centering
        \includegraphics[width=.95\textwidth]{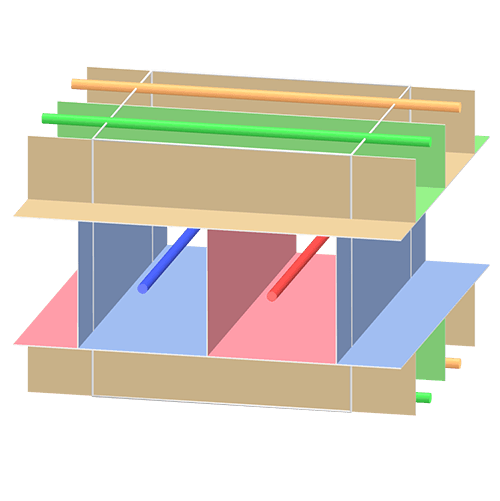}\\
        (a)
    \end{minipage}
    \begin{minipage}[b]{0.3\linewidth}
        \centering
        \includegraphics[width=.95\textwidth]{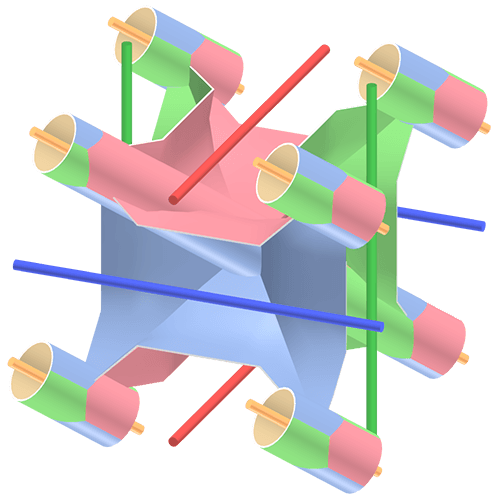}\\
        (b)
    \end{minipage}
    \begin{minipage}[b]{0.3\linewidth}
        \centering
        \includegraphics[width=.95\textwidth]{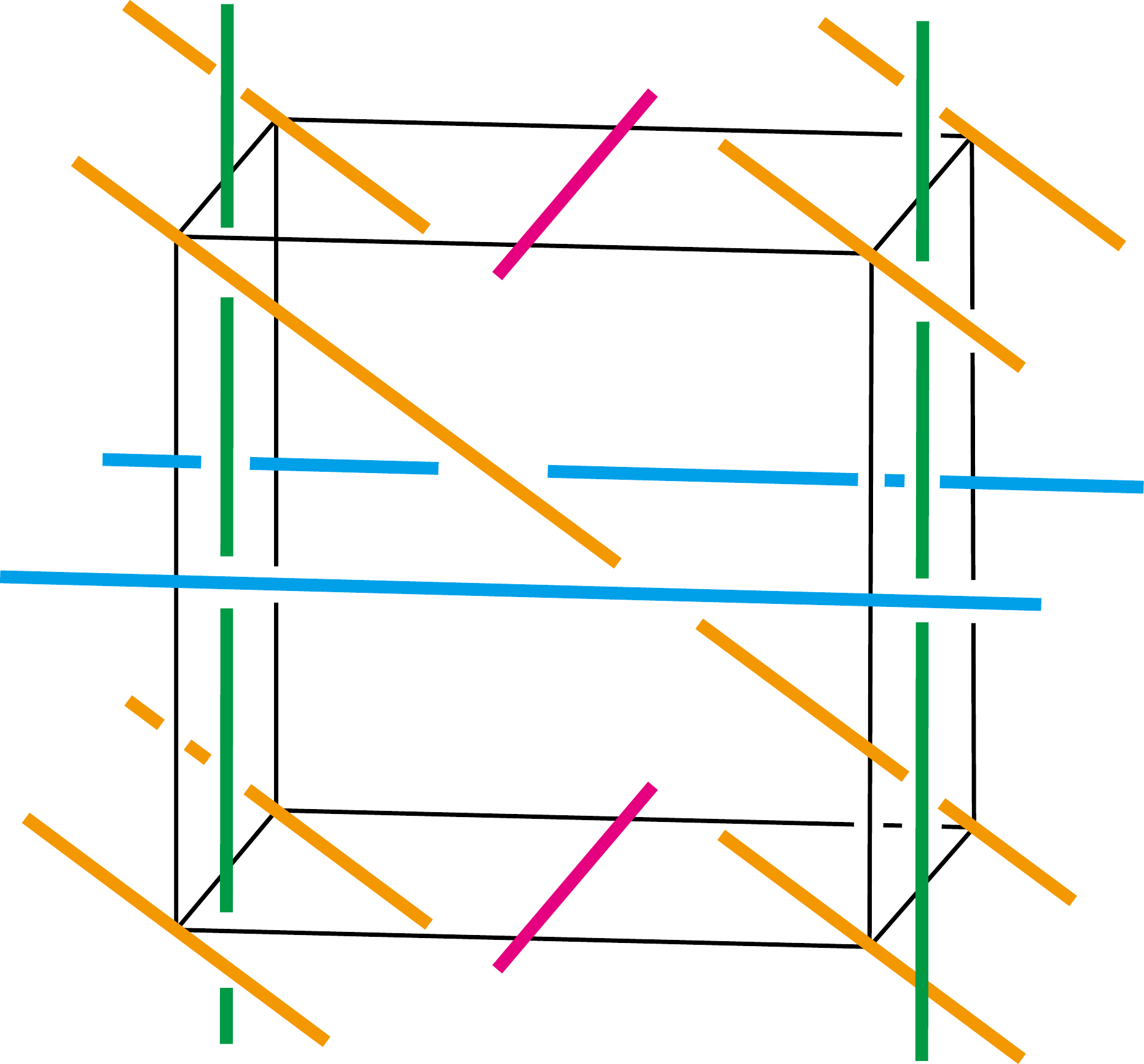}\\
        (c)
    \end{minipage}
    \caption{(a) and (b) Two framed simple colored \patts of type $(1, 1, 1, 1)$. (c) The labyrinthine nets of (b).}
    \label{fig:1111}
\end{figure}

\ifRSPA
\else

\appendix

\section{The independence of stabilizations}
\label{app:independence}

The stabilizations introduced in Definition~\ref{dfn:stabilization} are not independent under certain conditions.

\begin{prop}
For $n \geq 3$, a type-$0$ stabilization is generated by a type-1 stabilization and a type-2 stabilization.
\end{prop}

\begin{proof}
    Let $(H_1, \ldots, H_n; P)$ be a handlebody decomposition.
    Take a properly embedded arc in $H_i$ whose endpoints are in the interior of $F_{ij}$.
    Let $\beta$ be an arc in $\partial H_i$ such that $\partial \alpha = \partial \beta$, and $\alpha \cup \beta$ bounds a disk in $H_i$.
    We will realize a type-$0$ stabilization along $\beta$ by a type-$1$ stabilization and a type-$2$ stabilization.
    Since, by the definition of a handlebody decomposition, the partition $P$ is connected, there exists a handlebody $H_k$ that intersects both $H_i$ and $H_j$ ($k \neq i, j$).
    Thus, we can take an arc $\delta$ in $F_{ij}$ joining a point of $\partial F_{ik}$ and an endpoint of $\beta$ so that $\delta$ does not intersect $\beta$ except for the endpoint of $\beta$.
    See Figure~\ref{fig:type12to0}.
    Then, the boundary of a regular neighborhood $N(\delta)$ of $\delta$ in $H_i$ intersects $\alpha$ transversely in exactly one point.
    We denote the sub-arc $\alpha \setminus \inte(N(\delta))$ by $\alpha^\prime$.
    Since $N(\delta) \cap \inte(F_{ik})$ is not empty, there exists an arc $\gamma$ in $\partial N(\delta) \setminus \partial H_i$ joining a point $p$ in $\inte(F_{ik})$ and an endpoint of $\alpha^\prime$.
    Let $\delta^{\prime\prime}$ be an arc in $F_{ik}$ joining $p$ and $\partial \delta \cap \partial F_{ik}$, and $\delta^\prime$ the union of $\delta$ and $\delta^{\prime\prime}$.
    Then, the union of arcs $\alpha^\prime$, $\gamma$, $\delta^\prime$, and $\beta$ bounds a disk in $H_i$.
    Hence we can perform a type-2 stabilization along $\alpha\upr \cup \gamma$.
    Then a new disk intersection occurs between $H_k$ and $H_j$, and the boundary of the disk contains an endpoint of $\delta^{\prime\prime}$.
    Hence, we can perform a type-$1$ stabilization along $\delta^{\prime\prime}$.
    It is easily seen that the handlebody decomposition thus obtained is the one obtained from $(H_1, \ldots, H_n; P)$ by a type-$0$ stabilization along $\alpha$.
\end{proof}

\begin{figure}[htbp]
    \centering
    \includegraphics[scale=0.34]{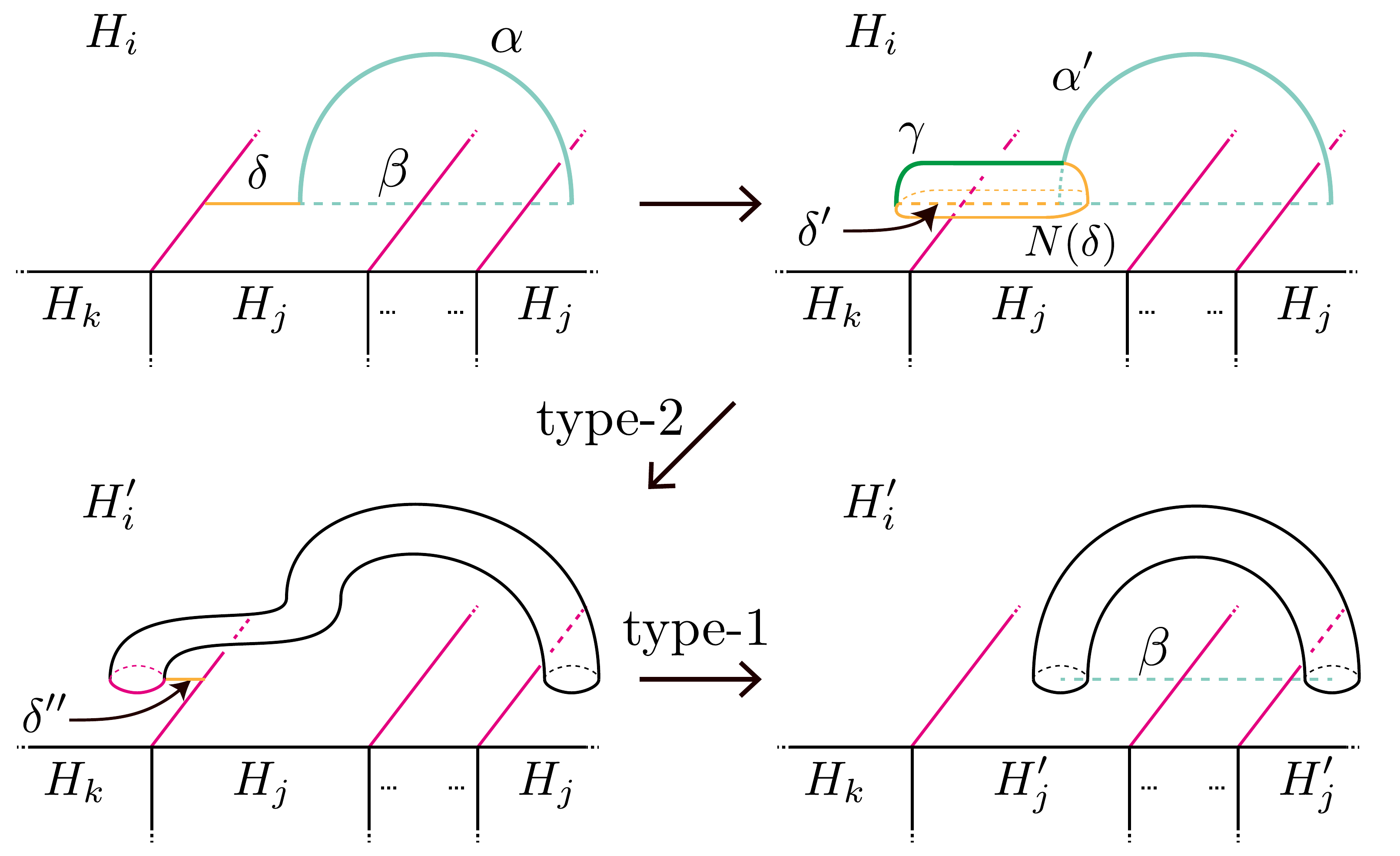}
    \caption{A type-$0$ stabilization can be obtained by a combination of a type-$2$ stabilization and a type-1 stabilization.}
    \label{fig:type12to0}
\end{figure}

Note that each of a type-1 and a type-2 stabilizations increases the genus of exactly one handlebody.
A type-1 stabilization can be realized by a type-2 stabilization if and only if the endpoints of an arc $\alpha$ as in Definition~\ref*{dfn:stabilization}(1) are contained in the same edge of the singular graph of the partition.
Conversely, a type-2 stabilization can be realized by a type-1 stabilization if and only if an arc $\beta$ as in Definition~\ref*{dfn:stabilization}(2) is isotopic to an arc intersecting the singular graph only once.


\section{Destabilizations of \patts}
\label{app:destab}

We give sufficient conditions for performing destabilization on \patts.

Let $(\wtP, \pi)$ be a simple, $n$-colored, framed \patt of type $(\mathfrak{g}_1, \ldots, \mathfrak{g}_n)$, where $\mathfrak{g}_i$ is a sequence of positive integers $[g\ssn{i}_{1}, \ldots, g\ssn{i}_{m_i}]$ for $1 \leq i \leq n$.
Put $P = \pi(\wtP)$.
By Corollary~\ref{cor:effective_pattern2decomp}, $P$ gives a simple proper handlebody decomposition $\mathcal{H} = (H\ssn{1}_1,$ $\ldots,$ $H\ssn{1}_{m_1},$ $H\ssn{2}_{1},$ $\ldots,$ $H\ssn{2}_{m_2}, \ldots, H\ssn{n}_1, \ldots, H\ssn{n}_{m_n}; P)$ of $\Torus$ such that $H\ssn{i}_j$ is a genus-$g\ssn{i}_j$ handlebody colored by $i$.
Let $U$, $V$, and $W$ be labyrinthine domains such that, for each pair of them, the two domains are different and share a sector.
We further assume that $\pi(U) = H\ssn{i_U}_{j_U}$, $\pi(V) = H\ssn{i_V}_{j_V}$, and $\pi(W) = H\ssn{i_W}_{j_W}$.
Here, $H\ssn{i_U}_{j_U}$, $H\ssn{i_V}_{j_V}$, and $H\ssn{i_W}_{j_W}$ are distinct handlebodies.
We denote by $\mathcal{T}(\wt{X})$ the disjoint union of the images of a lifted set $\wt{X}$ under all the covering transformations.
Depending on the type of destabilization, we take a disk $\wt{D}$ as follows.
\begin{enumerate}[label=(type-\arabic*),align=parleft,leftmargin=*,itemindent=37.21pt,itemsep=1mm,start=0]
    \item The disk $\wt{D}$ is a properly embedded lifted disk in $U$.
        We assume that there exists a properly embedded lifted disk $\wt{E}$ in $V$ so that the following conditions hold:
        \begin{enumerate}[label=(\roman*), leftmargin=*]
            \item The boundary of $\wt{D}$ is in a sector shared by $U$ and $V$.
            \item The intersection of $\wt{D}$ and $\wt{E}$ is exactly one point.
            \item The set $U \setminus \mathcal{T}(\wt{D})$ does not contain a bounded component.
        \end{enumerate}
    \item The disk $\wt{D}$ is a properly embedded lifted disk in $V$ satisfying the following conditions:
        \begin{enumerate}[label=(\roman*), leftmargin=*]
            \item The boundary of $\wt{D}$ intersects the singular graph of $\wtP$ transversely in exactly two points.
            \item The disk $\pi(\wt{D}) \subset H\ssn{i_V}_{j_V}$ is non-separating.
            \item The set $V \setminus \mathcal{T}(\wt{D})$ does not contain a bounded component.
        \end{enumerate}
    \item The disk $\wt{D}$ is a disk sector of $\wtP$ shared by $U$ and $W$.
        We assume that there exists a properly embedded lifted disk $\wt{E}$ in $V$ such that the boundary of $\wt{D}$ intersects $\wt{E}$ transversely once.
\end{enumerate}
Then, we can obtain a new handlebody decomposition $\mathcal{H}\upr$ performed by a suitable destabilization along $\pi(\wt{D})$ on $\mathcal{H}$.
The preimage of the partition of $\mathcal{H}\upr$ is a simple, colored, framed \patt.
We call this operation a \emph{type-$k$ destabilization (along $\wt{D}$ with respect to $\pi$)} for each $k$.
(Note that \patts do not contain a bounded component.
In each case of type-$0$ and type-$1$, condition~{(iii)} guarantees that the complement of the partition of $\mathcal{H}\upr$ consists of unbounded components.)

\fi

\ifRSPA
\ack{%
\else
\section*{Acknowledgment}
\fi
The authors would like to thank Professor Mikami Hirasawa for his practical advice and beautiful lecture on nets and polygons.
In particular, we learned the coordinates of the vertices of the srs net from him.
The last author would like to thank Professors Takeshi Aoyagi and Katsumi Hagita for their precious comments on the polymer scientific insight.
\ifRSPA
}
\funding{%
\fi
This research is supported by MEXT Grants-in-Aid for Scientific Research on Innovative Areas (JP17H06460 and JP17H06463) and JSPS KAKENHI Grant Number JP21H00978.
\ifRSPA
}
\fi

\bibliography{refs} 
\bibliographystyle{abbrv}


\end{document}